\documentclass{amsart}
\usepackage[margin=1.2in]{geometry}
\usepackage{amssymb,amsmath,amsthm,,graphicx,amscd,thmtools,tikz,multicol,float}
\usepackage[shortlabels]{enumitem}
\usepackage{vwcol}
\theoremstyle{plain}
\newtheorem*{thm*}{Theorem}
\newtheorem*{lm*}{Lemma}
\newtheorem*{ex*}{Example}
\newtheorem{thm}{Theorem}[section]
\newtheorem{lm}[thm]{Lemma}

\newtheorem{prop}[thm]{Proposition}
\newtheorem{conj}[thm]{Conjecture}

\newtheorem{que}[thm]{Question}

\DeclareSymbolFont{extraup}{U}{zavm}{m}{n}
\DeclareMathSymbol{\varheart}{\mathalpha}{extraup}{86}
\DeclareMathSymbol{\vardiamond}{\mathalpha}{extraup}{87}

\theoremstyle{definition}

\declaretheorem[sibling=thm,name=Definition,qed={$\vardiamond$}]{de}
\declaretheorem[sibling=thm,name=Example,qed={$\spadesuit$}]{ex}
\declaretheorem[sibling=thm,name=Remark,qed={$\clubsuit$}]{re}

\DeclareMathOperator{\CC}{\mathbb{C}}
\DeclareMathOperator{\PP}{\mathbb{P}}

\DeclareMathOperator{\As}{\mathcal{A}}
\DeclareMathOperator{\Bs}{\mathcal{B}}

\DeclareMathOperator{\Es}{\mathcal{E}}

\DeclareMathOperator{\Js}{\mathcal{J}}

\DeclareMathOperator{\Ls}{\mathcal{L}}

\DeclareMathOperator{\Ps}{\mathcal{P}}

\let\S\relax\DeclareMathOperator{\S}{\mathbb{S}}

\DeclareMathOperator{\Gr}{Gr}
\DeclareMathOperator{\GL}{GL}
\DeclareMathOperator{\Jo}{Jo}
\let\O\relax\DeclareMathOperator{\O}{O}
\DeclareMathOperator{\Diag}{Diag}
\DeclareMathOperator{\Jord}{Jord}

\DeclareMathOperator{\rk}{rk}

\DeclareMathOperator{\codim}{codim}
\DeclareMathOperator{\spann}{span}
\DeclareMathOperator{\lcm}{lcm}
\DeclareMathOperator{\adj}{adj}

\DeclareMathOperator{\rddots}{\mbox{\reflectbox{$\ddots$}}}

\title{The Geometries of Jordan nets and Jordan webs}

\author{Arthur Bik}
\address{MPI for Mathematics in the Sciences, Leipzig, Germany}
\email{arthur.bik@mis.mpg.de}

\author{Henrik Eisenmann}
\address{MPI for Mathematics in the Sciences, Leipzig, Germany}
\email{henrik.eisenmann@mis.mpg.de}

\begin{document}

\begin{abstract}
A Jordan net (resp. web) is an embedding of a unital Jordan algebra of dimension $3$ (resp. $4$) into the space $\mathbb{S}^n$ of symmetric $n\times n$ matrices. We study the geometries of Jordan nets and webs: we classify the congruence-orbits of Jordan nets (resp. webs) in $\mathbb{S}^n$ for $n\leq 7$ (resp. $n\leq 5$), we find degenerations between these orbits and list obstructions to the existence of such degenerations. For Jordan nets in $\mathbb{S}^n$ for $n\leq5$, these obstructions show that our list of degenerations is complete. For $n=6$, the existence of one degeneration is still undetermined.

To explore further, we used an algorithm that indicates numerically whether a degeneration between two orbits exists. We verified this algorithm using all known degenerations and obstructions, and then used it to compute the degenerations between Jordan nets in $\mathbb{S}^7$ and Jordan webs in $\mathbb{S}^n$ for $n=4,5$.
\end{abstract}
\maketitle

\section{Introduction}\label{sec:intro}

Let $\S^n$ be the space of symmetric $n\times n$ matrices over the complex numbers $\CC$. We say that a subspace $\Ls\subseteq\S^n$ is {\em regular} when it contains an invertible matrix. For a regular subspace $\Ls\subseteq\S^n$, we define its {\em reciprocal variety} $\Ls^{-1}$ to be the Zariski-closure of the set $\{X^{-1}\mid X\in\Ls,\det(X)\neq0\}$. The goal of this paper is to follow-up on \cite{BES:JordanSpaces} and study the geometry of the regular subspaces $\Ls\subseteq\S^n$ such that $\Ls^{-1}$ is again a linear subspace of $\S^n$. Denote the Grassmannian of $m$-dimensional subspaces of $\S^n$ by $\Gr(m,\S^n)$. The following theorem, which is a more general formulation of a result of Jensen \cite[Lemma 1]{J:LinCov}, was the starting point of our investigation.

\begin{thm}[{\cite[Theorem 1.1]{BES:JordanSpaces}}]\label{thm:equiv_conditions}
Let $\Ls\in\Gr(m,\S^n)$ be a regular subspace and let $U\in\Ls$ be an invertible matrix. Then the following are equivalent: 
\begin{itemize}
\item[$\mathrm{(a)}$] The reciprocal variety $\Ls^{-1}$ is also a linear space in $\S^n$.
\item[$\mathrm{(b)}$] The subspace $\Ls$ is a subalgebra of the Jordan algebra $(\S^n,\bullet_U)$.
\item[$\mathrm{(c)}$] We have $\Ls^{-1}=U^{-1}\Ls U^{-1}$.
\end{itemize}
\end{thm}

Here the operation $-\bullet_U-\colon\S^n\times\S^n\to\S^n$ is defined as
\[
X\bullet_U Y:= \frac{XU^{-1}Y+YU^{-1}X}{2}\in\S^n
\]
for all $X,Y\in\S^n$ and satisfies the {\em Jordan axiom}:
\[
(X\bullet_U X)\bullet_U(X\bullet_U Y)=X\bullet_U((X\bullet_UX)\bullet_U Y)\mbox{ for all $X,Y\in\S^n$.}
\]
We call a subspace $\Ls\in\Gr(m,\S^n)$ a {\em Jordan space} when these equivalent conditions are satisfied. Jordan spaces of dimensions $2,3$ and $4$ are also called {\em Jordan pencils}, {\em nets} and {\em webs}, respectively.\bigskip

Denote the adjoint of a matrix $U\in\S^n$ by $\adj(U)$. Define the {\em Jordan locus} $\Jo(m,\S^n)$ to be the subvariety of $\Gr(m,\S^n)$ consisting of all subspaces $\Ls=\CC\{X_1,\ldots,X_m\}$ such that
\[
X_1,X_2,\ldots,X_m, \,\,X_i\adj(U)X_j+X_j\adj(U)X_i
\] 
are linearly dependent for all matrices $U\in\Ls$ and all indices $1\leq i\leq j\leq m$. Theorem~\ref{thm:equiv_conditions} shows that a regular subspace $\Ls\in\Gr(m,\S^n)$ is a Jordan space if and only if $\Ls\in\Jo(m,\S^n)$. The group $\GL_n(\CC)$ acts on $\Gr(m,\S^n)$ by congruence:
\[
g\cdot \Ls:=g\Ls g^\top\in\Gr(m,\S^n)\mbox{ for all $\Ls\in\Gr(m,\S^n)$ and $g\in\GL_n(\CC)$.}
\] 
The subvariety $\Jo(m,\S^n)$ is closed under the congruence action. It is the goal of this paper to classify the orbits $\GL_n(\CC)\cdot\Ls$ of regular subspaces $\Ls\in\Jo(m,S^n)$ and to understand the degenerations between them. 

\subsection{Results}

\subsubsection*{Classification of orbits}
For $5$ out of $6$ isomorphism types of unital Jordan algebras of dimension $3$, we determine the congruence-orbits of Jordan nets in $\S^n$ for all $n$. For the last isomorphism type, we do this for $n\leq 7$. We also determine the congruence-orbits of embeddings of Jordan webs in $\S^n$ for~$n\leq 5$. These results are summarized in Section~\ref{sec:summary_embeddings}.

\subsubsection*{Degenerations and obstructions}
In the Appendix, we give lists of degenerations between orbits of Jordan spaces and in Section~\ref{sec:degenerations_obstructions} we list a series of obstructions to the existence of such degenerations. These obstructions suffice to show that our lists of degenerations Jordan nets of the first $5$ isomorphism types are complete.
We also determine the degeneration diagrams of Jordan nets in $\S^5$ and $\S^6$ (up to one degeneration) in Section~\ref{sec:Jordan_nets_n<=6}.

\subsubsection*{Numerical results}
We give an algorithm that indicates numerically whether a degeneration between two orbits of Jordan spaces exists. We verified this algorithm using all known degenerations and obstructions between Jordan nets in $\S^n$ for $n\leq 6$, and use the algorithm to compute the degenerations between Jordan nets in $\S^7$ and Jordan webs in $\S^n$ for $n=4,5$.

\subsection{Applications}
Subspaces $\Ls\subseteq\S^n$ such that $\Ls^{-1}$ is also a linear spaces arise naturally in statistics \cite{J:LinCov,STZ}: in many statistical applications, one studies normally distributed random variables $X_1,\ldots,X_n$ with linear conditions on their covariance matrix $\Sigma$ or their concentration matrix $\Sigma^{-1}$. The condition that the matrix $\Sigma$ lies in a subspace $\Ls\subseteq\S^n$ whose reciprocal is also a linear space is a mixture of these two conditions. Seely \cite{S:quadratic_completeness,S:completeness_family}  proved that such models are the only models of multivariate normal distributions with zero mean that have a complete sufficient statistic.\bigskip

In \cite{PP:DimensionReduction}, Parrilo and Permenter showed that minimal subspaces which contain primal and dual solutions of a semidefinite optimization problem are Jordan algebras. Indeed, Jordan algebras are in some sense the more general space for optimization problems. It is well known that linear optimization problems are equivalent to semidefinite optimization problems over diagonal matrices, however both are instances of symmetric cones~\cite{FK:SymmetricCones}. Symmetric cones are given by the squares of an Euclidean Jordan algebra, i.e. a Jordan algebra where sum of squares are always nonzero.

\subsection*{Structure of the paper.}
In Section~\ref{sec:abstract_jordan}, we recall the basic properties of abstract unital Jordan algebras and list the degenerations between them in dimensions $2$, $3$ and $4$. In Section~\ref{sec:summary_embeddings}, we list the orbits of embeddings of Jordan nets into $\S^n$ for $n\leq 7$ and Jordan webs into $\S^n$ for $n\leq 5$. In Section~\ref{sec:Jordan_nets_n<=6}, we list all degenerations of Jordan nets in $\S^n$ for $n=5,6$ assuming one degeneration does not exist. In Section~\ref{sec:numerical_results}, we find all degenerations between Jordan nets in $\S^7$ and Jordan webs in $\S^n$ for $n=4,5$ numerically. In Section~\ref{sec:degenerations_obstructions}, we give a list of obstructions to the existence of a degeneration between two orbits. In Section~\ref{sec:family_degenerations_Jordan_nets}, we find all degenerations between some families of Jordan nets. In Section~\ref{sec:proofs_main_results}, we prove the results from Section~\ref{sec:Jordan_nets_n<=6}. In Appendix~\ref{sec:irr_embeddings}, we classify embeddings of indecomposable Jordan algebras into $\S^n$ and prove the results from Section~\ref{sec:summary_embeddings}. In Appendix~\ref{sec:appendix_degenerations}, we list some families of degenerations between Jordan algebras.

\subsection*{Acknowledgements}
We would like to thank Jan Draisma for finding the example in Remark~\ref{re:sz_pencils_n=8} and Aline Marti for help with the proof of Proposition~\ref{prop:B231_not_to_C_6}. The first author was partially supported by Postdoc.Mobility Fellowship P400P2\_199196 from the Swiss National Science Foundation.

\section{Abstract Jordan algebras}\label{sec:abstract_jordan}

The goal of this section is to review the basic notions concerning (abstract unital) Jordan algebras. For Jordan algebras of dimension $\leq 4$, we give their classification up to isomorphism and the degenerations between the different isomorphism classes.

\begin{de}
An (abstract) {\em Jordan algebra} $\As$ is a complex vector space equipped with a symmetric bilinear operation $-\cdot-\colon\As\times\As\to\As$ such that
\[
(x\cdot x)\cdot(x\cdot y)=x\cdot((x\cdot x)\cdot y)
\]
holds for all $x,y\in\As$. This condition is called the {\em Jordan axiom}. A Jordan algebra $\As$ is called {\em unital} if there exists an $u\in\As$ such that $u\cdot x=x$ for all $x\in\As$.
\end{de}

All Jordan algebras in this paper are assumed to be unital. We often write $xy$ instead of $x\cdot y$ and $x^d$ instead of $x\cdot x^{d-1}$ for $d\geq 2$.

\begin{ex}
For any invertible matrix $U\in\S^n$, the operation
\begin{eqnarray*}
-\bullet_U-\colon\As\times\As&\to&\As\\
(X,Y)&\mapsto&\frac{1}{2}\left(XU^{-1}Y+YU^{-1}X\right)
\end{eqnarray*}
defines a Jordan algebra structure on $\S^n$ where the matrix $U$ is the unit.
\end{ex}

A morphism of Jordan algebras $\varphi\colon\As\to\Bs$ is a linear map that sends the unit of $\As$ to the unit of $\Bs$ such that $\varphi(x\cdot y)=\varphi(x)\cdot\varphi(y)$ for all $x,y\in\As$. An isomorphism is a morphism which is invertible as a linear map.

\begin{ex}
Let $d\geq 1$ and $0\leq r\leq d$ be integers and let $\beta\colon \CC^d\times\CC^d\to\CC$ be a symmetric bilinear form of rank~$r$.
Define $\Js^d_{r,\beta}$ to be the vector space $\CC\times\CC^d$ equipped with the operation 
\[
(\lambda,v)\cdot(\mu,w):=(\lambda\mu+\beta(v,w),\lambda w+\mu v).
\]
Then $\Js^d_{r,\beta}$ is a Jordan algebra. One can check that the isomorphism type of $\Js^d_{r,\beta}$ only depends on~$(d,r)$. When $\beta(v,w)=v_1w_1+\ldots+v_rw_r$, we also denote $\Js^d_{r,\beta}$ by $\Js^d_r$.
\end{ex}

Let $\As$ be a Jordan algebra with unit $u$.

\begin{de}
The {\em rank} of $\As$ is the minimal number $\rk(A):=k\geq1$ such that $u,x,x^2,\ldots,x^k$ are linearly dependent for all $x\in\As$.
\end{de}

\begin{ex}
Let $d\geq 1$ and $0\leq r\leq d$ again be integers and consider the Jordan algebra $\Js^d_r$. It has unit $u=(1,0)$. Let $x=(\lambda,v)\in\As$ be any element. Then 
\[
\spann(u,x,x^2,\ldots,x^k)=\spann(u,x-\lambda u,(x-\lambda u)^2,\ldots,(x-\lambda u)^k)
\]
for each $k\geq 1$. So for the purposes of determining the rank of $\Js^d_r$, it suffices to consider the case where $\lambda=0$. We see that $u=(1,0)$ and $x=(0,v)$ are linearly independent in general and that $u=(1,0)$ and $x^2=(\beta(v,v),0)$ are linearly dependent for all $v\in\CC^d$. Hence $\rk(\Js^d_r)=2$.
\end{ex}

\begin{de}
We say that $\As$ is {\em decomposable} when $\As\cong\Bs_1\times\Bs_2$ where $\Bs_1,\Bs_2$ are Jordan algebras and $(a_1,a_2)\cdot(b_1,b_2)=(a_1b_1,a_2b_2)$ for all $a_1,b_1\in\Bs_1$ and $a_2,b_2\in\Bs_2$.
\end{de}

\begin{prop}
Let $\Bs_1,\Bs_2$ be Jordan algebras with units $u_1,u_2$ and suppose that $\As\cong\Bs_1\times\Bs_2$. Then 
\[
\rk(A)=\rk(\Bs_1)+\rk(\Bs_2).
\]
\end{prop}
\begin{proof}
Let $x\in\As$ be an element and $P(X)=a_0+a_1X+\ldots+a_dX^d\in\CC[X]$ a polynomial. Then we write $P(x):= a_0u+a_1x+\ldots+a_dx^d$. The minimal polynomial $P_x$ of $x$ is the monic generator of the ideal $\{P\in\CC[X]\mid P(x)=0\}$ of $\CC[X]$. Note that 
\[
\rk(\As)=\max_{x\in\As}(\deg P_x).
\]
Now let $b_1,b_2$ be elements of $\Bs_1,\Bs_2$. Then $P_{(b_1,b_2)}=\lcm(P_{b_1},P_{b_2})$ and so $\rk(\As)\leq\rk(\Bs_1)+\rk(\Bs_2)$. Assume that $\rk(\Bs_i)=\deg P_{b_i}$. Then there exists a $\lambda\in\CC$ such that $P_{b_1}, P_{b_2}(X-\lambda)=P_{b_2+\lambda u_2}$ have distinct roots. It follows that 
\[
\rk(\As)\geq \deg P_{(b_1,b_2+\lambda u_2)}=\deg \lcm(P_{b_1},P_{b_2}(X-\lambda))=\deg P_{b_1}+\deg P_{b_2}=\rk(\Bs_1)+\rk(\Bs_2)
\]
and hence $\rk(\As)=\rk(\Bs_1)+\rk(\Bs_2)$.
\end{proof}

Let $n\geq1$ be an fixed integer and let $x_1,\ldots,x_n$ be a basis of a Jordan algebra $\As$. Then the Jordan algebra structure is determined by the constants $c_i^{j_1,j_2}\in\CC$ such that
\[
x_{j_1}\cdot x_{j_2}= \sum_{i=1}^k c_i^{j_1,j_2}\cdot x_i.
\]
We define the space $\Jord_n$ of abstract Jordan algebras of dimension $n$ to be the subvariety of $\CC^{n\times n\times n}$ consisting of all elements $(c_i^{j_1,j_2})_{i,j_1,j_2}$ defining a Jordan algebra structure. This means that $c_i^{j_1,j_2}=c_i^{j_2,j_1}$ and that the Jordan axiom is satisfied. Note that $\GL_n$ acts on $\Jord_n$ via base change. The orbit of a Jordan algebra $\As$ consists of all Jordan algebras isomorphic to $\As$.

\begin{de}
We say that a Jordan algebra $\As$ {\em degenerates to} a Jordan algebra $\Bs$ when there exists a matrix $g(t)\in\GL_n(\CC[t^{\pm1}])$ such that $\Bs=\lim_{t\to0}g(t)\cdot \As$. We denote this by $\As\to\Bs$. Equivalently, we say that $\As$ {\em degenerates to} $\Bs$ ({\em topologically}) when $\Bs$ is in the orbit-closure of $\As$. 
\end{de}

For a proof that these definitions are equivalent, see the proof of \cite[Theorem 20.24]{BCS:complextheory}.

\begin{ex}
There are two isomorphism classes of $2$-dimensional Jordan algebras, namely $\CC\times\CC$ and $\Js^1_0$. The former has a basis $x,y$ with $x^2=x$, $xy=0$ and $y^2=y$. The latter has a basis $u,z$ where $u$ is the unit and $z^2=0$. For $t\neq0$, the basis $(u,z_t)=(x+y,ty)$ satisfies
\begin{eqnarray*}
u^2&=&1\cdot u+0\cdot z_t,\\
uz_t&=&0\cdot u+1\cdot z_t, \\
z_t^2&=&0\cdot u+t\cdot z_t
\end{eqnarray*} 
Taking the limit of the structure constants $(1,0,0,1,0,t)$ for $t\to0$, we find the structure constants $(1,0,0,1,0,0)$ of $\Js^1_0$. Hence $\CC\times\CC\to\Js^1_0$.
\end{ex}

The Jordan algebras of dimension $4$ were classified by Martin.

\begin{thm}[{Martin \cite{M:FourDimJordan}}]
Let $\As$ be an indecomposable unital Jordan algebra of dimension $\leq 4$. Then $\As$ is isomorphic to one of the following Jordan algebras:
\begin{enumerate}
\item $\CC$
\item $\Js_0^1$
\item $\Js_0^2,\Js_1^2,\Js_2^2,\CC[x]/(x^3)$
\item $\Js_0^3,\Js_1^3,\Js_2^3,\Js_3^3,\CC[x]/(x^4)$, the subalgebras
\[
\Es_1:=\begin{pmatrix}v&w&x\\w&y\\x\end{pmatrix},\Es_2:=\begin{pmatrix}v&x&w\\x\\w&&&y\\&&y\end{pmatrix},\Es_3:=\begin{pmatrix}y&x&u&z\\x&u\\u\\z&&&u\end{pmatrix}
\]
of $\S^3,\S^4$ and $\Es_4:=\CC[x,y]/(x^2,xy,y^2)$.
\end{enumerate}
\end{thm}

We know the degeneration diagrams for Jordan algebras of dimension $3$ and $4$.

\begin{samepage}
\begin{thm}\label{thm:abstract_diagram3}
The following diagram gives all degenerations between $3$-dimensional Jordan algebras.
\end{thm}
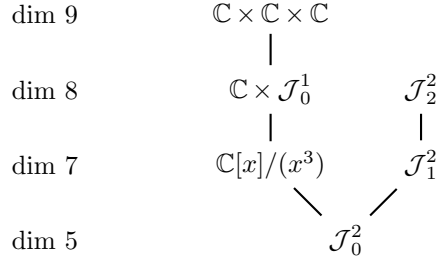
\begin{center}
\begin{tikzpicture}
	\node () at (-2,9) {dim $9$};
	\node () at (-2,8) {dim $8$};
	\node () at (-2,7) {dim $7$};
	\node () at (-2,6) {dim $5$};
	
    \node (1a) at (1,9) {$\CC\times\CC\times\CC$};
    \node (1b) at (3,8) {$\Js_2^2$};
    \node (2a) at (1,8) {$\CC\times\Js_0^1$};
    \node (2b) at (3,7) {$\Js_1^2$};
    \node (3a) at (1,7) {$\CC[x]/(x^3)$};
    \node (3b) at (2,6) {$\Js_0^2$};
    
    \draw [thick] (1a) -- (2a);
    \draw [thick] (2a) -- (3a);
    \draw [thick] (3a) -- (3b);
    \draw [thick] (1b) -- (2b);
    \draw [thick] (2b) -- (3b);
\end{tikzpicture}
\end{center}
\end{samepage}

To prove the theorem, we need the following proposition.

\begin{prop}
If $\As\to\Bs$, then $\rk(B)\leq\rk(\As)$.
\end{prop} 
\begin{proof}
This follows from the fact that bounded rank is a closed condition.
\end{proof}

\begin{proof}[Proof of Theorem~\ref{thm:abstract_diagram3}]
The degenerations and the fact that $\CC\times\CC\times\CC\not\to\Js^2_1$ are obtained by \cite{M:DeformationsThesis} and based on partial results from \cite{KS:DeformationsJordan3}. We have $\rk(\Js^2_2)=2<3=\rk(\CC[x]/(x^3))$ and hence $\Js_2^2\not\to\CC[x]/(x^3)$ by the previous proposition.
\end{proof}

\begin{thm}[{Martin \cite{M:DeformationsThesis}, Kashuba-Martin \cite{KM:DeformationsJordan4}}]\label{thm:abstract_diagram4}
The following diagram gives all degenerations between $4$-dimensional Jordan algebras.
\end{thm}
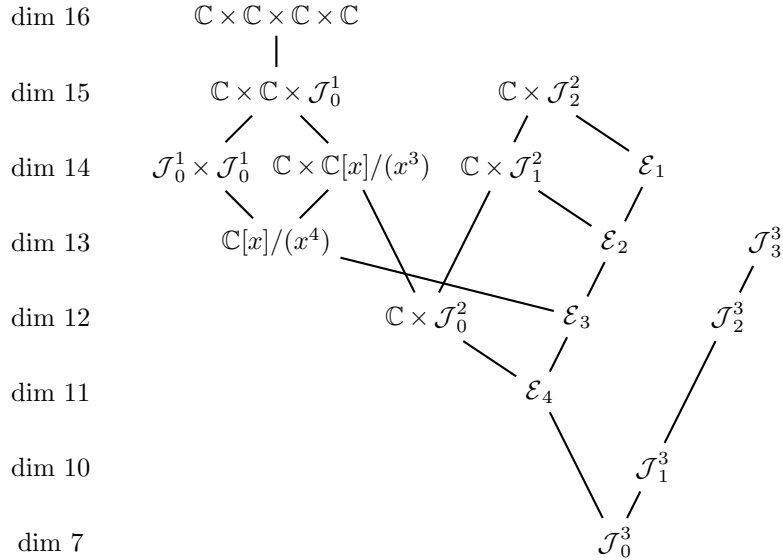
\begin{center}
\begin{tikzpicture}
	\node () at (-2,16) {dim $16$};
	\node () at (-2,15) {dim $15$};
	\node () at (-2,14) {dim $14$};
	\node () at (-2,13) {dim $13$};
	\node () at (-2,12) {dim $12$};
	\node () at (-2,11) {dim $11$};	
	\node () at (-2,10) {dim $10$};
	\node () at (-2,9) {dim $7$};
	
    \node (1) at (1,16) {$\CC\times\CC\times\CC\times\CC$};
    \node (2) at (1,15) {$\CC\times\CC\times\Js_0^1$};
    \node (3) at (0,14) {$\Js_0^1\times\Js_0^1$};
    \node (4) at (4.5,15) {$\CC\times\Js_2^2$};
    \node (5) at (4,14) {$\CC\times\Js_1^2$};
    \node (6) at (3,12) {$\CC\times\Js_0^2$};
    \node (7) at (2,14) {$\CC\times\CC[x]/(x^3)$};
    \node (8) at (7.5,13) {$\Js_3^3$};
    \node (9) at (7,12) {$\Js_2^3$};
    \node (10) at (6,10) {$\Js_1^3$};
    \node (11) at (5.5,9) {$\Js_0^3$};
    \node (12) at (1,13) {$\CC[x]/(x^4)$};
    \node (13) at (6,14) {$\Es_1$};
    \node (16) at (5.5,13) {$\Es_2$};
    \node (18) at (5,12) {$\Es_3$};
    \node (20) at (4.5,11) {$\Es_4$};

    \draw [thick] (1) -- (2);
    \draw [thick] (2) -- (3);
    \draw [thick] (2) -- (7);
    \draw [thick] (3) -- (12);
    \draw [thick] (4) -- (5);
    \draw [thick] (4) -- (13);
    \draw [thick] (5) -- (6);
    \draw [thick] (5) -- (16);
    \draw [thick] (6) -- (20);
    \draw [thick] (7) -- (6);
    \draw [thick] (7) -- (12);
    \draw [thick] (8) -- (9);
    \draw [thick] (9) -- (10);
    \draw [thick] (10) -- (11);
    \draw [thick] (12) -- (18);
    \draw [thick] (13) -- (16);
    \draw [thick] (16) -- (18);
    \draw [thick] (18) -- (20);
    \draw [thick] (20) -- (11);
\end{tikzpicture}
\end{center}

\section{Jordan nets in $\S^n$ for $n\leq 7$ and Jordan webs in $\S^n$ for $n\leq 5$}\label{sec:summary_embeddings}

\begin{de}
Let $n\geq 1$ be an integer. An {\em embedding} of a Jordan algebra $\As$ into $\S^n$ is the image of an injective morphism of Jordan algebras $\As\to\S^n$, where $\S^n$ is equipped with the product $\bullet_U$ for any invertible $U\in\S^n$. A subspace $\Ls\subseteq\S^n$ is called a {\em Jordan space} if it is the embedding of some Jordan algebra into $\S^n$. We denote the set of $m$-dimensional Jordan spaces in $\S^n$ by $\Jo(m,\S^n)$. We denote the subset of $\Jo(m,\S^n)$ of subspaces containing ${\bf 1}_n$ by $\Jo_{\bf 1}(m,\S^n)$.
\end{de}

Jordan spaces of dimensions $2,3$ and $4$ are also called {\em Jordan pencils}, {\em nets} and {\em webs}, respectively.

\begin{de}
Two subspaces $\Ls,\Ls'\subseteq\S^n$ are {\em congruent} if $\Ls'=P\Ls P^\top$ for some $P\in\GL(n)$. When $P\in\O(n)$, the spaces $\Ls,\Ls'$ are called {\em orthogonally congruent}.
\end{de}

The sets $\Jo(m,\S^n)$ and $\Jo_{\bf 1}(m,\S^n)$ are varieties \cite{BES:JordanSpaces} that are stable under congruence and orthogonal congruence, respectively. The goal of this section is to classify elements of $\Jo(m,\S^n)$ up to congruence. Every element of $\Jo(m,\S^n)$ is congruent to an element of $\Jo_{\bf 1}(m,\S^n)$. So equivalently, we wish to classify the elements of $\Jo_{\bf 1}(m,\S^n)$ up to orthogonal congruence. In this section, we list the orbits of Jordan nets and webs in low dimension. 

\subsection{Jordan nets}
Every Jordan net is the embedding of one of the following Jordan algebras:
\[
\CC\times\CC\times\CC,\CC\times\Js^1_0,\CC[x]/(x^3),\Js^2_2,\Js^2_1,\Js^2_0
\]
For the first five of these algebras, we classify the orbits for general $n$.

\begin{samepage}
\begin{thm}\label{thm:classification_nets_general}
Let $n\geq 1$ be an integer.
\begin{itemize}
\item[(1)] Every embedding of $\CC\times\CC\times\CC$ into $\S^n$ is congruent to\medskip

$
A^{(1)}_{k_1,k_2,k_3}:=\Diag(x{\bf 1}_{k_1+k_2+k_3},y{\bf 1}_{k_2+k_3},z{\bf 1}_{k_3})
$\medskip

\noindent for some $k_1,k_2\geq0$ and $k_3\geq1$ with $k_1+2k_2+3k_3=n$.
\item[(2)] Every embedding of $\CC\times\Js^1_0$ into $\S^n$ is congruent to\medskip

$
A^{(2)}_{r,k_1,k_2}:=\Diag\left(x{\bf 1}_r,{\bf 1}_{k_2}\otimes\begin{pmatrix}z&y\\y\end{pmatrix},y{\bf 1}_{k_1}\right)
$\medskip

\noindent for some $r,k_2\geq1$ and $k_1\geq0$ with $k_1+2k_2=n-r$.
\item[(3)] Every embedding of $\CC[x]/(x^3)$ into $\S^n$ is congruent to\medskip

$
A^{(3)}_{k_1,k_2,k_3}:=\Diag\left({\bf 1}_{k_3}\otimes\begin{pmatrix}z&y&x\\y&x\\x\end{pmatrix},{\bf 1}_{k_2}\otimes\begin{pmatrix}y&x\\x\end{pmatrix},x{\bf 1}_{k_1}\right)
$\medskip

\noindent for some $k_1,k_2\geq0$ and $k_3\geq1$ with $k_1+2k_2+3k_3=n$.
\item[(4)] The Jordan algebra $\Js^2_2$ has no embeddings into $\S^n$ when $n$ is odd. When $n$ is even, every embedding of $\Js^2_2$ into $\S^n$ is congruent to $B^{(1)}_{n/2}:=\S^2\otimes{\bf 1}_{n/2}$.
\item[(5)] Every embedding of $\Js^2_1$ into $\S^n$ is congruent to\medskip

$
B^{(2)}_{k,\ell_1,\ell_2}:=\Diag\left(\begin{pmatrix}xJ_{\ell_2}&z\Diag({\bf 1}_k,{\bf 0}_{\ell_2-k})\\z\Diag({\bf 1}_k,{\bf 0}_{\ell_2-k})&yJ_{\ell_2}\end{pmatrix},y {\bf 1}_{\ell_1}\right)
$\medskip

\noindent for some $\ell_1\geq0$, $\ell_2\geq2$ and $1\leq k\leq \ell_2/2$ such that $\ell_1+2\ell_2=n$.
\end{itemize}
\end{thm}
\end{samepage}

For embeddings of $\Js^2_0$ we classify the orbits of embeddings into $\S^n$ for $n\leq 7$.

\begin{thm}\label{thm:classification_nets_J20}
For $n\in\{1,2,3\}$, the Jordan algebra $\Js^2_0$ has no embeddings into $\S^n$. For $4\leq n\leq7$, every embedding of $\Js^2_0$ into $\S^n$ is congruent $C_{n,i}=\CC J_n+\Diag(\mathcal{P}_{\lfloor n/2\rfloor,i},{\bf 0}_{\lceil n/2\rceil})$ for some $i$ where
\begin{align*}
&\mathcal{P}_{2,1}:=\begin{pmatrix}x\\&y\end{pmatrix},&\hspace*{-30pt}&\mathcal{P}_{2,2}:=\begin{pmatrix}y&x\\x\end{pmatrix},\\
&\mathcal{P}_{3,1}:=\begin{pmatrix}x\\&y\\&&x+y\end{pmatrix},&\hspace*{-30pt}&
\mathcal{P}_{3,2}:=\begin{pmatrix}x&y\\y\\&&x\end{pmatrix},\\
&\mathcal{P}_{3,3}:=\begin{pmatrix}&y&x\\y&x\\x\end{pmatrix},&\hspace*{-30pt}&
\mathcal{P}_{3,4}:=\begin{pmatrix}x\\&x\\&&y\end{pmatrix},\\
&\mathcal{P}_{3,5}:=\begin{pmatrix}y&x\\x\\&&x\end{pmatrix},&\hspace*{-30pt}&
\mathcal{P}_{3,6}:=\begin{pmatrix}&x&y\\x\\y\end{pmatrix},\\
&\mathcal{P}_{3,7}:=\begin{pmatrix}x\\&y\\&&0\end{pmatrix},&\hspace*{-30pt}&
\mathcal{P}_{3,8}:=\begin{pmatrix}y&x\\x\\&&0\end{pmatrix}.
\end{align*}
\end{thm}

\begin{proof}[Proof of Theorem~\ref{thm:classification_nets_general}]
(1)
Every embedding of $\CC\times\CC\times\CC$ is congruent to $\Diag(\Ls_1,\Ls_2,\Ls_3)$ for some embeddings $\Ls_1,\Ls_2,\Ls_3$ of $\CC$ by Proposition~\ref{prop:embedding_product}. By Proposition~\ref{prop:classify_CC}, we may assume that $\Ls_i=\CC{\bf 1}_{n_i}$ for some $n_i\geq1$. After reordening, we can write $n_1=k_1+k_2+k_3$, $n_2=k_2+k_3$ and $n_3=k_3$ for some $k_1,k_2\geq0$ and $k_3\geq 1$ such that $k_1+2k_2+3k_3=n$.

(2)
Every embedding of $\CC\times\Js^1_0$ is congruent to $\Diag(\Ls_1,\Ls_2)$ for some embeddings $\Ls_1$ of $\CC$ and $\Ls_2$ of $\Js^1_0$ by Proposition~\ref{prop:embedding_product}. By Proposition~\ref{prop:classify_CC}, we may assume that $\Ls_1=\CC{\bf 1}_r$ for some $r\geq1$. By Proposition~\ref{prop:classify_CC[x]/x^m},  we may assume that 
\[
\Ls_2=\Diag\left({\bf 1}_{k_2}\otimes\begin{pmatrix}z&y\\y\end{pmatrix},y{\bf 1}_{k_1}\right)
\]
for some $k_1\geq 0$ and $k_2\geq 1$ with $k_1+2k_2=n-r$.

(3)
This is Proposition~\ref{prop:classify_CC[x]/x^m}.

(4)
This is Proposition~\ref{prop:classify_J22}.

(5)
This is Proposition~\ref{prop:classify_J21}.
\end{proof}

\begin{proof}[Proof of Theorem~\ref{thm:classification_nets_J20}]
(1)
A pencil in $\S^n$ always contains a matrix of rank $\geq 2$. On the other hand, all matrices in $\S^n$ whose squares are zero have rank $\leq n/2$. So square-zero pencils cannot exist in $\S^n$ when $n<4$.
 
(2)
This follows from Propositions~\ref{prop:squarezero_pencil} and~\ref{prop:pencilsS2}.

(3)
This follows from Propositions~\ref{prop:squarezero_pencil} and~\ref{prop:pencilsS2}.

(4)
This follows from Propositions~\ref{prop:squarezero_net} and~\ref{prop:pencilsS3}.

(5)
This follows from Propositions~\ref{prop:squarezero_net} and~\ref{prop:pencilsS3}.
\end{proof}

\subsection{Jordan webs}
Every Jordan web is the embedding of one of the following Jordan algebras:
\[
\CC\times\CC\times\CC\times\CC,\CC\times\CC\times\Js^1_0,\Js^1_0\times\Js^1_0,\CC\times\CC[x]/(x^3),\CC[x]/(x^4),\CC\times\Js^2_2,\CC\times\Js^2_1,\Js^3_3,\Js^3_2,
\]
\[
\Js^3_1,\CC\times\Js^2_0,\Es_1,\Es_2,\Es_3,\Es_4,\Js^3_0.
\]
For the algebras on the first line, we classify the orbits for general $n$.

\begin{thm}\label{thm:classification_webs_general}
Let $n\geq 1$ be an integer.
\begin{itemize}
\item[(1)] Every embedding of $\CC\times\CC\times\CC\times\CC$ into $\S^n$ is congruent to\medskip

$
\hspace{-12pt}A^{(1)}_{k_1,k_2,k_3,k_4}:=\Diag(x{\bf 1}_{k_1+k_2+k_3+k_4},y{\bf 1}_{k_2+k_3+k_4},z{\bf 1}_{k_3+k_4},w{\bf 1}_{k_4})
$\medskip

\noindent for some $k_1,k_2,k_3\geq0$ and $k_4\geq1$ with $k_1+2k_2+3k_3+4k_4=n$.

\item[(2)] Every embedding of $\CC\times\CC\times\Js^1_0$ into $\S^n$ is congruent to\medskip

$
\hspace{-12pt}A^{(2)}_{k_1,k_2,\ell_1,\ell_2}:=\Diag\left(x{\bf 1}_{k_1+k_2},y{\bf 1}_{k_2},{\bf 1}_{\ell_2}\otimes\begin{pmatrix}w&z\\z\end{pmatrix},z{\bf 1}_{\ell_1}\right)
$\medskip

\noindent for some $k_1,\ell_1\geq0$ and $k_2,\ell_2\geq1$ with $k_1+2k_2+\ell_1+2\ell_2=n$.

\item[(3)] Every embedding of $\Js^1_0\times\Js^1_0$ into $\S^n$ is congruent to\medskip

$
\hspace{-12pt}A^{(3)}_{k_1,k_2,\ell_1,\ell_2}:=\Diag\left({\bf 1}_{k_2}\otimes\begin{pmatrix}y&x\\x\end{pmatrix},x{\bf 1}_{k_1},{\bf 1}_{\ell_2}\otimes\begin{pmatrix}w&z\\z\end{pmatrix},z{\bf 1}_{\ell_1}\right)
$\medskip

\noindent for some $k_1,\ell_1\geq0$ and $k_2,\ell_2\geq1$ with $k_1+2k_2+\ell_1+2\ell_2=n$, where $A^{(3)}_{k_1,k_2,\ell_1,\ell_2}$ and $A^{(3)}_{\ell_1,\ell_2,k_1,k_2}$ are congruent.

\item[(4)] Every embedding of $\CC\times\CC[x]/(x^3)$ into $\S^n$ is congruent to\medskip

$
\hspace{-12pt}A^{(4)}_{r,k_1,k_2,k_3}:=\Diag\left(w{\bf 1}_r,{\bf 1}_{k_3}\otimes\begin{pmatrix}z&y&x\\y&x\\x\end{pmatrix},{\bf 1}_{k_2}\otimes\begin{pmatrix}y&x\\x\end{pmatrix},x{\bf 1}_{k_1}\right)
$\medskip

\noindent for some $k_1,k_2\geq0$ and $r,k_3\geq1$ with $k_1+2k_2+3k_3=n-r$.

\item[(5)] Every embedding of $\CC[x]/(x^4)$ into $\S^n$ is congruent to\medskip

$
\hspace{-12pt}A^{(5)}_{k_1,k_2,k_3,k_4}:=\Diag\left({\bf 1}_{k_4}\otimes\begin{pmatrix}w&z&y&x\\z&y&x\\y&x\\x\end{pmatrix},{\bf 1}_{k_3}\otimes\begin{pmatrix}z&y&x\\y&x\\x\end{pmatrix},{\bf 1}_{k_2}\otimes\begin{pmatrix}y&x\\x\end{pmatrix},x{\bf 1}_{k_1}\right)
$\medskip

\noindent for some $k_1,k_2,k_3\geq0$ and $k_4\geq1$ with $k_1+2k_2+3k_3+4k_4=n$.

\item[(6)] Every embedding of $\CC\times\Js^2_2$ into $\S^n$ is congruent to\medskip

$
\hspace{-12pt}B^{(1)}_{k_1,k_2}:=\Diag\left(w{\bf 1}_{k_1},\begin{pmatrix}x&y\\y&z\end{pmatrix}\otimes{\bf 1}_{k_2}\right)
$\medskip

\noindent for some $k_1,k_2\geq1$ with $k_1+2k_2=n$.

\item[(7)] Every embedding of $\CC\times\Js^2_1$ into $\S^n$ is congruent to\medskip

$
\hspace{-12pt}B^{(2)}_{r,k,\ell_1,\ell_2}:=\Diag\left(w{\bf 1}_r,\begin{pmatrix}xJ_{\ell_2}&z\Diag({\bf 1}_k,{\bf 0}_{\ell_2-k})\\z\Diag({\bf 1}_k,{\bf 0}_{\ell_2-k})&yJ_{\ell_2}\end{pmatrix},y {\bf 1}_{\ell_1}\right)
$\medskip

\noindent for some $\ell_1\geq0$, $r\geq 1$, $\ell_2\geq2$ and $1\leq k\leq \ell_2/2$ such that $\ell_1+2\ell_2=n-r$.

\item[(8)] The Jordan algebra $\Js^3_3$ has no embeddings into $\S^n$ when $4\nmid n$. When $4\mid n$, every embedding of $\Js_3^3$ into $\S^n$ is congruent to \medskip

$
\hspace{-12pt}C^{(1)}_{n/4}:=\begin{pmatrix}x{\bf 1}_{n/2}&z{\bf 1}_{n/2}+w{\bf 1}_{n/4}\otimes\begin{pmatrix}&1\\-1\end{pmatrix}\\z{\bf 1}_{n/2}+w{\bf 1}_{n/4}\otimes\begin{pmatrix}&-1\\1\end{pmatrix}
&y{\bf 1}_{n/2}\end{pmatrix}.
$

\item[(9)] The Jordan algebra $\Js^3_2$ has no embeddings into $\S^n$ when $n$ is odd. When $n$ is even, every embedding of $\Js_2^3$ into $\S^n$ is congruent to \medskip

$
\hspace{-12pt}C^{(2)}_{n/2,k}:=\begin{pmatrix}x{\bf J}_{n/2}&\!z{\bf J}_{n/2}+w\Diag\!\left(\!{\bf 1}_k\!\otimes\!\begin{pmatrix}&1\\-1\!\!\!\!\!\end{pmatrix}\!,{\bf 0}_{n/2-2k}\!\right)\!\\\!z{\bf J}_{n/2}+w\Diag\!\left(\!{\bf 1}_k\!\otimes\!\begin{pmatrix}&-1\\1\!\!\!\!\!\end{pmatrix}\!\!,{\bf 0}_{n/2-2k}\!\right)\!
&y{\bf J}_{n/2}\end{pmatrix}
$\medskip

\noindent for some integer $1\leq k\leq n/8$.
\end{itemize}
\end{thm}

For embeddings of $\Js^3_1,\CC\times\Js^2_0,\Es_1,\Es_2,\Es_3,\Es_4,\Js^3_0$ we classify the orbits of embeddings into $\S^n$ for $n\leq 5$.

\begin{prop}\label{prop:classification_webs_n=3}
The Jordan algebras $\Js^3_1,\CC\times\Js^2_0,\Es_2,\Es_3,\Es_4,\Js^3_0$ have no embeddings into $\S^3$. Every embedding of $\Es_1$ into $\S^3$ is congruent to $E^{(1)}_3:=\Es_1$.
\end{prop}

\begin{thm}\label{thm:classification_webs_n=4}\phantom{space}
\begin{itemize}
\item[(1)] The Jordan algebra $\Js^3_1$ has no embedding into $\S^4$.
\item[(2)] The Jordan algebra $\CC\times\Js^2_0$ has no embedding into $\S^4$.
\item[(3)] Every embedding of $\Es_1$ into $\S^4$ is congruent to one of
\[
E^{(1)}_{4,1}:=\begin{pmatrix}v&w&x\\w&y\\x\\&&&x\end{pmatrix},E^{(1)}_{4,2}:=\begin{pmatrix}v&w&x\\w&y\\x\\&&&y\end{pmatrix}.
\]
\item[(4)] Every embedding of $\Es_2$ into $\S^4$ is congruent to
\[
E^{(2)}_4:=\begin{pmatrix}v&x&w\\x\\w&&&y\\&&y\end{pmatrix}.
\]
\item[(5)] Every embedding of $\Es_3$ into $\S^4$ is congruent to
\[
E^{(3)}_4:=\begin{pmatrix}y&x&u&z\\x&u\\u\\z&&&u\end{pmatrix}.
\]
\item[(6)] The Jordan algebra $\Es_4$ has no embeddings into $\S^4$.
\item[(7)] Every embedding of $\Js^3_0$ into $\S^4$ is congruent to
\[
F_4 := \begin{pmatrix}
x&y&&u\\y&z&u\\&u\\u
\end{pmatrix}.
\]
\end{itemize}
\end{thm}

\begin{thm}\label{thm:classification_webs_n=5}\phantom{space}
\begin{itemize}
\item[(1)] The Jordan algebra $\Js^3_1$ has no embedding into $\S^5$.
\item[(2)] Every embedding of $\CC\times\Js^2_0$ into $\S^5$ is congruent to one of
\[
D_{5,1}:=\begin{pmatrix}x\\&v&&&y\\&&w&y\\&&y\\&
y\end{pmatrix},~ D_{5,2}:=\begin{pmatrix}x\\&v&w&&y\\&w&&y\\&&y\\&y\end{pmatrix}.
\]
\item[(3)] Every embedding of $\Es_1$ into $\S^5$ is congruent to one of
\[
E^{(1)}_{5,1}:=\begin{pmatrix}v&w&x\\w&y\\x\\&&&x\\&&&&x\end{pmatrix},E^{(1)}_{5,2}:=\begin{pmatrix}v&w&x\\w&y\\x\\&&&x\\&&&&y\end{pmatrix},E^{(1)}_{5,3}:=\begin{pmatrix}v&w&x\\w&y\\x\\&&&y\\&&&&y\end{pmatrix}.
\]
\item[(4)] Every embedding of $\Es_2$ into $\S^5$ is congruent to one of
\[
E^{(2)}_{5,1}:=\begin{pmatrix}v&x&w\\x\\w&&&y\\&&y\\&&&&x\end{pmatrix},E^{(2)}_{5,2}:=\begin{pmatrix}v&x&w\\x\\w&&&y\\&&y\\&&&&y\end{pmatrix}.
\]
\item[(5)] Every embedding of $\Es_3$ into $\S^5$ is congruent to one of
\[
E^{(3)}_{5,1}:=\begin{pmatrix}y&x&u&z\\x&u\\u\\z&&&u\\&&&&u\end{pmatrix},E^{(3)}_{5,2}:=\begin{pmatrix}y&x&u&&z\\x&u\\u\\&&&x&u\\z&&&u\end{pmatrix}.
\]
\item[(6)] Every embedding of $\Es_4$ into $\S^5$ is congruent to one of
\[
E^{(4)}_{5,1}:=\begin{pmatrix}y&x&u\\x&u\\u\\&&&z&u\\&&&u\end{pmatrix},E^{(4)}_{5,2}:=\begin{pmatrix}y&x&u&z\\x&u\\u\\z&&&&u\\&&&u\end{pmatrix}.
\]
\item[(7)] Every embedding of $\Js^3_0$ into $\S^5$ is congruent to
\[
F_5:=\begin{pmatrix}
x&y&&&u\\y&z&&u\\&&u\\&u\\u
\end{pmatrix}.
\]
\end{itemize}
\end{thm}

\begin{proof}[Proof of Theorem~\ref{thm:classification_webs_general}]
(1)
Every embedding of $\CC\times\CC\times\CC\times\CC$ is congruent to $\Diag(\Ls_1,\Ls_2,\Ls_3,\Ls_4)$ for some embeddings $\Ls_1,\Ls_2,\Ls_3,\Ls_4$ of $\CC$ by Proposition~\ref{prop:embedding_product}. By Proposition~\ref{prop:classify_CC}, we may assume that $\Ls_i=\CC{\bf 1}_{n_i}$ for some $n_i\geq1$. After reordening, we can write $n_1=k_1+k_2+k_3+k_4$, $n_2=k_2+k_3+k_4$, $n_3=k_3+k_4$ and $n_4=k_4$ for some $k_1,k_2,k_3\geq0$ and $k_4\geq 1$ such that $k_1+2k_2+3k_3+4k_4=n$.

(2)
This follows from Propositions~\ref{prop:embedding_product},~\ref{prop:classify_CC} and~\ref{prop:classify_CC[x]/x^m}.

(3)
This follows from Propositions~\ref{prop:embedding_product} and~\ref{prop:classify_CC[x]/x^m}.

(4)
This follows from Propositions~\ref{prop:embedding_product},~\ref{prop:classify_CC} and~\ref{prop:classify_CC[x]/x^m}.

(5)
This is Proposition~\ref{prop:classify_CC[x]/x^m}.

(6)
This follows from Propositions~\ref{prop:embedding_product},~\ref{prop:classify_CC} and~\ref{prop:classify_J22}.

(7)
This follows from Propositions~\ref{prop:embedding_product},~\ref{prop:classify_CC} and~\ref{prop:classify_J21}.

(8)
This is Proposition~\ref{prop:classify_J33}.

(9)
This is Proposition~\ref{prop:classify_J32}.
\end{proof}

\begin{proof}[Proof of Proposition~\ref{prop:classification_webs_n=3}]
By \cite[Proposition 4.8]{BES:JordanSpaces}, we know that inside $\S^3$ there are two orbits of Jordan space, consisting of embeddings of the Jordan algebras $\CC\times\Js^2_2$ and $\Es_1$. The proposition follows.
\end{proof}

\begin{proof}[Proof of Theorem~\ref{thm:classification_webs_n=4}]
(1)
This is Proposition~\ref{prop:classify_J31}.

(2)
This follows by Proposition~\ref{prop:embedding_product} since $\Js^2_0$ has no embeddings into $\S^n$ for $n\leq 3$.

(3)
This is Proposition~\ref{prop:classify_E1}.

(4)
This is Proposition~\ref{prop:classify_E2}.

(5)
This is Proposition~\ref{prop:classify_E3}.

(6)
This is Proposition~\ref{prop:classify_E4}.

(7)
This follows from Propositions~\ref{prop:classify_J30} and~\ref{prop:squarezero_net}.
\end{proof}

\begin{proof}[Proof of Theorem~\ref{thm:classification_webs_n=5}]
(1)
This is Proposition~\ref{prop:classify_J31}.

(2)
This follows from Propositions~\ref{prop:embedding_product},~\ref{prop:classify_CC},~\ref{prop:classify_J20} and~\ref{prop:pencilsS2}.

(3)
This is Proposition~\ref{prop:classify_E1}.

(4)
This is Proposition~\ref{prop:classify_E2}.

(5)
This is Proposition~\ref{prop:classify_E3}.

(6)
This is Proposition~\ref{prop:classify_E4}.

(7)
This follows from Propositions~\ref{prop:classify_J30} and~\ref{prop:squarezero_net}.
\end{proof}

\section{Degenerations between Jordan nets}\label{sec:Jordan_nets_n<=6}

Let $\Ls,\Ls'\subseteq\S^n$ be Jordan spaces of the same dimension $m$ with bases $X_1,\ldots,X_m$ and $X'_1,\ldots,X'_m$.

\begin{de}\label{de:degeneration_jordan_space}
We say that $\Ls$ {\em degenerates to} $\Ls'$, denoted as $\Ls\to\Ls'$, when there exist matrices $P\in\GL_n(\CC[t^{\pm1}])$ and $Q\in\GL_m(\CC[t^{\pm1}])$ such that every entry of
\[
(Y_1,\ldots,Y_m):= (PX_1P^\top,\ldots,PX_mP^\top)Q
\]
is a matrix with coefficients in $\CC[t]$ and $X'_i=\lim_{t\to0}Y_i$ for all $i\in\{1,\ldots,m\}$. Note that the existence of such matrices $P,Q$ does not depend on the choice of the bases of $\Ls,\Ls'$. We say that $\Ls$ {\em degenerates to} $\Ls'$ ({\em topologically}) when $\overline{\{(Y_1,\ldots,Y_m)\mid \spann(Y_1,\ldots,Y_m)\in\GL_n\cdot \Ls\}}$ contains $(X_1',\ldots,X_m')$. 
\end{de}

For a proof that these definitions are equivalent, see the proof of \cite[Theorem 20.24]{BCS:complextheory}.

\begin{re}\label{re:over_CC((t))}
Note that $\CC[t^{\pm1}]$ is a subalgebra of $\CC((t))$ and $\CC[t]$ is a subalgebra of $\CC[[t]]$. So if $\Ls\to\Ls'$, then there exist matrices $P\in\GL_n(\CC((t)))$ and $Q\in\GL_m(\CC((t)))$ such that every entry of
\[
(Y_1,\ldots,Y_m):= (PX_1P^\top,\ldots,PX_mP^\top)Q
\]
is a matrix with coefficients in $\CC[[t]]$ and $X'_i=\lim_{t\to0}Y_i$ for all $i\in\{1,\ldots,m\}$. The converse also holds: let $\ell\geq0$ be an integer such that the coefficients of $P,Q$ are contained in $t^{-\ell}\CC[[t]]$ and write $(P,Q)=\sum_{k=-\ell}^\infty t^k(P_k,Q_k)$ with $(P_k,Q_k)\in\CC^{n\times n}\times\CC^{m\times m}$. Then $\lim_{t\to0}\det(P)$, $\lim_{t\to0}\det(Q)$ and $\lim_{t\to0}Y_i$ do not depend on the $P_k$'s, $Q_k$'s with $k>\ell\max(3,n,m)$. Hence, we are free to set these to zero and obtain matrices $P\in\GL_n(\CC[t^{\pm1}])$ and $Q\in\GL_m(\CC[t^{\pm1}])$ showing that $\Ls\to\Ls'$.
\end{re}

In this section, we give the diagrams of congruence-orbits of Jordan nets in $\S^n$ and their degenerations for $n\leq 6$. See Section~\ref{sec:proofs_main_results} for the proofs of these results.\bigskip

For $n=2$, the whole space $\S^2$ is the only Jordan net. 

\begin{figure}[H]
\begin{tikzpicture}

	\node () at (-2.5,2) {$\codim$};
	\node () at (-1.5,2) {$0$};
	
    \node (1a) at (0,2) {$B^{(1)}_1$};
    
\end{tikzpicture}
\caption{Jordan nets in $\S^2$ and their degenerations.}\label{fig:nets_n=2}
\end{figure}

For $n=3$, we have $3$ orbits which form a chain. For $n=4$, the degenerations where classified in \cite{BES:JordanSpaces}. See Figures~\ref{fig:nets_n=2}, \ref{fig:nets_n=3} and~\ref{fig:nets_n=4}.

\setlength{\columnsep}{1cm}
\begin{multicols}{2}
\begin{figure}[H]
\hspace{-30pt}
\begin{tikzpicture}

	\node () at (-2.5,4) {\phantom{$A^{(1)}_{0,0,1}$}};
	\node () at (-2.5,3) {$\codim$};
	\node () at (-2.5,2) {$\codim$};
	\node () at (-2.5,1) {$\codim$};
	\node () at (-1.5,3) {$3$};
	\node () at (-1.5,2) {$4$};
	\node () at (-1.5,1) {$5$};
	
    \node (1a) at (0,3) {$A^{(1)}_{0,0,1}$};
    \node (2a) at (0,2) {$A^{(2)}_{1,0,1}$};
    \node (3a) at (0,1) {$A^{(3)}_{0,0,1}$};

    \draw [thick] (1a) -- (2a);
    \draw [thick] (2a) -- (3a);

\end{tikzpicture}
\caption{Jordan nets in $\S^3$ and their degenerations.}\label{fig:nets_n=3}
\end{figure}
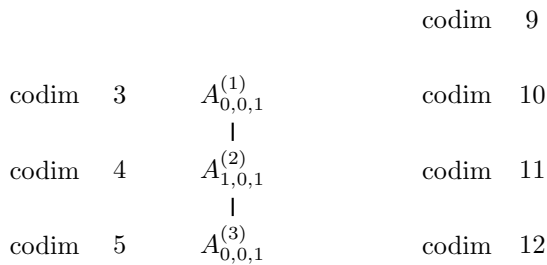

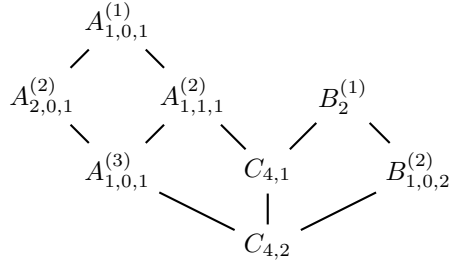
\begin{figure}[H]
\hspace{-45pt}
\begin{tikzpicture}

	\node () at (-2.5,3) {$\codim$};
    \node () at (-2.5,2) {$\codim$};
    \node () at (-2.5,1) {$\codim$};
    \node () at (-2.5,0) {$\codim$};	
	\node () at (-1.5,3) {$9$};
	\node () at (-1.5,2) {$10$};
	\node () at (-1.5,1) {$11$};
	\node () at (-1.5,0) {$12$};
	
    \node (1a) at (1,3) {$A^{(1)}_{1,0,1}$};
    \node (1b) at (4,2) {$B^{(1)}_2$};
    \node (2a1) at (0,2) {$A^{(2)}_{2,0,1}$};
    \node (2a2) at (2,2) {$A^{(2)}_{1,1,1}$};
    \node (2b) at (5.,1) {$B^{(2)}_{1,0,2}$};
    \node (3a) at (1.,1) {$A^{(3)}_{1,0,1}$};
    \node (3b1) at (3.,1) {$C_{4,1}$};
    \node (3b2) at (3.,0) {$C_{4,2}$};
    
    \draw [thick] (1a) -- (2a1);
    \draw [thick] (1a) -- (2a2);
    \draw [thick] (1b) -- (2b);
    \draw [thick] (1b) -- (3b1);
    \draw [thick] (2a1) -- (3a);
    \draw [thick] (2a2) -- (3b1);
    \draw [thick] (2b) -- (3b2);
    \draw [thick] (2a2) -- (3a);
    \draw [thick] (3a) -- (3b2);
    \draw [thick] (3b1) -- (3b2);
\end{tikzpicture}
\caption{Jordan nets in $\S^4$ and their degenerations.}\label{fig:nets_n=4}
\end{figure}
\end{multicols}
\setlength{\columnsep}{-2cm}

The first new case is that of $n=5$. We have the following result.

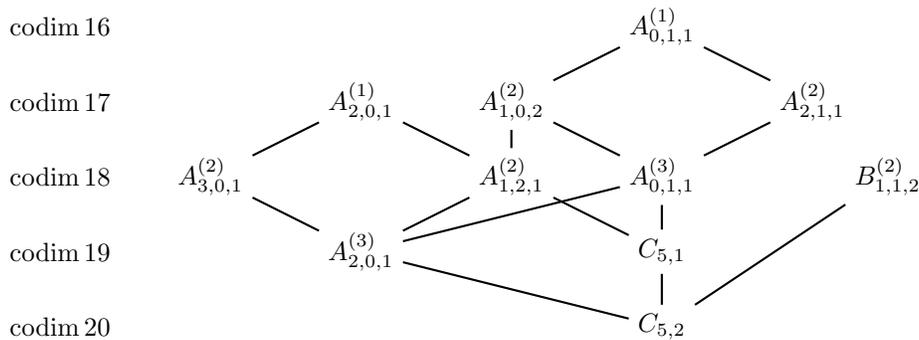
\begin{figure}[H]
\begin{center}
\begin{tikzpicture}

    \node () at (-7,-16) {$\codim 16$};
	\node () at (-7,-17) {$\codim 17$};
    \node () at (-7,-18) {$\codim 18$};
    \node () at (-7,-19) {$\codim 19$};
    \node () at (-7,-20) {$\codim 20$};	
	
    \node (1a1) at (-3,-17) {$A^{(1)}_{2,0,1}$};
    \node (1a2) at (1,-16) {$A^{(1)}_{0,1,1}$};
    \node (2a1) at (-1,-18) {$A^{(2)}_{1,2,1}$};
    \node (2a2) at (-1,-17) {$A^{(2)}_{1,0,2}$};
    \node (2a3) at (3,-17) {$A^{(2)}_{2,1,1}$};
    \node (2a4) at (-5,-18) {$A^{(2)}_{3,0,1}$};
    \node (2b) at (4,-18) {$B^{(2)}_{1,1,2}$};
    \node (3a1) at (1,-18) {$A^{(3)}_{0,1,1}$};
    \node (3a2) at (-3,-19) {$A^{(3)}_{2,0,1}$};
    \node (3b1) at (1,-19) {$C_{5,1}$};
    \node (3b2) at (1,-20) {$C_{5,2}$};
    
    \draw [thick] (1a2) -- (2a2);
    \draw [thick] (1a2) -- (2a3);
    \draw [thick] (1a1) -- (2a1);
    \draw [thick] (1a1) -- (2a4);
    \draw [thick] (2a2) -- (2a1);
    \draw [thick] (2a2) -- (3a1);
    \draw [thick] (2a3) -- (3a1);
    \draw [thick] (2a4) -- (3a2);
    \draw [thick] (2a1) -- (3a2);
    \draw [thick] (2a1) -- (3b1);
    \draw [thick] (3a1) -- (3a2);
    \draw [thick] (3a1) -- (3b1);
    \draw [thick] (2b) -- (3b2);
    \draw [thick] (3a2) -- (3b2);
    \draw [thick] (3b1) -- (3b2);    
 
\end{tikzpicture}
\end{center}
\caption{Jordan nets in $\S^5$ and their degenerations.}\label{fig:nets_n=5}
\end{figure}

\begin{thm}\label{thm:diagramNetsS5}
The diagram in Figure~\ref{fig:nets_n=5} describes all degenerations of Jordan nets in $\S^5$.
\end{thm}

Next is the case of $n=6$. We have the following conjecture and weaker statement.

\begin{conj}\label{conj:B1nottoC6}
We have $B^{(1)}_3\not\to C_{6,6}$.
\end{conj}

\begin{prop}\label{prop:B231_not_to_C_6}
We have $B^{(2)}_{1,0,3}\not\to C_{6,6}$.
\end{prop}

Assuming the conjecture, the diagram in Figure~\ref{fig:nets_n=6} is already complete.

\begin{thm}\label{thm:diagramNetsS6}
Apart from possibly the dotted line, the diagram in Figure~\ref{fig:nets_n=6} describes all degenerations of Jordan nets in $\S^6$.
\end{thm}

\begin{figure}[H]
\begin{center}
\begin{tikzpicture}

    \node () at (-10,-24) {$\codim 24$};
    \node () at (-10,-25) {$\codim 25$};
    \node () at (-10,-26) {$\codim 26$};
	\node () at (-10,-27) {$\codim 27$};
    \node () at (-10,-28) {$\codim 28$};
    \node () at (-10,-29) {$\codim 29$};
    \node () at (-10,-30) {$\codim 30$};	
	
    \node (1a1) at (-7,-27) {$A^{(1)}_{3,0,1}$};
    \node (1a2) at (-5,-25) {$A^{(1)}_{1,1,1}$};
    \node (1a3) at (-1,-24) {$A^{(1)}_{0,0,2}$};
    
    \node (1b) at (3.5,-25) {$B^{(1)}_3$};
    
    \node (2a1) at (-6,-28) {$A^{(2)}_{1,3,1}$};
    \node (2a2) at (-4,-26) {$A^{(2)}_{1,1,2}$};
    \node (2a3) at (-2.5,-26) {$A^{(2)}_{2,2,1}$};
    \node (2a4) at (-1,-25) {$A^{(2)}_{2,0,2}$};
    \node (2a5) at (-6,-26) {$A^{(2)}_{3,1,1}$};
    \node (2a6) at (-8,-28) {$A^{(2)}_{4,0,1}$};
    
    \node (2b1) at (4,-27) {$B^{(2)}_{1,2,2}$};
    \node (2b2) at (2,-26) {$B^{(2)}_{1,0,3}$};
    
    \node (3a1) at (-1,-26) {$A^{(3)}_{0,0,2}$};
    \node (3a2) at (-4,-27) {$A^{(3)}_{1,1,1}$};
    \node (3a3) at (-6,-29) {$A^{(3)}_{3,0,1}$};
    
    \node (3b1) at (.5,-25) {$C_{6,1}$};
    \node (3b2) at (.5,-26) {$C_{6,2}$};
    \node (3b3) at (2,-27) {$C_{6,3}$};
    \node (3b4) at (-1,-27) {$C_{6,4}$};
    \node (3b5) at (-1,-28) {$C_{6,5}$};
    \node (3b6) at (2,-29) {$C_{6,6}$};
    \node (3b7) at (-1,-29) {$C_{6,7}$};
    \node (3b8) at (-1,-30) {$C_{6,8}$};
    
    \draw [thick] (1a1) -- (2a1); 
    \draw [thick] (1a1) -- (2a6); 
    
    \draw [thick] (1a2) -- (2a2); 
    \draw [thick] (1a2) -- (2a3); 
    \draw [thick] (1a2) -- (2a5); 
    
    \draw [thick] (1a3) -- (2a4); 
    \draw [thick] (1a3) -- (3b1); 
    
    \draw [thick] (1b) -- (2b2); 
    
    \draw [thick] (2a1) -- (3a3); 
    \draw [thick] (2a1) -- (3b7); 
    
    \draw [thick] (2a2) -- (2a1); 
    \draw [thick] (2a2) -- (3a2); 
    \draw [thick] (2a2) -- (3b4); 
    
    \draw [thick] (2a3) -- (3a2); 
    \draw [thick] (2a3) -- (3b4); 

    \draw [thick] (2a4) -- (2a3); 
    \draw [thick] (2a4) -- (3a1); 
    \draw [thick] (2a4) -- (3b2); 
      
    \draw [thick] (2a5) -- (3a2);   
    
    \draw [thick] (2a6) -- (3a3);   
    
    \draw [thick] (2b1) -- (3b6);   
    
    \draw [thick] (2b2) -- (3b4);   
    
    \draw [thick] (3a1) -- (3a2);   
    \draw [thick] (3a1) -- (3b3);   
    
    \draw [thick] (3a2) -- (3a3);   
    \draw [thick] (3a2) -- (3b5);   
    
    \draw [thick] (3a3) -- (3b8);   
    
    \draw [thick] (3b1) -- (3b2); 
      
    \draw [thick] (3b2) -- (3b3);   
    \draw [thick] (3b2) -- (3b4);   
    
    \draw [thick] (3b3) -- (3b5);   
    \draw [thick] (3b3) -- (3b6);   
    
    \draw [thick] (3b4) -- (3b5);   
    
    \draw [thick] (3b5) -- (3b7); 
    
    \draw [thick] (3b6) -- (3b8); 
      
    \draw [thick] (3b7) -- (3b8);   
    
    \draw [thick,dotted] (1b) -- (3b2);   
    \draw [thick,dotted] (1b) -- (3b3);   
    \draw [thick,dotted] (1b) -- (3b6);    

\end{tikzpicture}
\end{center}
\caption{Jordan nets in $\S^6$ and their degenerations. The dotted lines indicate that we do not know whether these degenerations exist, but we believe that they do not.}\label{fig:nets_n=6}
\end{figure}
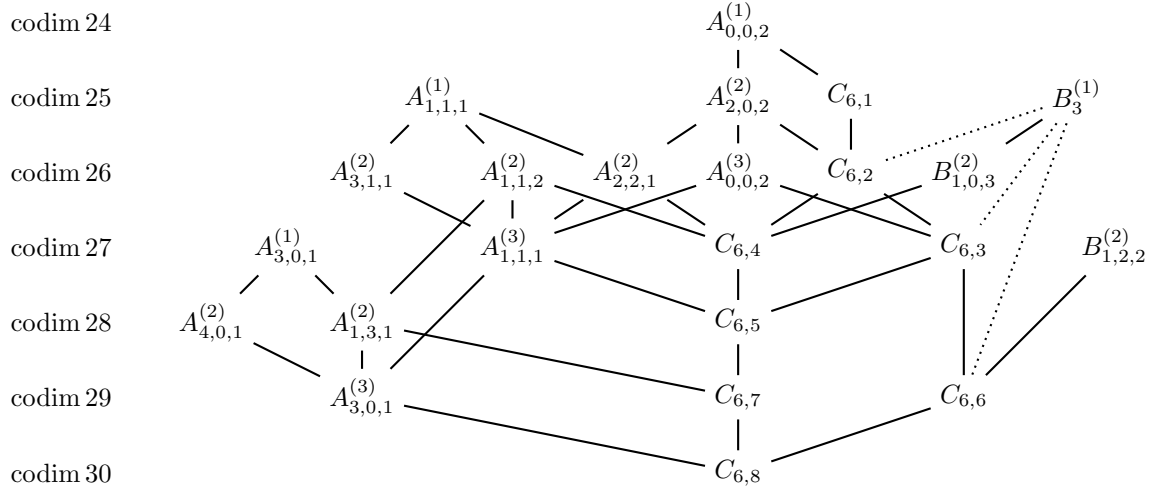

\section{Numerical results for bigger $n$ and Jordan webs}\label{sec:numerical_results}

This section is devoted to finding degenerations between embedded Jordan algebras in the sense of Definition~\ref{de:degeneration_jordan_space}. A basis $X_1,\ldots,X_m$ of an embedded Jordan algebra $\Ls\subseteq\S^n$ gives the $n\times n\times m$ tensor 
\begin{align*}
\mathbf{X}=[X_1|\dots|X_m]
\end{align*}
and, for $P_1,P_2\in\GL_n(\CC)$ and $Q\in\GL_m(\CC)$, we write $[\![\mathbf{X};P_1,P_2,Q]\!]$ for the result of acting on the rows, columns and layers of $\mathbf{X}$ by $P_1,P_2,Q$, respectively. Definition~\ref{de:degeneration_jordan_space} is then equivalent to 
\begin{align*}
\lim_{t \to 0} [\![\mathbf{X};P(t),P(t),Q(t)]\!]=\mathbf{Y}
\end{align*}
for some $P\in\GL_n(\CC[t^{\pm1}])$ and $Q\in\GL_m(\CC[t^{\pm1}])$ and a corresponding basis tensor $\mathbf{Y}$ for $\Ls'$.
Given a basis tensor $\mathbf{X}$ of $\Ls$, the basis tensors of $\Ls'$ in the orbit of $\Ls$ are $[\![\mathbf{X};P,P,Q]\!]$ for $P\in\GL_n(\CC)$ and $Q\in\GL_m(\CC)$, and if $\mathbf{Y}$ is a basis tensor of a degeneration of $\Ls$, then $\mathbf{Y}$ lies in the Zariski closure of the polynomial map
\begin{eqnarray*}
\CC^{n\times n}\times \CC^{m\times m}&\to& \CC^{n\times n\times m}\\
(P,Q)&\mapsto& [\![\mathbf{X};P,P,Q]\!].
\end{eqnarray*}
On the other hand, if $\mathbf{Y}$ does not lie in the Zariski closure, then the corresponding embedded Jordan algebra is no degeneration of $\Ls$ since the closures in the Euclidean and Zariski topologies coincide. 
A definite answer can be given by eliminating the variables $(P,Q)$ from the ideal $\mathcal{I}\subset \CC[P,Q,Z]$ generated by the equations $Z- [\![\mathbf{X};P,P,Q]\!]$: If $\mathbf{Y}\in \mathcal{V}(\mathcal{I}\cap\CC[Z])$, then $\mathbf{Y}$ is a basis tensor of a degeneration of $\Ls$. This is however only feasible for very small $n$ and $m$.\bigskip

Instead, we use gradient descent algorithms to find the distance 
\[
\inf_{\substack{P\in \CC^{n\times n}\\Q\in \CC^{m\times m}}} f(P,Q)=\inf_{\substack{P\in \CC^{n\times n}\\Q\in \CC^{m\times m}}}  \|\mathbf{Y}-[\![\mathbf{X};P,P,Q]\!]\|^2
\]
of the orbit 
\[
\{ [\![\mathbf{X};P,P,Q]\!]\mid P\in \CC^{n\times n}, \; Q\in \CC^{m\times m}\}
\]
of $\mathbf{X}$ to $\mathbf{Y}$. This way can find a sequence $(P_i,Q_i)$ such that 
$\lim_{i\to \infty}f(P_i,Q_i)=
0$
when $\mathbf{Y}$ is a basis of a degeneration of $\Ls$ with basis tensor $\mathbf{X}$.
We will say $\mathbf{X}$ degenerates to $\mathbf{Y}$ if $\inf_{P,Q} f(P,Q)=0$.
The function $f$ is a real smooth function in $2(n^2+m^2)$ real variables. We can either use gradient descent methods directly for $f$ or we can use gradient methods for 
\[
\inf_{P\in \CC^{n\times n}}  F(P)=\inf_{P\in \CC^{n\times n}}\left(\inf_{Q\in \CC^{n\times n}}f(P,Q)\right)=\inf_{P\in \CC^{n\times n}}\left(\min_{Q\in \CC^{n\times n}}f(P,Q)\right),
\]
where the infimum is attained since $f$ is quadratic and convex in $Q$. 
We observed that using gradient methods for $F$ instead of $f$ was more effective. This might be a result of having less variables to minimize over.
The value of $F$ at $P$ can be computed as the solution of a linear equation. If the minimizer $Q(P)$ is unique in a neighbourhood of $P$, we can compute the gradient of $F$ at $P$ via the chain rule using
$F(P)=f(P,Q(P))$ and we get $\nabla F(P)=\nabla_P f(P,Q)$ where $Q$ is the minimizer in the definition of $F$. If $F$ is not smooth at $P$, i.e., if the minimizer $Q$ is not unique, then its generalized gradient in the sense of~\cite[Definition 1.1]{C:generalizedgradients} is given by the convex hull of 
\[
\{\nabla_P f(P,Q)\mid\text{$Q$ is minimizer of $f(P,Q)$}\}.
\]

In practice we used the BFGS method to minimize $F$ in order to utilize second order information. A major obstacle is that both $f$ and $F$ are highly non-convex. Therefore there can be many local minima. Hence, to find degenerations we therefore used gradient descents with 50 randomly generated starting guesses $P_0\in \CC^{n\times n}$ to have higher odds of finding the global minimum of $F$.\bigskip

Another obstacle is that for degenerations only a minimizing sequence exists, i.e., the infimum value $F=0$ is not attained. It is therefore not entirely obvious at which value of $F$ we have found a degeneration. Note that infimum value for nondegenerations is not only dependent on the choice of Jordan algebra in the orbit but  even on the choice of basis tensor $\mathbf{Y}$ of $\Ls'$. To tackle one of these issues, we always use an orthonormal basis. We also use the following idea: let $\mathbf{X},\mathbf{Y},\mathbf{Z}$ be basis tensors and take
\[
f(P,Q)= \|\mathbf{Y}-[\![\mathbf{X};P,P,Q]\!]\|^2,
\quad g(P,Q)= \|\mathbf{Z}-[\![\mathbf{Y};P,P,Q]\!]\|^2,
 \quad h(P,Q)= \|\mathbf{Z}-[\![\mathbf{X};P,P,Q]\!]\|^2.
\]
If $\inf_{P,Q} f(P,Q)=0$, then 
\begin{align*}
\inf_{P,Q} h(P,Q)&=\inf_{P_1,P_2,Q_1,Q_2}   \|\mathbf{Z}-[\![\mathbf{Y};P_1,P_1,Q_1]\!]+[\![\mathbf{Y};P_1,P_1,Q_2]\!]-[\![\mathbf{X};P_2P_1,P_2P_1,Q_2Q_1]\!]\|^2\\
&\leq\inf_{P_1,P_2,Q_1,Q_2}   \left(\|\mathbf{Z}-[\![\mathbf{Y};P_1,P_1,Q_1]\!]\|+\|[\![\mathbf{Y};P_1,P_1,Q_1]\!]-[\![\mathbf{X};P_2P_1,P_2P_1,Q_2Q_1]\!]\|\right)^2\\
&=\inf_{P,Q} g(P,Q)
\end{align*}
where the last equality holds since $[\![\mathbf{Y};P_1,P_1,Q_1]\!]$ is also a basis tensor of a degeneration of $\mathbf{X}$ if $\mathbf{Y}$ is one. Now suppose that we know that $\mathbf{X}$ degenerates to $\mathbf{Y}$. Then the distance of the orbit of $\mathbf{X}$ to $\mathbf{Z}$ is not larger than the distance of the orbit of $\mathbf{Y}$ to $\mathbf{Z}$. So if $h(P_1,Q_1)\geq g(P_2,Q_2)+\varepsilon$ for some $\varepsilon>0$, we know that $h(P_1,Q_1)\geq  \inf_{P,Q} h(P,Q)+\varepsilon$ is not close to the minimum.\bigskip

As a first experiment we confirmed Theorem~\ref{thm:diagramNetsS5} numerically. For this, we used the orthonormal bases given by the Jordan algebras described in Theorems~\ref{thm:classification_nets_general} and~\ref{thm:classification_nets_J20}. The results are summarized in Table~\ref{table:resultsforS^5Nets} and confirm Theorem~\ref{thm:diagramNetsS5}.

\begin{table}[H]
\begin{tabular}{ |c|c|c|c|c|c|c|c|c|c|c|c| } 
\hline
& $A_{0,1,1}^{(1)} $ &$A_{2,0,1}^{(1)} $&$A_{1,0,2}^{(2)} $ & $A_{2,1,1}^{(2)} $ & $A_{3,0,1}^{(2)} $ &$A_{1,2,1}^{(2)} $ &$A_{0,1,1}^{(3)} $&$B_{1,1,2}^{(2)} $&$A_{2,0,1}^{(3)} $ &$C_{5,1} $ & $C_{5,2} $\\
\hline

orbit of $A_{0,1,1}^{(1)} $& - & 0.382 & 0.0 & 0.0 & 1.0 & 0.0 & 0.0 & 1.75 & 0.001 & 0.0 & 0.0 \\
orbit of $A_{2,0,1}^{(1)} $& 0.438 & - & 1.0 & 1.0 & 0.0 & 0.0 & 1.0 & 1.75 & 0.0 & 0.0 & 0.0 \\
orbit of $A_{1,1,2}^{(2)} $& 0.4 & 0.43 & - & 1.0 & 1.0 & 0.0 & 0.0 & 1.75 & 0.0 & 0.0 & 0.0 \\
orbit of $A_{2,1,1}^{(2)} $& 0.379 & 0.438 & 0.985 & - & 1.0 & 0.917 & 0.0 & 1.75 & 0.0 & 0.0 & 0.0 \\
orbit of $A_{3,0,1}^{(2)} $& 0.628 & 0.222 & 1.0 & 1.0 & - & 0.762 & 1.0 & 1.75 & 0.0 & 0.5 & 0.0 \\
orbit of $A_{1,2,1}^{(2)} $& 0.591 & 0.43 & 1.0 & 1.0 & 1.0 & - & 1.0 & 1.75 & 0.0 & 0.0 & 0.0 \\
orbit of $A_{0,1,1}^{(3)} $& 0.916 & 0.937 & 1.0 & 1.0 & 1.0 & 0.924 & - & 1.75 & 0.0 & 0.0 & 0.0 \\
orbit of $B_{1,1,2}^{(2)} $& 0.628 & 0.438 & 1.0 & 1.0 & 1.0 & 1.0 & 1.141 & - & 0.75 & 0.5 & 0.0 \\
orbit of $A_{2,0,1}^{(3)} $& 0.988 & 1.003 & 1.0 & 1.0 & 1.0 & 0.924 & 1.0 & 1.75 & - & 0.5 & 0.0 \\
orbit of $C_{5,1} $& 0.958 & 0.955 & 1.0 & 1.0 & 1.667 & 0.919 & 1.548 & 1.75 & 0.75 & - & 0.0 \\
orbit of $C_{5,2} $& 1.028 & 1.029 & 1.0 & 1.5 & 1.667 & 1.0 & 1.548 & 1.75 & 0.75 & 0.5 & - \\
\hline
\end{tabular}
\caption{Squared distances found between orbits of the Jordan nets in $\S^5$ via gradient descent.}\label{table:resultsforS^5Nets}
\end{table}

As a second experiment we confirmed all degenerations for Jordan nets in $\S^6$ in Figure~\ref{fig:nets_n=6}. For Conjecture~\ref{conj:B1nottoC6}, we have found further evidence. The smallest value of $F$ for basis tensors $\mathbf{X}$ and $\mathbf{Y}$ of the Jordan algebras~$B^{(1)}_3$ and $C_{6,6}$ was $1.0$ which suggests, that there is indeed no degeneration.\bigskip

For Jordan nets in $\S^7$ we found the Hasse diagram in Figure~\ref{fig:nets_n=7}. 

\begin{figure}[H]
\begin{center}
\begin{tikzpicture}

    \node () at (-1,-34) {$\codim 34$};
    \node () at (-1,-35) {$\codim 35$};
    \node () at (-1,-36) {$\codim 36$};
	\node () at (-1,-37.5) {$\codim 37$};
    \node () at (-1,-39) {$\codim 38$};
    \node () at (-1,-40) {$\codim 39$};
    \node () at (-1,-41) {$\codim 40$};
    \node () at (-1,-42) {$\codim 41$};
    \node () at (-1,-43) {$\codim 42$};
	
    \node (a1401) at (10,-40) {$A^{(1)}_{4,0,1}$};
    \node (a1211) at (7,-36) {$A^{(1)}_{2,1,1}$};
    \node (a1102) at (3.5,-34) {$A^{(1)}_{1,0,2}$};
    \node (a1021) at (7,-35) {$A^{(1)}_{0,2,1}$};
    
    \node (a251) at (11,-41) {$A^{(2)}_{5,0,1}$};
    \node (a241) at (10,-37.5) {$A^{(2)}_{4,1,1}$};
    \node (a232) at (5,-35) {$A^{(2)}_{3,0,2}$};
    \node (a231) at (5,-36) {$A^{(2)}_{3,2,1}$};
    \node (a222) at (2,-35) {$A^{(2)}_{2,1,2}$};
    \node (a221) at (3.5,-37.5) {$A^{(2)}_{2,3,1}$};
    \node (a213) at (9,-36) {$A^{(2)}_{1,0,3}$};
    \node (a212) at (6,-37.5) {$A^{(2)}_{1,2,2}$};
    \node (a211) at (9,-41) {$A^{(2)}_{1,4,1}$};

    \node (a3401) at (9,-42) {$A^{(3)}_{4,0,1}$};
    \node (a3211) at (6,-39) {$A^{(3)}_{2,1,1}$};
    \node (a3102) at (3.5,-36) {$A^{(3)}_{1,0,2}$};
    \node (a3021) at (8,-37.5) {$A^{(3)}_{0,2,1}$};
    
    \node (b221) at (1,-39) {$B^{(2)}_{1,3,2}$};
    \node (b231) at (1,-36) {$B^{(2)}_{1,1,3}$};

    \node (c1) at (2,-36) {$C_{7,1}$};
    \node (c2) at (2,-37.5) {$C_{7,2}$};
    \node (c3) at (2,-39) {$C_{7,3}$};
    \node (c4) at (4,-39) {$C_{7,4}$};
    \node (c5) at (4,-40) {$C_{7,5}$};
    \node (c6) at (3,-41) {$C_{7,6}$};
    \node (c7) at (6,-42) {$C_{7,7}$};
    \node (c8) at (6,-43) {$C_{7,8}$};

    \draw [thick] (a1401) -- (a211);
    \draw [thick] (a1401) -- (a251);
    
    \draw [thick] (a1211) -- (a212);
    \draw [thick] (a1211) -- (a221);
    \draw [thick] (a1211) -- (a241);
    
    \draw [thick] (a1102) -- (a222);
    \draw [thick] (a1102) -- (a232);
    
    \draw [thick] (a1021) -- (a213);
    \draw [thick] (a1021) -- (a231);
    
    \draw [thick] (a212) -- (a211);
    \draw [thick] (a213) -- (a212);
    \draw [thick] (a222) -- (a221);
    \draw [thick] (a232) -- (a231);

    \draw [thick] (a211) -- (a3401);
    \draw [thick] (a251) -- (a3401);
    
    \draw [thick] (a212) -- (a3211);
    \draw [thick] (a221) -- (a3211);
    \draw [thick] (a241) -- (a3211);
    
    \draw [thick] (a222) -- (a3102);
    \draw [thick] (a232) -- (a3102);
    
    \draw [thick] (a213) -- (a3021);
    \draw [thick] (a231) -- (a3021);
    
    \draw [thick] (a3102) -- (a3021);
    \draw [thick] (a3021) -- (a3211);
    \draw [thick] (a3211) -- (a3401);

    \draw [thick] (c1) -- (c2);
    \draw [thick] (c2) -- (c3);
    \draw [thick] (c2) -- (c4);
    \draw [thick] (c3) -- (c5);
    \draw [thick] (c3) -- (c6);
    \draw [thick] (c4) -- (c5);
    \draw [thick] (c5) -- (c7);
    \draw [thick] (c6) -- (c8);
    \draw [thick] (c7) -- (c8);
    
    \draw [thick] (a222) -- (c1);
    \draw [thick] (a3102) -- (c2);
    \draw [thick] (a221) -- (c4);
    \draw [thick] (a212) -- (c4);
    \draw [thick] (a3021) -- (c4);
    \draw [thick] (a3211) -- (c5);
    \draw [thick] (a211) -- (c7);
    \draw [thick] (a3401) -- (c8);

    \draw [thick] (b221) -- (c6);
    \draw [thick] (b231) -- (c2);

\end{tikzpicture}
\end{center}
\caption{Jordan nets in $\S^7$ and their degenerations obtained numerically.}\label{fig:nets_n=7}
\end{figure}
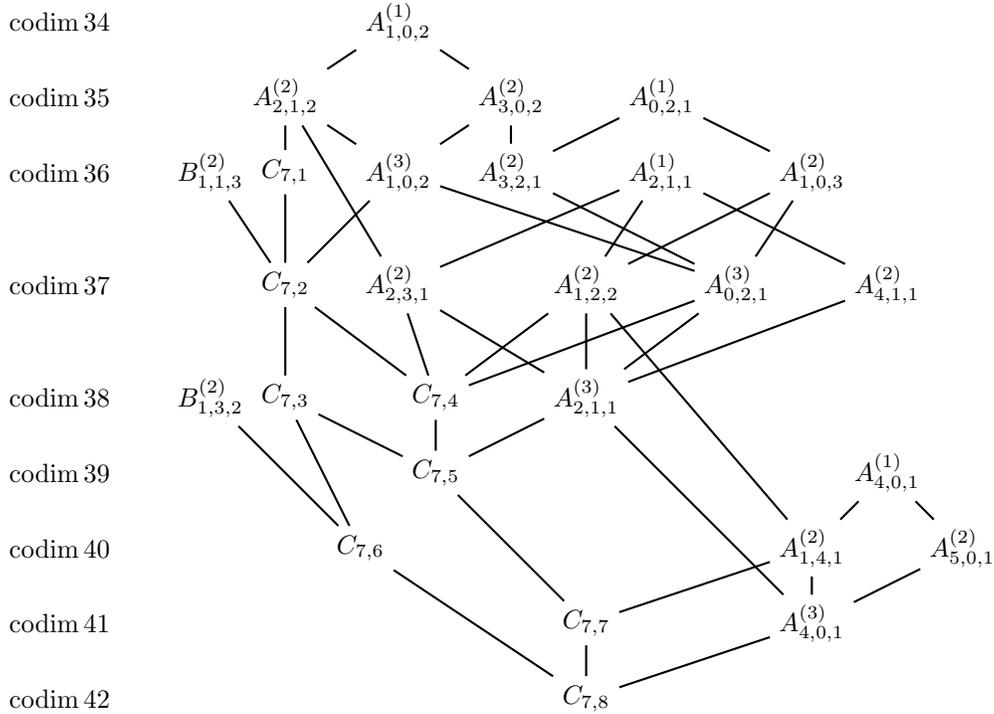

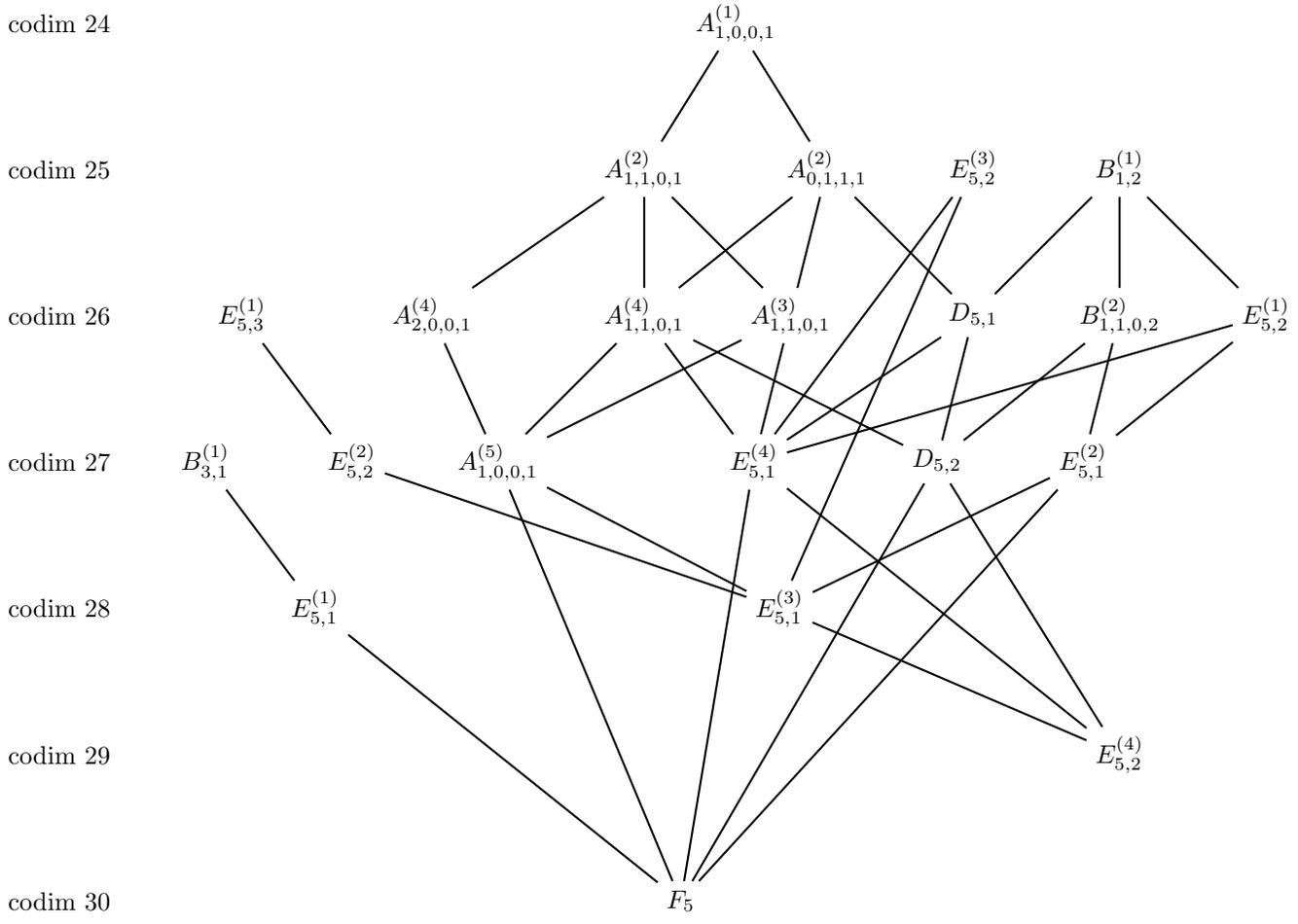
\begin{figure}[H]
\vspace{70pt}
\begin{center}
\begin{tikzpicture}[rotate=90,transform shape]

	\node () at (-2,-0) {codim $24$};
	\node () at (-2,-2) {codim $25$};
	\node () at (-2,-4) {codim $26$};
	\node () at (-2,-6) {codim $27$};
	\node () at (-2,-8) {codim $28$};
	\node () at (-2,-10) {codim $29$};	
	\node () at (-2,-12) {codim $30$};	

    \node (1) at (7.25,-0) {$A^{(1)}_{1,0,0,1}$};
    \node (2) at (6,-2) {$A^{(2)}_{1,1,0,1}$};
    \node (3) at (8.5,-2) {$A^{(2)}_{0,1,1,1}$};
    \node (4) at (8.,-4) {$A^{(3)}_{1,1,0,1}$};
    \node (5) at (3.1,-4) {$A^{(4)}_{2,0,0,1}$};   
    \node (6) at (6,-4) {$A^{(4)}_{1,1,0,1}$};
    \node (7) at (4,-6) {$A^{(5)}_{1,0,0,1}$};
    \node (8) at (0,-6) {$B^{(1)}_{3,1}$};
    \node (9) at (12.5,-2) {$B^{(1)}_{1,2}$};
    \node (10) at (12.5,-4) {$B^{(2)}_{1,1,0,2}$};
    \node (11) at (10.5,-4) {$D_{5,1}$};
    \node (12) at (10,-6) {$D_{5,2}$};
    \node (13) at (1.5,-8) {$E^{(1)}_{5,1}$};
    \node (14) at (14.5,-4) {$E^{(1)}_{5,2}$};
    \node (15) at (.5,-4) {$E^{(1)}_{5,3}$};    
    \node (16) at (12,-6) {$E^{(2)}_{5,1}$};
    \node (17) at (2,-6) {$E^{(2)}_{5,2}$};
    \node (18) at (7.85,-8) {$E^{(3)}_{5,1}$};
	\node (19) at (10.5,-2) {$E^{(3)}_{5,2}$};
    \node (20) at (7.5,-6) {$E^{(4)}_{5,1}$};
	\node (21) at (12.5,-10) {$E^{(4)}_{5,2}$};
    \node (23) at (6.5,-12) {$F_5$};
    
    \draw [thick] (1) -- (2);
    \draw [thick] (1) -- (3);
    \draw [thick] (2) -- (4);
    \draw [thick] (2) -- (5);
    \draw [thick] (2) -- (6);
    \draw [thick] (3) -- (4);
    \draw [thick] (3) -- (6);
    \draw [thick] (3) -- (11);    
    \draw [thick] (4) -- (7);   
    \draw [thick] (4) -- (20);
    \draw [thick] (5) -- (7);    
    \draw [thick] (6) -- (7);
    \draw [thick] (6) -- (12);
    \draw [thick] (6) -- (20);
    \draw [thick] (7) -- (18);
    \draw [thick] (7) -- (23);
    \draw [thick] (8) -- (13);
    \draw [thick] (9) -- (10);
    \draw [thick] (9) -- (11);   
    \draw [thick] (9) -- (14);
    \draw [thick] (10) -- (16);                  
 	\draw [thick] (10) -- (12);        
    \draw [thick] (11) -- (12);
    \draw [thick] (11) -- (20);
    \draw [thick] (12) -- (23);
    \draw [thick] (12) -- (21);
    \draw [thick] (13) -- (23);
    \draw [thick] (14) -- (16);  
    \draw [thick] (14) -- (20);
    \draw [thick] (15) -- (17);
    \draw [thick] (16) -- (18);
    \draw [thick] (16) -- (23);
    \draw [thick] (17) -- (18);    
    \draw [thick] (18) -- (21);
    \draw [thick] (19) -- (20);

    \draw [thick] (19) -- (18);

        \draw [thick] (20) -- (21);
      \draw [thick] (20) -- (23);

\end{tikzpicture}
\end{center}
\caption{Jordan webs in $\S^5$ and their degenerations obtained numerically. 
}\label{fig:webs_n=5}
\end{figure}

For Jordan webs in $\S^4$ and $\S^5$ we found the diagrams in Figures~\ref{fig:webs_n=4} and~\ref{fig:webs_n=5}.

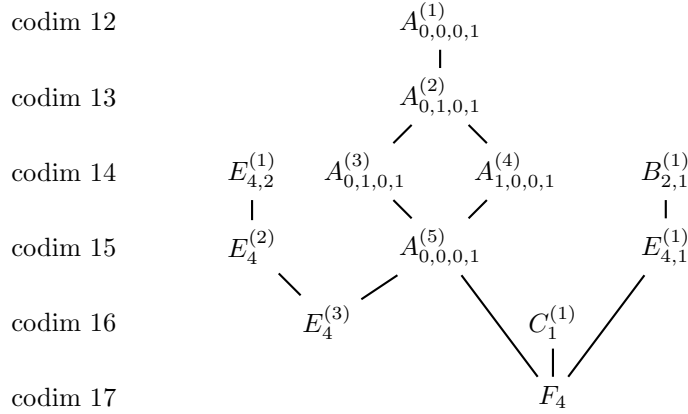
\begin{figure}[H]
\begin{center}
\begin{tikzpicture}

	\node () at (-1,-12) {codim $12$};
	\node () at (-1,-13) {codim $13$};
	\node () at (-1,-14) {codim $14$};
	\node () at (-1,-15) {codim $15$};
	\node () at (-1,-16) {codim $16$};
	\node () at (-1,-17) {codim $17$};	
		
    \node (1) at (4,-12) {$A^{(1)}_{0,0,0,1}$};
    \node (2) at (4,-13) {$A^{(2)}_{0,1,0,1}$};
    \node (3) at (3,-14) {$A^{(3)}_{0,1,0,1}$};
    \node (4) at (5,-14) {$A^{(4)}_{1,0,0,1}$};
    \node (5) at (4,-15) {$A^{(5)}_{0,0,0,1}$};
    \node (6) at (7,-14) {$B^{(1)}_{2,1}$};
    \node (7) at (5.5,-16) {$C^{(1)}_1$};
    \node (8) at (7,-15) {$E^{(1)}_{4,1}$};
    \node (9) at (1.5,-14) {$E^{(1)}_{4,2}$};
    \node (10) at (1.5,-15) {$E^{(2)}_4$};
    \node (11) at (2.5,-16) {$E^{(3)}_4$};
    \node (12) at (5.5,-17) {$F_4$};

    \draw [thick] (1) -- (2);
    \draw [thick] (2) -- (3);
    \draw [thick] (2) -- (4);
    \draw [thick] (3) -- (5);
    \draw [thick] (6) -- (8);
	\draw [thick] (8) -- (12);
    \draw [thick] (5) -- (12);
    \draw [thick] (4) -- (5);
	\draw [thick] (7) -- (12);
    \draw [thick] (5) -- (11);
    \draw [thick] (9) -- (10);
    \draw [thick] (10) -- (11);

\end{tikzpicture}
\end{center}
\caption{Jordan webs in $\S^4$ and their degenerations obtained numerically.}\label{fig:webs_n=4}
\end{figure}

\section{Obstructions to degenerations between Jordan spaces}\label{sec:degenerations_obstructions}

The goal of this section is to make a list of obstructions to the existence of degenerations $\Ls\to\Ls'$.

\subsection{The abstract obstruction}
Suppose that $\Ls,\Ls'$ are embeddings of Jordan algebras $\As,\As'$. 

\begin{prop}\label{prop:obstr_abstract}
If $\Ls\to\Ls'$, then also $\As\to\As'$.
\end{prop}
\begin{proof}
Suppose $\Ls\to\Ls'$. Let $P\in\GL_n(\CC[t^{\pm1}])$ and $Q\in\GL_m(\CC[t^{\pm1}])$ be matrices such that
\[
(Y_1,\ldots,Y_m):= (PX_1P^\top,\ldots,PX_mP^\top)Q
\]
is a tuple of matrices with coefficients in $\CC[t]$ and $X'_i=\lim_{t\to0}Y_i$ for all $i\in\{1,\ldots,m\}$. Then it follows that $\As'=\lim_{t\to0}Q\cdot \As$ and hence $\As\to\As'$.
\end{proof}

\subsection{The determinantal obstruction}

\begin{de}
The {\em determinant} of $\Ls$ is defined as 
\[
\det(\Ls):=\det(x_1X_1+\ldots+x_mX_m)\in\CC[x_1,\ldots,x_m]_n.
\]
When $(Y_1,\ldots,Y_m)=(X_1,\ldots,X_m)Q$ for some $Q\in\GL_m(\CC)$, then 
\[
y_1Y_1+\ldots+y_mY_m = x_1X_1+\ldots+x_mX_m
\] 
for $(x_1,\ldots,x_m)=(y_1,\ldots,y_m)Q^{-\top}$. Hence $\det(\Ls)$ is well-defined up to coordinate-change.
\end{de}

Since $\Ls$ is regular, its determinant is not the zero polynomial. Note that congruent subspaces have the same determinant up to scaling.

\begin{de}
Let $f,g\in\CC[x_1,\ldots,x_m]_n$ be forms. We say that $f$ degenerates to $g$, denoted as $f\to g$, when $g\in\overline{\GL_n(\CC)\cdot f}$.
\end{de}

\begin{prop}
If $\Ls\to\Ls'$, then $\det(\Ls)\to\det(\Ls')$.
\end{prop}
\begin{proof}
Suppose $\Ls\to\Ls'$. Let $P\in\GL_n(\CC[t^{\pm1}])$ and $Q\in\GL_m(\CC[t^{\pm1}])$ be matrices such that
\[
(Y_1,\ldots,Y_m):= (PX_1P^\top,\ldots,PX_mP^\top)Q
\]
is a tuple of matrices with coefficients in $\CC[t]$ and $X'_i=\lim_{t\to0}Y_i$ for all $i\in\{1,\ldots,m\}$. Then 
\[
\det(\Ls')=\det(x_1X_1'+\ldots+x_mX_m')=\lim_{t\to0}\det(x_1Y_1+\ldots+x_mY_m)=\lim_{t\to0}\det(P\Ls P^\top)
\]
and hence $\det(\Ls)\to\det(\Ls')$.
\end{proof}

\subsection{The rank-minimal subspace obstruction}

\begin{de}
We define the {\em minimal matrix rank} of $\Ls$ to be $\tau_1(\Ls):=\min\{\rk(X)\mid X\in\Ls\setminus\{{\bf 0}_n\}\}$. More generally, for $1\leq k\leq m$, we define 
\[
\tau_k(\Ls):=\min\{\max\{\rk(X)\mid X\in \Ps\setminus\{{\bf 0}_n\}\}\mid \Ps\in\Gr(k,\Ls)\}
\]
to be the minimal upperbound on the rank of nonzero elements of a $k$-dimensional subspace of $\Ls$.
\end{de}

When $k=2,3,4$, we call $\tau_k(\Ls)$ the {\em minimal pencil/net/web rank} of $\Ls$.

\begin{prop}
If $\Ls\to\Ls'$, then $\tau_k(\Ls')\leq\tau_k(\Ls)$ for all $1\leq k\leq m$.
\end{prop}
\begin{proof}
This holds since the set
\[
\{(\Ls,\Ps)\in\Gr(m,\S^n)\times\Gr(k,\S^n)\mid \Ps\subseteq\Ls, \forall X\in\Ps:\rk(X)\leq \ell\}
\]
is closed for all $\ell\geq0$.
\end{proof}

\subsection{The Segre symbol obstruction}
Let $A$ be an $n\times n$ matrix in Jordan normal form. The {\em Segre symbol} corresponding to $A$ is a multiset of partitions 
\[
\sigma=\{(k^{(1)}_1,\ldots,k^{(1)}_{n_1}),\ldots,(k^{(\ell)}_1,\ldots,k^{(\ell)}_{n_\ell}\}
\]
where $A$ has $\ell$ distinct eigenvalues $\lambda_1,\ldots,\lambda_\ell$ and has Jordan blocks of sizes $k^{(i)}_1,\ldots,k^{(i)}_{n_i}$ corresponding to $\lambda_i$. In particular, we have $\sum_{i,j}k^{(i)}_{j}=n$. Denote by $\geq$ the partial order on the set of Segre symbols of $n\times n$ matrices in Jordan normal form generated by
\[
\{\tau_1,\ldots,\tau_{\ell-1},\tau_\ell,\tau_{\ell+1}\}\geq \{\tau_1,\ldots,\tau_{\ell-1},\tau_\ell+\tau_{\ell+1}\}
\]
and 
\[
\{\tau_1,\ldots,\tau_{\ell},\mu\}\geq \{\tau_1,\ldots,\tau_\ell,\mu'\} \mbox{ when } \mu\triangleright\mu'
\]
where $\triangleright$ is the dominance order on partitions and the sum $(k_1,\ldots,k_n)+(\ell_1,\ldots,\ell_m)$ is defined to be 
\[
(k_1+\ell_1,\ldots,k_{\max(n,m)}+\ell_{\max(n,m)})
\]
with $k_i=0$ for $i>n$ and $\ell_j=0$ for $j>m$. See \cite{FMS:pencil_old_new} for the history of the term Segre symbol.

\begin{prop}
Let $\Ls\subseteq\S^n$ be a Jordan space and let $U,V\in\Ls$ be invertible matrices. Then the following statements hold:
\begin{itemize}
\item[(1)] The sets of Segre symbols corresponding to the Jordan normal forms of matrices in $\Ls U^{-1}$ and in $\Ls V^{-1}$ are equal.
\item[(2)] The set of Segre symbols corresponding to the Jordan normal form of a matrix in $\Ls U^{-1}$ has a unique maximal element.
\end{itemize}

\end{prop}
\begin{proof}
(1)
Since the Jordan normal form of a matrix and its transpose are the same, it suffices to prove that $\Ls V^{-1}$ is similar to $(\Ls U^{-1})^\top= U^{-1}\Ls$. By \cite[Lemma 2.3]{BES:JordanSpaces}, we have $WU^{-1}W=V$ for some matrix $W\in\Ls$. And by Theorem~\ref{thm:equiv_conditions}(c) applied to $\Ls^{-1}$ and afterwards to $\Ls$, we have
\[
\Ls=(\Ls^{-1})^{-1}=V(\Ls^{-1})V=VW^{-1}\Ls W^{-1} V.
\]
Hence $\Ls V^{-1}=VW^{-1}\Ls W^{-1}$ is indeed similar to $W^{-1}VW^{-1}\Ls=U^{-1}\Ls$.

(2)
Suppose that $\sigma_1,\ldots,\sigma_k$ are the maximal Segre symbols corresponding to the Jordan normal form of a matrix in $\Ls U^{-1}$. Then we see that
\[
\Ls U^{-1}=\bigcup_{i=1}^k\{ X\in\Ls U^{-1}\mid \mbox{the Segre symbol of $X$ is at most $\sigma_i$}\}
\]
Since $\Ls U^{-1}$ is subspace of $\CC^{n\times n}$, it is irreducible. Since each of the sets on the right hand side is closed, $k$ must be equal to $1$.
\end{proof}

\begin{de}
We define the {\em Segre symbol} $\sigma(\Ls)$ of a Jordan space $\Ls\subseteq\S^n$ to be the maximal Segre symbol corresponding to the Jordan normal form of a matrix in $\Ls U^{-1}$ for any invertible matrix $U\in\Ls$.
\end{de}

\begin{re}
The set of Segre symbols corresponding to Jordan normal form of a matrix in $\Ls U^{-1}$ depends on $U$ for general linear spaces of symmetric matrices. For example, let 
\[
\Ls=\left\{\begin{pmatrix}
x &z &y\\
z&x+y&0\\
y&0&x
\end{pmatrix}\right\}.
\]
Then $\Ls$ does not contain nilpotent matrices. However ${\bf J}_3\in\Ls$ and 
\[
\left\{\begin{pmatrix}
0&1&0\\
0&0&1\\
0&0&0
\end{pmatrix}\right\}\in \Ls{\bf J}_3
\]
is nilpotent. Hence, the sets of Segre symbols are not the same.
\end{re}

\begin{que}
Is the Segre symbol of a general linear space of symmetric matrices well-defined?
\end{que}

\begin{prop}
Congruent Jordan spaces have the same Segre symbol.
\end{prop}
\begin{proof}
Let $\Ls\subseteq\S^n$ be a Jordan space and $U\in\Ls,P\in\GL_n$ invertible matrices. Then $\Ls U^{-1}$ is similar to $P\Ls U^{-1}P^{-1}=(P\Ls P^\top)(PUP^\top)^{-1}$ and hence $\sigma(\Ls)=\sigma(P\Ls P^\top)$.
\end{proof}

\begin{prop}
If $\Ls\to\Ls'$, then $\sigma(\Ls)\geq\sigma(\Ls')$.
\end{prop}
\begin{proof}
Suppose that $\Ls\to\Ls'$. Let $U'\in\Ls'$ an invertible matrix and $X'\in\Ls'$ any other matrix. Then $(U',X')$ is the limit of a sequence of pairs $(U_n,X_n)$ of matrices contained in a Jordan space $\Ls_n$ congruent to $\Ls$. By replacing the sequence by a subsequence, we may assume that $U_n$ is invertible. Since $XU^{-1}=\lim_{n\to\infty}X_nU_n^{-1}$ and the Segre symbol corresponding to the Jordan normal form of $X_nU_n^{-1}$ is at most $\sigma(\Ls)$, we see that the same holds for the Segre symbol corresponding to the Jordan normal form of $XU^{-1}$. Hence $\sigma(\Ls')\leq\sigma(\Ls)$.
\end{proof}

\subsection{The orbit dimension obstruction}

\begin{de}
We define the {\em orbit dimension} $d(\Ls)$ of $\Ls$ to be the dimension of the irreducible variety $\overline{\{(Y_1,\ldots,Y_m)\mid \spann(Y_1,\ldots,Y_m)\in\GL_n\cdot \Ls\}}$.
\end{de}

\begin{prop}
If $\Ls\to\Ls'$, then $d(\Ls)\geq d(\Ls')$ with equality if and only if $\Ls,\Ls'$ are congruent.
\end{prop}
\begin{proof}
Assume that $\Ls\to\Ls'$. Then 
\[
\overline{\{(Y_1,\ldots,Y_m)\mid \spann(Y_1,\ldots,Y_m)\in\GL_n\cdot \Ls\}}\supseteq\overline{\{(Y_1,\ldots,Y_m)\mid \spann(Y_1,\ldots,Y_m)\in\GL_n\cdot \Ls'\}}
\]
and hence $d(\Ls)\geq d(\Ls')$. Since both varieties are irreducible, equality of the dimensions implies that $\Ls,\Ls'$ are congruent.
\end{proof}

\section{Families of degenerations of Jordan nets}\label{sec:family_degenerations_Jordan_nets}

Fix an integer $n\geq 2$. Theorem~\ref{thm:classification_nets_general} classifies all embeddings of the Jordan algebras $\CC\times\CC\times\CC$, $\CC\times\Js^1_0$, $\CC[x]/(x^3)$, $\Js^2_2$ and $\Js^2_1$ into $\S^n$. The goal of this section is to determine all degenerations between these embeddings. The abstract obstruction shows that there are no degenerations between the $A$'s and the $B$'s. So we can handle them separately. For the construction of the degenerations, see Appendix~\ref{ssec:degen_nets}. We now prove that these degeneration generate everything.

\begin{prop}
The degenerations\vspace{-10pt}
\begin{multicols}{2}
\begin{enumerate}[(a)]
\item[$\mathrm{(a)}$] $A^{(1)}_{k_1,k_2,k_3}\to A^{(2)}_{k_3,k_1,k_2+k_3}$
\item[$\mathrm{(b)}$] $A^{(1)}_{k_1,k_2,k_3}\to A^{(2)}_{k_2+k_3,k_1+k_2,k_3}$
\item[$\mathrm{(c)}$] $A^{(1)}_{k_1,k_2,k_3}\to A^{(2)}_{k_1+k_2+k_3,k_2,k_3}$
\item[$\mathrm{(d)}$] $A^{(2)}_{r,k_1,k_2}\to A^{(2)}_{r,k_1+2,k_2-1}$ for $k_2>1$
\item[$\mathrm{(e)}$] $A^{(2)}_{r,k_1,k_2}\to A^{(3)}_{k_1,k_2-r,r}$ for $r\leq k_2$
\item[$\mathrm{(f)}$] $A^{(2)}_{r,k_1,k_2}\to A^{(3)}_{k_1+k_2-r,r-k_2,k_2}$ for $k_2\leq r\leq k_1+k_2$
\item[$\mathrm{(g)}$] $A^{(2)}_{r,k_1,k_2}\to A^{(3)}_{r-(k_1+k_2),k_1,k_2}$ for $r\geq k_1+k_2$
\item[$\mathrm{(h)}$] $A^{(3)}_{k_1,k_2,k_3}\to A^{(3)}_{k_1+1,k_2+1,k_3-1}$ for $k_3>1$
\item[$\mathrm{(i)}$] $A^{(3)}_{k_1,k_2,k_3}\to A^{(3)}_{k_1+2,k_2-1,k_3}$ for $k_2>0$
\item[$\mathrm{(j)}$] $A^{(3)}_{k_1,k_2,k_3}\to A^{(3)}_{k_1-1,k_2+2,k_3-1}$ for $k_1>0,k_3>1$
\end{enumerate}
\end{multicols}\vspace{-10pt}
generate all degenerations between Jordan nets labelled with an $A$.
\end{prop}
\begin{proof}
The following table shows for each orbit its determinant, minimal rank, minimal pencil rank and Segre symbol. 
\[
\begin{array}{l|c|c|c|c|}
&\det&\tau_1&\tau_2&\sigma\\\hline
A^{(1)}_{k_1,k_2,k_3}&x^{k_1+k_2+k_3}y^{k_2+k_3}z^{k_3}&k_3&k_2+2k_3&(\overbrace{1\cdots\cdot\cdot1}^{k_1+k_2+k_3})(\overbrace{1\cdots\cdot1}^{k_2+k_3})(\overbrace{1\cdots1}^{k_3})\\\hline
A^{(2)}_{r,k_1,k_2}&x^ry^{k_1+2k_2}&\min(r,k_2)&\min(r+k_2,k_1+2k_2)&(\overbrace{1\cdots1}^{r})(\overbrace{2\cdots2}^{k_2}\overbrace{1\cdots1}^{k_1})\\\hline
A^{(3)}_{k_1,k_2,k_3}&x^n&k_3&k_2+2k_3&(\overbrace{3\cdots3}^{k_3}\overbrace{2\cdots2}^{k_2}\overbrace{1\cdots1}^{k_1})\\\hline
\end{array}
\]
The abstract obstruction shows that we have the following 6 possible cases:

(1)
Suppose that $A^{(3)}_{k_1,k_2,k_3}\to A^{(3)}_{\ell_1,\ell_2,\ell_3}$. Then the Segre symbol obstruction shows that 
\[
(\overbrace{3\cdots3}^{\ell_3}\overbrace{2\cdots2}^{\ell_2}\overbrace{1\cdots1}^{\ell_1})\leq
(\overbrace{3\cdots3}^{k_3}\overbrace{2\cdots2}^{k_2}\overbrace{1\cdots1}^{k_1})
\]
This means that we can reach $(\ell_1,\ell_2,\ell_3)$ from $(k_1,k_2,k_3)$ by a series of moves where in each step we replace $(k_1,k_2,k_3)$ by either $(k_1+1,k_2+1,k_3-1)$, $(k_1+2,k_2-1,k_3)$ or $(k_1-1,k_2+2,k_3-1)$. Hence $A^{(3)}_{k_1,k_2,k_3}\to A^{(3)}_{\ell_1,\ell_2,\ell_3}$ is obtained as a composition of degenerations from (h),(i),(j).

(2)
Suppose that $A^{(2)}_{r,k_1,k_2}\to A^{(3)}_{\ell_1,\ell_2,\ell_3}$. The Segre symbol obstruction shows that
\[
(\overbrace{3\cdots3}^{\ell_3}\overbrace{2\cdots2}^{\ell_2}\overbrace{1\cdots1}^{\ell_1})\leq
\left\{\begin{array}{cl}
(\overbrace{3\cdots3}^{r}\overbrace{2\cdots2}^{k_2-r}\overbrace{1\cdots1}^{k_1})&\mbox{if $r\leq k_2$}\\
(\overbrace{3\cdots3}^{k_2}\overbrace{2\cdots2}^{r-k_2}\overbrace{1\cdots\cdot1}^{k_1+k_2-r})&\mbox{if $k_2\leq r\leq k_1+k_2$}\\
(\overbrace{3\cdots3}^{k_2}\overbrace{2\cdots2}^{k_1}\overbrace{1\cdots\cdots1}^{r-(k_1+k_2)})&\mbox{if $r\geq k_1+k_2$}
\end{array}\right.
\]
and hence we have:
\begin{itemize}
\item $A^{(3)}_{k_1,k_2-r,r}\to A^{(3)}_{\ell_1,\ell_2,\ell_3}$ when $r\leq k_2$;
\item $A^{(3)}_{k_1+k_2-r,r-k_2,k_2}\to A^{(3)}_{\ell_1,\ell_2,\ell_3}$ when $k_2\leq r\leq k_1+k_2$; and
\item $A^{(3)}_{r-(k_1+k_2),k_1,k_2}\to A^{(3)}_{\ell_1,\ell_2,\ell_3}$ when $r\geq k_1+k_2$
\end{itemize}
by (1). So $A^{(2)}_{r,k_1,k_2}\to A^{(3)}_{\ell_1,\ell_2,\ell_3}$ is obtained as a composition of degenerations from (1),(e),(f),(g).

(3)
Suppose that $A^{(2)}_{r,k_1,k_2}\to A^{(2)}_{r',\ell_1,\ell_2}$. Then the determinantal obstruction shows that $\{r,k_1+2k_2\}=\{r',\ell_1+2\ell_2\}$. Now, the Segre symbol obstruction shows that $r=r'$ and $\ell_2\leq k_2$. Hence the degeneration is from (d).

(4)
Suppose that $A^{(1)}_{k_1,k_2,k_3}\to A^{(3)}_{\ell_1,\ell_2,\ell_3}$. Then the Segre symbol obstruction shows that
\[
(\overbrace{3\cdots3}^{\ell_3}\overbrace{2\cdots2}^{\ell_2}\overbrace{1\cdots1}^{\ell_1})\leq
(\overbrace{3\cdots3}^{k_3}\overbrace{2\cdots2}^{k_2}\overbrace{1\cdots1}^{k_1})
\]
and hence $A^{(3)}_{k_1,k_2,k_3}\to A^{(3)}_{\ell_1,\ell_2,\ell_3}$. We have $A^{(1)}_{k_1,k_2,k_3}\to A^{(3)}_{k_1,k_2,k_3}$ using for example (a),(e). So using~(1), we are done.

(5)
Suppose that $A^{(1)}_{k_1,k_2,k_3}\to A^{(2)}_{r,\ell_1,\ell_2}$. Then the Segre symbol obstruction shows that
\[
(\overbrace{1\cdots1}^{r})(\overbrace{2\cdots2}^{\ell_2}\overbrace{1\cdots1}^{\ell_1})\leq (\overbrace{1\cdots\cdot\cdot1}^{k_1+k_2+k_3})(\overbrace{1\cdots\cdot1}^{k_2+k_3})(\overbrace{1\cdots1}^{k_3}) 
\]
and hence $r\in \{k_1+k_2+k_3,k_2+k_3,k_3\}$. When $r=k_3$, we see that
\[
(\overbrace{2\cdots2}^{\ell_2}\overbrace{1\cdots1}^{\ell_1})\leq (\overbrace{1\cdots\cdot\cdot1}^{k_1+k_2+k_3})(\overbrace{1\cdots\cdot1}^{k_2+k_3})
\]
which implies that $\ell_2\leq k_2+k_3$ and so $A^{(2)}_{k_3,k_1,k_2+k_3}\to A^{(2)}_{r,\ell_1,\ell_2}$ using (a),(d). When $r>k_3$, the minimal rank obstruction shows that $\ell_2\leq k_3$ and so $A^{(2)}_{r,n-r-2k_3,k_3}\to A^{(2)}_{r,\ell_1,\ell_2}$ using (b),(d) or (c),(d). 

(6) Suppose that $A^{(1)}_{k_1,k_2,k_3}\to A^{(1)}_{\ell_1,\ell_2,\ell_3}$. Then the determinant obstruction shows that 
\[
x^{k_1+k_2+k_3}y^{k_2+k_3}z^{k_3}\to x^{\ell_1+\ell_2+\ell_3}y^{\ell_2+\ell_3}z^{\ell_3}
\]
which implies that $(k_1,k_2,k_3)=(\ell_1,\ell_2,\ell_3)$.
\end{proof}

\begin{prop}
The degenerations
\begin{itemize}
\item[$\mathrm{(a)}$] $B^{(1)}_{n/2}\to B^{(2)}_{k,0,n/2}$ for $1\leq k\leq n/4$ if $2\mid n$
\item[$\mathrm{(b)}$] $B^{(2)}_{k,\ell_1,\ell_2}\to B^{(2)}_{k-1,\ell_1,\ell_2}$ for $k>1$ 
\end{itemize}
generate all degenerations between Jordan nets labelled with an $B$.
\end{prop}
\begin{proof}
The following table shows for each orbit its determinant, minimal rank and minimal pencil rank. 
\[
\begin{array}{l|c|c|c|}
&\det&\tau_1&\tau_2\\\hline
B^{(1)}_{n/2}&(xy-z^2)^{n/2}&n/2&n\\\hline
B^{(2)}_{k,\ell_1,\ell_2}&x^{\ell_2}y^{\ell_1+\ell_2}&2k&\ell_2\\\hline
\end{array}
\]
The abstract obstruction shows that we have the following 2 possible cases:

(1)
Suppose that $B^{(2)}_{k,\ell_1,\ell_2}\to B^{(2)}_{k',\ell_1',\ell_2'}$. Then the determinantal obstruction shows that $(\ell_1,\ell_2)=(\ell_1',\ell_2')$ and the minimum rank condition shows that $k'\leq k$. So the degenerations is from (b).

(2)
Suppose that $B^{(1)}_{n/2}\to B^{(2)}_{k,\ell_1,\ell_2}$. Then the determinantal obstruction shows that $(\ell_1,\ell_2)=(0,n/2)$ and the minimum rank obstruction shows that $k\leq n/4$. So the degenerations is from (a).
\end{proof}

We see that the closures of the orbits $A^{(1)}_{k_1,k_2,k_3}$ and $B^{(1)}_{n/2}$ are always components of the Jordan locus. We compute their codimensions in $\Gr(3,\S^n)$.

\begin{prop}
Write $(n_1,n_2,n_3)=(k_1+k_2+k_3,k_2+k_3,k_3)$. Then $A^{(1)}_{k_1,k_2,k_3}$ is invariant under
\[
\left\{\begin{pmatrix}\lambda_1Q_1\\&\lambda_2Q_2\\&&\lambda_3Q_3\end{pmatrix}\,\middle|\,\begin{array}{l}Q_1\in O(n_1),\lambda_1\in\CC^*,\\Q_2\in O(n_2),\lambda_2\in\CC^*,\\Q_3\in O(n_3),\lambda_3\in\CC^*\end{array}
\right\}
\]
together with
\[
P_1:=\begin{pmatrix}&{\bf 1}_{n_1}\\{\bf 1}_{n_1}\\&&{\bf 1}_{n_3}\end{pmatrix}\mbox{ and }P_2:=\begin{pmatrix}{\bf 1}_{n_1}\\&&{\bf 1}_{n_2}\\&{\bf 1}_{n_2}\end{pmatrix}
\]
when $n_1=n_2$ and when $n_2=n_3$, respectively. These matrices generate the stabilizer of $A^{(1)}_{k_1,k_2,k_3}$. The orbit of $A^{(1)}_{k_1,k_2,k_3}$ has codimension $n_1^2 + n_1n_2 + n_2^2 + n_1n_3 + n_2n_3 + n_3^2 + n_1 + n_2 + n_3 - 6$.
\end{prop}
\begin{proof}
Let 
\[
\begin{array}{rrr}
A\in\CC^{n_1\times n_1},&B\in\CC^{n_1\times n_2},&C\in\CC^{n_1\times n_3},\\
D\in\CC^{n_2\times n_1},&E\in\CC^{n_2\times n_2},&F\in\CC^{n_2\times n_3},\\
G\in\CC^{n_3\times n_1},&H\in\CC^{n_3\times n_2},&I\in\CC^{n_3\times n_3}\phantom{,}
\end{array}
\]
be matrices such that
\[
\begin{pmatrix}
A&B&C\\D&E&F\\G&H&I
\end{pmatrix}A^{(1)}_{k_1,k_2,k_3}\begin{pmatrix}
A&B&C\\D&E&F\\G&H&I
\end{pmatrix}^\top=A^{(1)}_{k_1,k_2,k_3}.
\]
Then
\[
\begin{pmatrix}
AA^\top&AD^\top&AG^\top\\DA^\top&DD^\top&DG^\top\\GA^\top&GD^\top&GG^\top
\end{pmatrix}=\begin{pmatrix}
A&B&C\\D&E&F\\G&H&I
\end{pmatrix}\begin{pmatrix}{\bf 1}_{n_1}\\&{\bf 0}_{n_2}\\&&{\bf 0}_{n_3}\end{pmatrix}
\begin{pmatrix}
A&B&C\\D&E&F\\G&H&I
\end{pmatrix}^\top\in A^{(1)}_{k_1,k_2,k_3}.
\]
So of $A,D,G$, we see that one is of the form $\lambda_1 Q_1$ with $Q_1\in O(n_1)$ and $\lambda_1\in\CC^*$. It then easily follows that the other two matrices are zero. Note that $D=\lambda_1 Q_1$ is only possible when $n_1=n_2$ and $G=\lambda Q_1$ is only possible when $n_1=n_2=n_3$. Similarly, we see that one of $B,E,H$ is of the form $\lambda_2 Q_2$ with $Q_2\in O(n_2)$ and $\lambda_2\in\CC^*$ and the other two matrices are zero. And, we see that one of $C,F,I$ is of the form $\lambda_2 Q_2$ with $Q_3\in O(n_3)$ and $\lambda_3\in\CC^*$ and the other two matrices are zero. It is straightforward to check that this matrix is in the group generated by the given matrices.
\end{proof}

\begin{prop}
The stabilizer of $B^{(1)}_{n/2}$ is
\[
\left\{\begin{pmatrix}aQ&bQ\\cQ&dQ\end{pmatrix}\,\middle|\,Q\in O(n/2),\begin{pmatrix}a&b\\c&d\end{pmatrix}\in\GL_2(\CC)\right\}.
\]
The orbit of $B^{(1)}_{n/2}$ has codimension $5(n^2/8 + n/4 - 1)$.
\end{prop}
\begin{proof}
Let $A,B,C,D\in\CC^{n/2\times n/2}$ be such that
\[
\begin{pmatrix}A&B\\C&D\end{pmatrix}B^{(1)}\begin{pmatrix}A&B\\C&D\end{pmatrix}^\top=B^{(1)}.
\]
Then 
\[
\begin{pmatrix}AA^\top&AC^\top\\CA^\top&CC^\top\end{pmatrix}=\begin{pmatrix}A&B\\C&D\end{pmatrix}\begin{pmatrix}{\bf 1}_{n/2}\\&{\bf 0}_{n/2}\end{pmatrix}\begin{pmatrix}A&B\\C&D\end{pmatrix}^\top\in B^{(1)}
\]
and hence $(A,C)=(a Q,c Q)$ for some $Q\in O(n/2)$ and $a,c\in\CC$. Similarly, we find that $(B,D)=(b P,d P)$ for some $P\in O(n/2)$ and $b,d\in\CC$. Now
\[
g:=\begin{pmatrix}A&B\\C&D\end{pmatrix}=\begin{pmatrix}a{\bf 1}_{n/2}&b{\bf 1}_{n/2}\\c{\bf 1}_{n/2}&d{\bf 1}_{n/2}\end{pmatrix}\begin{pmatrix}Q\\&P\end{pmatrix}
\]
and so $\Diag(Q,P)$ also lies in the stabilizer of $B^{(1)}_{n/2}$. It is straightforward to check that this is only possible when $P,Q$ are linearly dependent. So $g$ must be of the required form.
\end{proof}

We believe that the codimensions of the other orbits are also polynomials in their parameters.

\begin{conj}
The codimensions of $A^{(2)}_{r,k_1,k_2},A^{(3)}_{k_1,k_2,k_3},B^{(2)}_{k,\ell_1,\ell_2}$ in $\Gr(3,\S^n)$ are polynomials functions $f_2(r,k_1,k_2),f_3(k_1,k_2,k_3),g(k,\ell_1,\ell_2)$ of degree $\leq 2$, respectively.
\end{conj}

Assuming the conjecture, we find
\begin{eqnarray*}
f_2(r,k_1,k_2)&=&r^2 + rk_1 + k_1^2 + 2rk_2 + 3k_1k_2 + 3k_2^2 + r + k_1 + 2k_2 - 5,\\
f_3(k_1,k_2,k_3)&=&k_1^2 + 3k_1k_2 + 3k_2^2 + 4k_1k_3 + 8k_2k_3 + 6k_3^2 + k_1 + 2k_2 + 3k_3 - 4,\\
g(k,\ell_1,\ell_2)&=&3\ell_1^2/2 + 2\ell_1\ell_2 + 5\ell_2^2/2 - \ell_1/2 + 5\ell_2/2 - 5.
\end{eqnarray*}

\section{Proofs of the main results}\label{sec:proofs_main_results}

In this section, we prove the results from Section~\ref{sec:Jordan_nets_n<=6}.

\begin{re}
To prove that Figure~\ref{fig:nets_n=4} contains all degenerations, one needs to prove in particular that $A^{(1)}_{1,0,1}\not\to B^{(2)}_{1,0,2}$ and $A^{(2)}_{2,0,1}\not\to C_{4,1}$. In \cite{BES:JordanSpaces}, this was proven using equations on the orbits of $A^{(1)}_{1,0,1}$ and $A^{(2)}_{2,0,1}$ that do not hold for $B^{(2)}_{1,0,2}$ and $C_{4,1}$, respectively, found by a computer search. Now, we also see that $A^{(1)}_{1,0,1}\not\to B^{(2)}_{1,0,2}$ follows from the minimal rank obstruction and $A^{(2)}_{2,0,1}\not\to C_{4,1}$ follows from the Segre symbol obstruction. 
\end{re}

\begin{proof}[Proof of Theorem~\ref{thm:diagramNetsS5}]
To show that the diagram describes all degenerations, we need to show that $A^{(2)}_{3,0,1},B^{(2)}_{1,1,2}\not\to C_{5,1}$. Note that the set
$$
\left\{(\Ls,\Ps)\in\Gr(3,\S^n)\times\Gr(2,\S^n)\,\middle|\, \begin{array}{c}\Ps\subseteq\Ls, \rk(X)\leq 2\mbox{ for all }X\in\Ps,\\ \det(Q\Ps Q^\top)\in \{f^2\}\mbox{ for all }Q\in\CC^{2\times n}\end{array}\right\}
$$
is closed and $\GL_n(\CC)$-stable. Hence so is its projection on $\Gr(3,\S^n)$. The orbits $A^{(2)}_{3,0,1},B^{(2)}_{1,1,2}$ are contained in this projection while the orbit $C_{5,1}$ is not. So indeed $A^{(2)}_{3,0,1},B^{(2)}_{1,1,2}\not\to C_{5,1}$.
\end{proof}

\begin{proof}[Proof of Proposition~\ref{prop:B231_not_to_C_6}]
Suppose that $B^{(2)}_{1,0,3}\to C_{6,6}$ and identify $\Ls\subseteq\S^6$ with 
\[
(a,b,c,d,e,f)\Ls(a,b,c,d,e,f)^\top
\]
so that $B^{(2)}_{1,0,3}=\spann(2ac+b^2,ad,2df+e^2)$ and $C_{6,6}=\spann(ab,ac,af+be+cd)$. Using Remark~\ref{re:over_CC((t))}, there exist linear forms $A,B,C,D,E,F$ in $a,b,c,d,e,f$ over $\CC((t))$ and a matrix $Q\in\GL_3(\CC((t)))$ such that every entry of
\[
(G_1,G_2,G_3):= (2AC+B^2,AD,2DF+E^2)Q
\]
is a form with coefficients in $\CC[[t]]$ and $\lim_{t\to0}(G_1,G_2,G_2)=(ab,ac,af+be+cd)$. Note that 
\[
2AC+B^2,AD,2DF+E^2
\]
are linearly independent over $\CC((t))$ and
\[
\lambda_1(2AC+B^2)+\lambda_2AD+\lambda_3(DF+E^2), \quad \lambda_1,\lambda_2,\lambda_3\in\CC((t))
\]
converges to an element of $C_6$ as $t\to0$ whenever its coefficients lie in $\CC[[t]]$. Note that we are allowed to replace $(A,B,C,D,E,F)$ by $(\mu_1A,\mu_2B,\mu_3C,\mu_4D,\mu_5E,\mu_6F)$ as long as $\mu_1\mu_3=\mu_2^2$ and $\mu_4\mu_6=\mu_5^2$. So we may assume that $A,D$ converge to nonzero forms in $a,b,c,d,e,f$. Now $AD$ converges to a nonzero element of $C_6$. Since $AD$ has rank $2$, this element must be $a(\lambda b+\mu c)$ for some $(\lambda:\mu)\in\PP^1$. Using a basechange in $b,c$, we may assume that $(\lambda,\mu)=(1,0)$. Using symmetry and by scaling $A,D$, we may assume that $A\to a$ and $D\to b$ as $t\to0$. Next, consider the form $2DF+E^2$. Since $AD\to ab$ as $t\to0$, there exists an $\lambda\in\CC((t))$ such that
\[
\lambda AD+2DF+E^2\in\spann_{\CC((t))}(a^2,b^2,ac,bc,c^2)
\]
and we replace $F$ by $F+\lambda D/2$. We scale $E,F$ such that $2DF+E^2$ converges to a nonzero element of $C_6$. Since $2DF+E^2$ has rank $3$ and its limit lies in $\spann(a^2,b^2,ac,bc,c^2)$, we can scale so that $2DF+E^2\to ac$ as $t\to0$.\bigskip

We now see that $(b,ac)$ is a limit of pairs of the form $(D,2DF+E^2)$ where $D,E,F$ are forms in $a,b,c,d,e,f$. By setting $d,e,f$ to zero, we see that $(b,ac)$ is also a limit of pairs of the form $(D,2DF+E^2)$ where $D,E,F$ are forms in $a,b,c$. The closure of such pairs forms a hyperplane in $\CC\{a,b,c\}\times\CC\{a^2,ab,b^2,ac,bc,c^2\}$ that does not contain $(b,ac)$. This is a contradiction.
\end{proof}

\begin{proof}[Proof of Theorem~\ref{thm:diagramNetsS6}]
We have the following obstructions:
\begin{itemize}
\item We have $B^{(2)}_{1,0,3}\not\to C_{6,6}$ by Proposition~\ref{prop:B231_not_to_C_6}.
\item We have $A^{(1)}_{3,0,1}\not\to C_{6,5}$ since $\tau_2(A^{(1)}_{3,0,1})=2<3=\tau_2(C_{6,5})$.
\item We have $A^{(1)}_{3,0,1}\not\to C_{6,6}$ since $\tau_1(A^{(1)}_{3,0,1})=1<2=\tau_1(C_{6,6})$.
\item We have $A^{(2)}_{3,1,1}\not\to C_{6,4}$ since $\sigma(A^{(2)}_{3,1,1})=(111)(21)\not\geq(222)=\sigma(C_{6,4})$.
\item We have $A^{(2)}_{4,0,1}\not\to C_{6,7}$ since $\sigma(A^{(2)}_{4,0,1})=(1111)2\not\geq(2211)=\sigma(C_{6,7})$.
\item We have $B^{(2)}_{1,2,2}\not\to C_{6,5}$ since $\tau_2(B^{(2)}_{1,2,2})=2<3=\tau_2(C_{6,5})$.
\item We have $A^{(3)}_{0,0,2}\not\to C_{6,4}$ since $\sigma(A^{(3)}_{0,0,2})=(33)\not\geq(222)=\sigma(C_{6,4})$.
\end{itemize}
Using these obstructions, we see that the only possibly missing degenerations are the dotted lines.
\end{proof}

\appendix

\section{Embeddings of indecomposable Jordan algebras into $\S^n$}\label{sec:irr_embeddings}

Throughout this section, we write
\[
B:=\begin{pmatrix}1&i\\i&-1\end{pmatrix}
\]
and fix integers $n,m\geq 1$. For matrices $X=(X_{ij})_{ij},Y\in\CC^{n\times n}$, the {\em Kronecker product} $X\otimes Y$ is the matrix
\[
\begin{pmatrix}X_{11}Y&\cdots&X_{1n}Y\\\vdots&&\vdots\\X_{n1}Y&\cdots&X_{nn}Y\end{pmatrix}.
\]
Denote the matrix in $\S^n$ with $1$'s on its anti-diagonal and $0$'s everywhere else by~${\bf J}_n$ and define $\Jo_{\bf J}(m,\S^n)$ to be the subset of $\Jo(m,\S^n)$ of subspaces containing ${\bf J}_n$.

\begin{prop}\label{prop:J_to_I}
Take
\[
Q_{2m}:= \frac{1}{\sqrt{2}}\begin{pmatrix}{\bf 1}_m&{\bf J}_m\\i{\bf J}_m&-i{\bf 1}_m\end{pmatrix}\quad\mbox{and}\quad Q_{2m+1}:= \frac{1}{\sqrt{2}}\begin{pmatrix}{\bf 1}_m&&{\bf J}_m\\&\sqrt{2}\\i{\bf J}_m&&-i{\bf 1}_m\end{pmatrix}
\]
for all integers $m\geq 1$. Then the map
\begin{eqnarray*}
\Jo_{\bf J}(m,\S^n)&\to&\Jo_{\bf 1}(m,\S^n)\\
\Ls&\mapsto& Q_n\Ls Q_n^\top
\end{eqnarray*}
is a bijection.
\end{prop}
\begin{proof}
This follows from the fact that $Q_n{\bf J}_nQ_n^\top={\bf 1}_n$.
\end{proof}

Using the previous proposition, we will often represent orbits using elements of $\Jo_{\bf J}(m,\S^n)$, but work with elements of $\Jo_{\bf 1}(m,\S^n)$ during proofs. 

\begin{ex}\label{ex:J_to_I}
Consider the Jordan pencil
\[
\Ps_{n,k}:=x{\bf J}_n+y\Diag({\bf 1}_k,{\bf 0}_{n-k})
\]
for $n\geq2$ and $1\leq k\leq n/2$. We have $Q_2\Diag(1,0)Q_2^\top=B$ and so Proposition~\ref{prop:J_to_I} shows that $\Ps_{2,1}$ is congruent to $x{\bf 1}_2+yB$. More generally, we have
\[
Q_n\Ps_{n,k}Q_n^\top= P(x{\bf 1}_n+y\Diag({\bf 1}_k\otimes B,{\bf 0}_{n-2k})/2)P^\top
\]
for some permutation matrix $P$. So $\Ps_{n,k}$ is congruent to $x{\bf 1}_n+y\Diag({\bf 1}_k\otimes B,{\bf 0}_{n-2k})$.
\end{ex}

\begin{lm}\label{lm:similar=orth_congruent}
Let $X,Y\in\CC^{n\times n}$ be matrices. Assume that $X,Y$ are both symmetric or both skew-symmetric. If $X,Y$ are similar, then they are orthogonally congruent.
\end{lm}
\begin{proof}
When $X,Y$ are both symmetric, this is \cite[Lemma 1]{BO:symmetricnilpotent}. When $X,Y$ are both skew-symmetric, the proof equals that of \cite[Lemma 1]{BO:symmetricnilpotent}.
\end{proof}

\begin{lm}\label{lm:idempotent}
Let $X\in\S^n$ be an idempotent matrix of rank $r$. Then $X$ is orthogonally congruent to the matrix $\Diag({\bf 1}_r,{\bf 0}_{n-r})$. If $X=\Diag({\bf 1}_r,{\bf 0}_{n-r})$, then we have
\[\begin{array}{rcccl}
\{Y\in\S^n\mid XY+YX=&\!\!\!0\!\!\!\!\!\!&\}&=&\Diag({\bf 0}_r,\S^{n-r}),\\
\{Y\in\S^n\mid XY+YX=&\!\!\!Y\!\!\!\!\!\!&\}&=&\left\{\begin{pmatrix}{\bf 0}_r&Z\\Z^\top&{\bf 0}_{n-r}\end{pmatrix}\,\middle|\, Z\in\CC^{r\times(n-r)}\right\},\\
\{Y\in\S^n\mid XY+YX=&\!\!\!2Y\!\!\!\!\!\!&\}&=&\Diag(\S^r,{\bf 0}_{n-r}).
\end{array}\]
\end{lm}
\begin{proof} 
The first statement holds by Lemma~\ref{lm:similar=orth_congruent}. The second statement follows easily.
\end{proof}

\begin{prop}\label{prop:embedding_product}
Let $\As_1,\As_2$ be Jordan algebras. Then any embedding of $\As_1\times\As_2$ into $\S^n$ is congruent to $\Diag(\Ls_1,\Ls_2)$ for some embeddings $\Ls_i$ of $\As_i$ into $\S^{n_i}$ with $n_1+n_2=n$.
\end{prop}
\begin{proof}
Let $\As_1,\As_2$ be Jordan algebras and let $\iota\colon \As_1\times\As_2\to\S^n$ be an injective morphism of Jordan algebras. By Lemma~\ref{lm:idempotent}, we can assume that $\iota$ sends the unit of $\As_1$ to $\Diag({\bf 1}_{n_1},{\bf 0}_{n_2})$ and the unit of $\As_2$ to $\Diag({\bf 0}_{n_1},{\bf 1}_{n_2})$ for some integers $n_1,n_2\geq 1$ adding up to $n$. Lemma~\ref{lm:idempotent} now also shows that $\iota(\As_1\times\,0)=\Diag(\Ls_1,{\bf 0}_{n_2})$ and $\iota(0\times\As_1)=\Diag({\bf 0}_{n_1},\Ls_2)$ where $\Ls_i$ is an embedding of $\As_i$ into $\S^{n_i}$.
\end{proof}

By the proposition, in order to classify embeddings of Jordan algebras of dimension $m$ into $\S^n$, it suffices to classify embeddings of the indecomposable Jordan algebras of dimension $\leq m$ into $\S^{n'}$ for all integers $n'\leq n$. We restrict to $m\leq 4$ and consider the indecomposable Jordan algebras one-by-one. In some cases, we also restrict to low $n$.

\subsection{The Jordan algebra $\CC$} 
An embedding of $\CC$ into $\S^n$ is of the form $\CC U$ where $U\in\S^n$ is an invertible matrix.

\begin{prop}\label{prop:classify_CC}
Every embedding of $\CC$ is congruent to $\CC{\bf 1}_n$.
\end{prop}
\begin{proof}
All invertible matrices in $\S^n$ are congruent. Hence $\CC U$ is congruent to~$\CC{\bf 1}_n$.
\end{proof}

\subsection{The Jordan algebra $\CC[x]/(x^m)$}
An embedding of $\CC[x]/(x^m)$ into $\S^n$ is of the form 
\[
\CC\{U,X,X^{\bullet_U2},\ldots,X^{\bullet_Um-1}\}
\]
where $U\in\S^n$ is an invertible matrix and $X\in\S^n$ satisfies $X^{\bullet_U(m-1)}\neq{\bf 0}_n$ and $X^{\bullet_Um}={\bf 0}_n$.

\begin{prop}\label{prop:classify_CC[x]/x^m} 
Every embedding of $\CC[x]/(x^m)$ is congruent to
\[
\Diag\left({\bf 1}_{k_m}\otimes\begin{pmatrix}x_m&x_{m-1}&\ldots &x_2&x_1\\x_{m-1}&&\rddots&\rddots\\\vdots&\rddots&\rddots\\x_2&\rddots\\x_1\end{pmatrix},\ldots,{\bf 1}_{k_2}\otimes\begin{pmatrix}x_2&x_1\\x_1\end{pmatrix},x_1{\bf 1}_{k_1}\right)
\]
for some integers $k_1,\ldots,k_{m-1}\geq 0$ and $k_m\geq 1$ such that $\sum_{i=1}^m ik_i =n$.
\end{prop}
\begin{proof}
After a congruence, we may assume that $U={\bf 1}_n$. Now $X^{m-1}\neq{\bf 0}_n$ and $X^m={\bf 0}_n$. So there exist unique integers $k_1,\ldots,k_{m-1}\geq 0$ and $k_m\geq 1$ such that the Jordan normal form of $X$ has $k_i$ blocks of size $i\times i$. Using Proposition~\ref{prop:J_to_I}, we see that
\[
\Diag\left({\bf 1}_{k_m}\otimes\begin{pmatrix}x_m&x_{m-1}&\dots &x_2&x_1\\x_{m-1}&&\rddots&\rddots\\\vdots&\rddots&\rddots\\x_2&\rddots\\x_1\end{pmatrix},\ldots,{\bf 1}_{k_2}\otimes\begin{pmatrix}x_2&x_1\\x_1\end{pmatrix},x_1{\bf 1}_{k_1}\right)
\]
is congruent to $\CC\{{\bf 1}_n,Y,Y^2,\ldots Y^{m-1}\}$ for a $Y\in\S^n$ with the same Jordan normal form. Hence $X,Y$ are orthogonally congruent and hence so are $\CC\{{\bf 1}_n,X,X^2,\ldots, X^{m-1}\},\CC\{{\bf 1}_n,Y,Y^2,\ldots, Y^{m-1}\}$.
\end{proof}

\subsection{The Jordan algebra $\Js^2_0$}
An embedding of $\Js^2_0$ into $\S^n$ is of the form $\CC\{U,X,Y\}$ where $U\in\S^n$ is an invertible matrix and $X,Y\in\S^n$ satisfy $X\bullet_UX=X\bullet_UY=Y\bullet_UY={\bf 0}_n$.

\begin{de}
A subspace $\Ls\subseteq\S^n$ is called {\em square-zero} when $X^2=0$ for all matrices $X\in\Ls$.
\end{de}

\begin{prop}\label{prop:classify_J20}
Every embedding of $\Js_0^2$ is congruent to 
\[
\CC{\bf 1}_n\oplus \Ps
\]
for some square-zero pencil $\Ps\subseteq\S^n$.
\end{prop}
\begin{proof}
By applying a congruence, we may assume that $U={\bf 1}_n$. Then we get $X^2=XY+YX=Y^2={\bf 0}_n$. So $\Ls=\CC{\bf 1}_n\oplus \Ps$ for the square-zero pencil $\Ps:=\CC\{X,Y\}$.
\end{proof}

Next, we wish to classify square-zero pencils in $\S^n$. We do this for $n\leq 7$ using the following lemma.

\begin{lm}\label{lm:orbits_isotropic_spaces}
Let $n\geq 2$ and $1\leq k\leq n/2$ be integers. Then the set
\[
\{(v_1,\ldots,v_k)\in\CC^{n\times k}\mid v_1,\ldots,v_k\mbox{ linearly independent},\forall i,j:v_i^\top v_j=0\}
\]
forms a single $\O(n)$-orbit.
\end{lm}
\begin{proof}
Let $(v_1,\ldots,v_k),(w_1,\ldots,w_k)$ be any two tuples in this set. Take $V=\CC\{v_1,\ldots,v_k\}$ and let $\overline{V}=\{\overline{v}\mid v\in V\}$ be its conjugate subspace. Then $\overline{V}=\CC\{x_1,\ldots,x_k\}$ for some vectors $x_1,\ldots,x_k\in\CC^n$ such that $v_i^\top x_j=\delta_{ij}$. Similarly, define $y_1,\ldots,y_k\in\CC^n$ such that $w_i^\top y_j=\delta_{ij}$. The sets
\[
\left\{\frac{v_1+x_1}{\sqrt 2},\frac{v_1-x_1}{\sqrt 2},\ldots,\frac{v_k+x_k}{\sqrt 2},\frac{v_k-x_k}{\sqrt 2}\right\}\mbox{ and }\left\{\frac{w_1+y_1}{\sqrt 2},\frac{w_1-y_1}{\sqrt  2},\ldots,\frac{w_k+y_k}{\sqrt 2},\frac{w_k-y_k}{\sqrt 2}\right\}
\]
are both orthonormal. So there exists a matrix in $\O(n)$ sending the first set to the second. This matrix sends $v_i$ to $w_i$.
\end{proof}

Write $k=\lfloor n/2\rfloor$. When $n$ is odd, we view $\S^{2k}$ as the subspace of $\S^n$ consisting of all symmetric matrices whose last column/row is zero. 

\begin{prop}\label{prop:squarezero_pencil}
Let $B\in\S^2$ be the square zero matrix from Example~\ref{ex:J_to_I}. If $4\leq n\leq 7$, then any square-zero pencil in $\S^n$ is orthogonally congruent to $\Ps\otimes B\subseteq\S^{2k}$ for some pencil $\Ps\subseteq\S^k$. Let $\Ps,\Ps'\subseteq\S^k$ be congruent pencils. Then $\Ps\otimes B,\Ps'\otimes B$ are orthogonally congruent.
\end{prop}
\begin{proof}
Let $\Ps:=\CC\{X,Y\}\subseteq\S^n$ be a square-zero pencil. Suppose that $4\leq n\leq 7$. By changing the basis $X,Y$, we assume that the rank $r\leq n/2$ of $Y$ is maximal in~$\Ps$. When $X,Y$ both have rank $1$, then $X+Y$ has rank $2$. Hence $r\geq 2$. By applying an orthogonal congruence, we may assume that $Y=\Diag({\bf 1}_r\otimes B,{\bf 0}_{n-2r})$. We now check the conditions $X^2=XY+YX={\bf 0}_n$ and $\rk(\lambda X+\mu Y)\leq r$ for all $\lambda,\mu\in\CC$ by computer:
\begin{itemize}
\item For $(n,r)=(4,2)$, we find that $X=Z\otimes B$ for some $Z\in\S^2$.
\item For $(n,r)=(5,2)$, we find that $X=\Diag(Z\otimes B,0)$ for some $Z\in\S^2$.
\item For $(n,r)=(6,3)$, we find that $X=Z\otimes B$ for some $Z\in\S^3$.
\item For $(n,r)=(6,2)$, we find that the variety of $X$ satisfying these conditions has $5$ components $X_1,X_2,X_3,X_4,X_5$. The component $X_1$ consists of matrices of the form $X=\Diag(Z\otimes B,{\bf 0}_2)$ for some $Z\in\S^2$. Acting with the matrix $\Diag({\bf 1}_4,-1,1)$ permutes $X_2,X_3$ and acting with the matrix $\Diag({\bf 1}_5,-1)$ permutes $X_4,X_5$. The components $X_2,X_4$ consist of matrices $Z\otimes B$ for $Z\in\S^3$ (with additional conditions on~$Z$).
\item For $(n,r)=(7,3)$, we find that $X=\Diag(Z\otimes B,0)$ for some $Z\in\S^3$.
\item For $(n,r)=(7,2)$, we find that the variety of $X$ satisfying these conditions has $3$ components. One component consists of matrices of the form $X=\Diag(Z\otimes B,{\bf 0}_3)$ for some $Z\in\S^2$. Write $v=(1,i)^\top$. Then the $2$ other components consist of matrices of the form
\[
\begin{pmatrix}
dB&eB&av&bv&cv\\
eB&fB&\pm av&\pm bv&\pm cv\\
av^\top&\pm av^\top\\
bv^\top&\pm bv^\top\\
cv^\top&\pm cv^\top
\end{pmatrix}
\]
such that $f\pm2ie-d=0$ and $a^2+b^2+c^2=0$. Acting with $\Diag({\bf 1}_4,\O(3))$, we reduce to the case where $(a,b,c)=(0,0,0)$ or $(a,b,c)=(1,i,0)$ by Lemma~\ref{lm:orbits_isotropic_spaces}. Now, in both cases, the matrix is of the form $X=\Diag(Z\otimes B,0)$ for some $Z\in\S^3$.
\end{itemize}
Write $k=\lfloor n/2\rfloor$. In all cases, we see that after an orthogonal congruence $\Ps$ is of the form $\Ps'\otimes B\subseteq\S^{2k}\subseteq\S^n$ for some pencil $\Ps\subseteq\S^k$.\bigskip

Let $\Ps,\Ps'\subseteq\S^k$ be pencils. If $P\in\GL(n)$ is an orthogonal matrix, then $P\otimes{\bf 1}_2$ is as well. So if $\Ps,\Ps'$ are orthogonally congruent, then so are $\Ps\otimes B,\Ps'\otimes B$. Next, suppose that $\Ps'=P\Ps P^{\top}$ where $P=\Diag(\lambda_1,\ldots,\lambda_n)\in\GL(n)$ is a diagonal matrix. As $B=(1,i)^{\top}(1,i)$, it follows that 
\[
\Ps'\otimes B=Q(\Ps\otimes B)Q^{\top}
\]
where $Q=\Diag(Q_1,\ldots,Q_n)$ and $Q_i\in\O(2)$ such that $Q_i(1,i)^{\top}=\lambda_i(1,i)^{\top}$. By Lemma~\ref{lm:orbits_isotropic_spaces}, such matrices $Q_i$ exist. So also in this case, the pencils $\Ps\otimes B,\Ps'\otimes B$ are orthogonally congruent. As $\O(n)$ and the diagonal matrices generate $\GL(n)$, we see that the proposition holds.
\end{proof}

\begin{re}\label{re:sz_pencils_n=8}
The proposition does not hold for $n=8$. Indeed, consider the embedding $\Ls=\CC\{U,X,Y\}$ of $\Js^2_0$ where 
\[
U=\begin{pmatrix}
&&&{\bf J}_2\\
&&{\bf J}_2\\
&{\bf J}_2\\
{\bf J}_2
\end{pmatrix},
X=\begin{pmatrix}
{\bf 1}_2\\
&{\bf 1}_2\\
&&{\bf 0}_2\\
&&&{\bf 0}_2
\end{pmatrix},
Y=\begin{pmatrix}
&&&B\\
&&-B\\
&-B\\
B
\end{pmatrix}
\]
and let $P\in\GL_n$ be such that $PUP^\top={\bf 1}_8$. Then $\Ps=P\CC\{X,Y\}P^\top$ is a square-zero pencil. We have $XU^{-1}Y\neq{\bf 0}_8$ and hence $X'Y'\neq{\bf 0}_8$ for $X'=PXP^\top$ and $Y'=PYP^\top$. This is not possible if $\Ps$ is orthogonally congruent to $\Ps'\otimes B$ for some pencil $\Ps'\subseteq\S^4$.
\end{re}

\begin{que}
Is it possible to classify the embeddings of $\Js^2_0$ into $\S^n$ for general $n$?
\end{que}

Using the proposition, we see that to classify square-zero pencils for $n\leq 7$, it suffices to classify pencils in $\S^k$ up to congruence for $k\leq 3$.

\begin{prop}\label{prop:pencilsS2}
Every pencil in $\S^2$ is congruent one of
\[
\begin{pmatrix}x\\&y\end{pmatrix},\begin{pmatrix}y&x\\x\end{pmatrix}
\]
\end{prop}
\begin{proof}
This is \cite[Examples 1.2]{FMS:pencil_old_new}.
\end{proof}

\begin{samepage}
\begin{prop}\label{prop:pencilsS3}
We have the following:
\begin{enumerate}
\item[(1)] Every regular pencil in $\S^3$ is congruent one of
\[
\begin{pmatrix}
x\\&y\\&&x+y
\end{pmatrix},\begin{pmatrix}
x&y\\y\\&&x
\end{pmatrix},\begin{pmatrix}
&y&x\\y&x\\x
\end{pmatrix},\begin{pmatrix}
x\\&x\\&&y
\end{pmatrix},\begin{pmatrix}
y&x\\x\\&&x
\end{pmatrix}.
\]

\item[(2)] Every singular pencil in $\S^3$ is congruent one of
\[
\begin{pmatrix}
&x&y\\x\\y
\end{pmatrix},\begin{pmatrix}
x\\&y\\&&0
\end{pmatrix},\begin{pmatrix}
y&x\\x\\&&0
\end{pmatrix}.
\]
\end{enumerate}
\end{prop}
\end{samepage}
\begin{proof}
Part (1) is \cite[Examples 1.3]{FMS:pencil_old_new}. For part (2), let $\Ps\subseteq\S^3$ be a singular pencil. So every matrix in $\Ps$ has rank $\leq 2$. As $\Ps\subseteq\S^3$ and $\dim(\Ps)\geq 2$, the pencil $\Ps$ must contain a matrix of rank $\geq2$. So $\Ps$ contains a matrix $X$ of rank $2$. After congruence, we may assume that $X=\Diag({\bf J}_2,0)$. Let $X,Y$ be a basis of $\Ps$. Since no linear combination of $X,Y$ has rank~$3$, we have
\[
Y=\begin{pmatrix}
a&e&b\\
e&c&d\\
b&d&0
\end{pmatrix}
\]
for $a,b,c,d,e\in\CC$ with $bd=ad=bc=0$. By subtracting a multiple of $X$, we may assume that $e=0$. We find that either $d=b=0$, $d=c=0$ or $b=a=0$. The last two cases are congruent. When the last column/row of $Y$ is zero, we have $\Ps=\Diag(\Ps',0)$ for some pencil $\Ps'\subseteq\S^2$. Otherwise, we find that $\Ps$ is congruent to 
\[
\begin{pmatrix}
&x&y\\x\\y
\end{pmatrix}.
\]
\end{proof}

\subsection{The Jordan algebra $\Js^2_1$}
An embedding of $\Js^2_1$ into $\S^n$ is of the form $\CC\{X,Y,V\}$ where $U=X+Y\in\S^n$ is an invertible matrix and $X,Y,V\in\S^n$ satisfy 
\[
\begin{array}{ccc}
X\bullet_UX=X,&X\bullet_UY={\bf 0}_n,&X\bullet_UV=V/2,\\
&Y\bullet_UY=Y,&Y\bullet_UV=V/2,\\
&&V\bullet_UV={\bf 0}_n.
\end{array}
\]

\begin{lm}\label{lm:anti-diag_square-zero}
Let $1\leq r\leq n/2$ be an integer and consider matrices of the form
\[
Y=\begin{pmatrix}{\bf 0}_r&Z\\Z^\top&{\bf 0}_{n-r}\end{pmatrix}
\]
for $Z\in\CC^{r\times(n-r)}\setminus\{{\bf 0}_{r\times(n-r)}\}$. Then we have $Y^2=0$ if and only if $Y$ lies in the $\Diag(\O(r),\O(n-r))$-orbit of the matrix
\[
\Diag\left(\begin{pmatrix}{\bf 0}_r&\Diag({\bf 1}_k\otimes B,{\bf 0}_{r-2k})\\\Diag({\bf 1}_k\otimes B,{\bf 0}_{r-2k})&{\bf 0}_r\end{pmatrix},{\bf 0}_{n-2r}\right)
\]
where $k=\rk(Z)$.
\end{lm}
\begin{proof}
We have $Y^2=0$ if and only if $ZZ^{\top}={\bf 0}_r$ and $Z^{\top}Z={\bf 0}_{n-r}$. Take $1\leq k=\rk(Z)$ and write $Z=v_1w_1^\top+\ldots+v_kw_k^\top$ with $v_1,\ldots,v_k\in\CC^r$ and $w_1,\ldots,w_k\in\CC^{n-r}$. Note that $v_1,\ldots, v_k$ and $w_1,\ldots, w_k$ are both linearly independent. Therefore we have
\[
{\bf 0}_r=ZZ^\top=\sum_{i,j=1}^k w_i^\top w_j\cdot v_iv_j^\top,\quad {\bf 0}_{n-r}=Z^\top Z=\sum_{i,j=1}^k v_i^\top v_j\cdot w_iw_j^\top
\]
and so $w_i^\top w_j=v_i^\top v_j=0$ for all $i,j$. Applying Lemma~\ref{lm:orbits_isotropic_spaces} to $(v_1,\ldots,v_k)$ and $(w_1,\ldots,w_k)$, we find that $Y$ indeed lies in the stated $\Diag(\O(r),\O(n-r))$-orbit.
\end{proof}

\begin{prop}\label{prop:classify_J21}
Every embedding of $\Js_1^2$ is congruent to 
\[
\Diag\left(\begin{pmatrix}x{\bf J}_r&z\Diag({\bf 1}_k,{\bf 0}_{r-k})\\z\Diag({\bf 1}_k,{\bf 0}_{r-k})&y{\bf J}_r\end{pmatrix},y {\bf 1}_{n-2r}\right)
\]
for some integers $2\leq r\leq n/2$ and $1\leq k\leq r/2$.
\end{prop}
\begin{proof}
After a congruence, we may assume that $X=\Diag({\bf 1}_r,{\bf 0}_{n-r})$ and $Y=\Diag({\bf 0}_r,{\bf 1}_{n-r})$ for some integer $1\leq r\leq n-1$. By switching $X,Y$, we may assume that $r\leq n/2$. Now 
\[
V=\begin{pmatrix}{\bf 0}_r&Z\\Z^\top&{\bf 0}_{n-r}\end{pmatrix}
\]
for some in $Z\in\CC^{r\times(n-r)}\setminus\{{\bf 0}_{r\times(n-r)}\}$. So using Lemma~\ref{lm:anti-diag_square-zero} and Proposition~\ref{prop:J_to_I}, we are done.
\end{proof}

\subsection{The Jordan algebra $\Js^2_2$}
An embedding of $\Js^2_2$ into $\S^n$ is of the form $\CC\{X,Y,Z\}$ where $U=X+Y\in\S^n$ is an invertible matrix and $X,Y,Z\in\S^n$ satisfy 
\[
\begin{array}{ccc}
X\bullet_UX=X,&X\bullet_UY={\bf 0}_n,&X\bullet_UV=V/2,\\
&Y\bullet_UY=Y,&Y\bullet_UV=V/2,\\
&&V\bullet_UV=U.
\end{array}
\]

\begin{prop}\label{prop:classify_J22}
The Jordan algebra $\Js^2_2$ has no embeddings when $n$ is odd. When $n$ is even, every embedding of $\Js_2^2$ is congruent to ${\bf 1}_{n/2}\otimes\S^2$.
\end{prop}
\begin{proof}
This is part of \cite[Theorem 5.3]{BES:JordanSpaces}.
\end{proof}

\subsection{The Jordan algebra $\Js^3_0$}
An embedding of $\Js^3_0$ into $\S^n$ is of the form $\CC\{U,X,Y,Z\}$ where $U\in\S^n$ is an invertible matrix and $X,Y,Z\in\S^n$ satisfy 
\[
X\bullet_UX=X\bullet_UY=X\bullet_UZ=Y\bullet_UY=Y\bullet_UZ=Z\bullet_UZ={\bf 0}_n.
\]

\begin{prop}\label{prop:classify_J30}
Every embedding of $\Js_0^3$ is congruent to $\CC{\bf 1}_n\oplus \Ls$ for some square-zero net $\Ls\subseteq\S^n$.
\end{prop}
\begin{proof}
After a congruence, we may assume that $U={\bf 1}_n$. Now the embedding equals $\CC{\bf 1}_n\oplus \Ls$ for the square-zero net $\Ls:=\CC\{X,Y,Z\}$.
\end{proof}

Write $k=\lfloor n/2\rfloor$.

\begin{prop}\label{prop:squarezero_net}
If $4\leq n\leq 5$, then any square-zero net in $\S^n$ is orthogonally congruent to $\S^2\otimes B\subseteq\S^{2k}\subseteq\S^n$. 
\end{prop}
\begin{proof}
Let $\Ls:=\CC\{X,Y,Z\}$ be a square-zero net in $\S^n$. If $X,Y$ has rank $1$, then $X+Y$ has rank $2$. So by changing basis, we may assume that the rank of $X$ is $2$. By applying an orthogonal congruence, we may assume that $X={\bf 1}_2\otimes B$ when $n=4$ and $X=\Diag({\bf 1}_2\otimes B,0)$ when $n=5$. In this case, we verify by computer that $\Ls=\Ls'\otimes B$  when $n=4$ and $\Ls=\Diag(\Ls'\otimes B,0)$ when $n=5$ for some net $\Ls'\subseteq\S^2$. Clearly, it follows that $\Ls'=\S^2$.
\end{proof}

\subsection{The Jordan algebra $\Js^3_1$}
An embedding of $\Js^3_1$ into $\S^n$ is of the form $\CC\{X,Y,V,W\}$ where $U=X+Y\in\S^n$ is an invertible matrix and $X,Y,V,W\in\S^n$ satisfy 
\[
\begin{array}{cccc}
X\bullet_UX=X,&X\bullet_UY={\bf 0}_n,&X\bullet_UV=V/2,&X\bullet_UW=W/2,\\
&Y\bullet_UY=Y,&Y\bullet_UV=V/2,&Y\bullet_UW=W/2,\\
&&V\bullet_UV={\bf 0}_n,&V\bullet_UW={\bf 0}_n,\\
&&&W\bullet_UW={\bf 0}_n.
\end{array}
\]

\begin{prop}\label{prop:classify_J31}
Every embedding of $\Js_1^3$ is congruent to
\[
\begin{pmatrix}x{\bf 1}_r\\&y{\bf 1}_{n-r}\end{pmatrix} +\left\{\begin{pmatrix}&Z\\Z^\top\end{pmatrix}\,\middle|\, Z\in  \Ps\right\}
\]
for some integer $2\leq r\leq n/2$ and some pencil $\Ps\subseteq\CC^{r\times(n-r)}$ such that $ZZ^\top={\bf 0}_r$ and $Z^\top Z={\bf 0}_{n-r}$ for all $Z\in\Ps$. If $r=2$, then $n\geq 6$ and we may take
\[
\Ps=(1~i)^\top(x,ix,y,iy,0,\ldots,0).
\]
\end{prop}
\begin{proof}
After a congruence, we may assume that $X=\Diag({\bf 1}_r,{\bf 0}_{n-r})$ and $Y=\Diag({\bf 0}_r,{\bf 1}_{n-r})$ for some integer $1\leq r\leq n-1$. By switching $X,Y$, we may assume that $r\leq n/2$. Now we see that $\CC\{V,W\}$ is of the form
\[
\left\{\begin{pmatrix}&Z\\Z^\top\end{pmatrix}\,\middle|\, Z\in  \Ps\right\}
\]
for some pencil $\Ps\subseteq\CC^{r\times(n-r)}$ such that $ZZ^\top={\bf 0}_k$ and $Z^\top Z={\bf 0}_{n-r}$ for all $Z\in\Ps$. When $r=1$, this is not possible. When $r=2$, we get $\Ps^\top=(1~i)^\top U$ where $U\subseteq\CC^{n-2}$ is a $2$-dimensional subspace such that $v^\top v=0$ for all $v\in U$. This is only possible when $n-2\geq 4$. And, by Lemma~\ref{lm:orbits_isotropic_spaces}, we may assume that $U$ is spanned by $e_1+ie_2$ and $e_3+ie_4$.
\end{proof}

\subsection{The Jordan algebra $\Js^3_2$}
An embedding of $\Js^3_2$ into $\S^n$ is of the form $\CC\{X,Y,V,W\}$ where $U=X+Y\in\S^n$ is an invertible matrix and $X,Y,V,W\in\S^n$ satisfy 
\[
\begin{array}{cccc}
X\bullet_UX=X,&X\bullet_UY={\bf 0}_n,&X\bullet_UV=V/2,&X\bullet_UW=W/2,\\
&Y\bullet_UY=Y,&Y\bullet_UV=V/2,&Y\bullet_UW=W/2,\\
&&V\bullet_UV=U,&V\bullet_UW={\bf 0}_n,\\
&&&W\bullet_UW={\bf 0}_n.
\end{array}
\]

\begin{prop}\label{prop:classify_J32}
The Jordan algebra $\Js^3_2$ has no embeddings when $n$ is odd. When $n$ is even, every embedding of $\Js_2^3$ is congruent to 
\[
\begin{pmatrix}x{\bf J}_{n/2}&\!z{\bf J}_{n/2}+w\Diag\!\left(\!{\bf 1}_k\!\otimes\!\begin{pmatrix}&1\\-1\!\!\!\!\!\end{pmatrix}\!,{\bf 0}_{n/2-2k}\!\right)\!\\\!z{\bf J}_{n/2}-w\Diag\!\left(\!{\bf 1}_k\!\otimes\!\begin{pmatrix}&1\\-1\!\!\!\!\!\end{pmatrix}\!\!,{\bf 0}_{n/2-2k}\!\right)\!
&y{\bf J}_{n/2}\end{pmatrix}
\]
for some integer $1\leq k\leq n/8$.
\end{prop}
\begin{proof}
For $X,Y,V$ to satisfy the required relations, we need $2\mid n$. In this case, we may assume that 
\[
xX+yY+zV=\begin{pmatrix}x{\bf 1}_{m}&z{\bf 1}_{m}\\z{\bf 1}_{m}&y{\bf 1}_{m}\end{pmatrix}
\]
where $n=2m$. Now we find that
\[
W=\begin{pmatrix}&P\\P^\top\end{pmatrix}
\]
for some $P\in\CC^{m\times m}$ such that $P^\top=-P$ and $P^2={\bf 0}_m$. The rank of $P$ is therefore even. Denote it by $2k$. Then $1\leq k\leq m/4$ as $1\leq\rk(P)\leq m/2$. As $P^2={\bf 0}_m$, the Jordan normal form of $P$ is uniquely determined by $k$. For a fixed $k$, all $P$ are orthogonally congruent. So acting with $\{\Diag(Q,Q)\mid Q\in O(m)\}$, we get a single congruence-orbit. Acting with $\Diag(Q_{n/2}^\top,Q_{n/2}^\top)$ as in Proposition~\ref{prop:J_to_I}, we get the required form.
\end{proof}

\subsection{The Jordan algebra $\Js^3_3$}
An embedding of $\Js^3_3$ into $\S^n$ is of the form $\CC\{X,Y,V,W\}$ where $U=X+Y\in\S^n$ is an invertible matrix and $X,Y,V,W\in\S^n$ satisfy 
\[
\begin{array}{cccc}
X\bullet_UX=X,&X\bullet_UY={\bf 0}_n,&X\bullet_UV=V/2,&X\bullet_UW=W/2,\\
&Y\bullet_UY=Y,&Y\bullet_UV=V/2,&Y\bullet_UW=W/2,\\
&&V\bullet_UV=U,&V\bullet_UW={\bf 0}_n,\\
&&&W\bullet_UW=U.
\end{array}
\]

\begin{prop}\label{prop:classify_J33}
The Jordan algebra $\Js^3_3$ has no embeddings when $4\nmid n$. When $4\mid n$, every embedding of $\Js_3^3$ is congruent to 
\[
\begin{pmatrix}x{\bf 1}_{n/2}&z{\bf 1}_{n/2}+w{\bf 1}_{n/4}\otimes\begin{pmatrix}&1\\-1\end{pmatrix}\\z{\bf 1}_{n/2}-w{\bf 1}_{n/4}\otimes\begin{pmatrix}&1\\-1\end{pmatrix}
&y{\bf 1}_{n/2}\end{pmatrix}.
\]
\end{prop}
\begin{proof}
For $X,Y,V$ to satisfy the required relations, we need $2\mid n$. In this case, we may assume that 
\[
xX+yY+zV=\begin{pmatrix}x{\bf 1}_{m}&z{\bf 1}_{m}\\z{\bf 1}_{m}&y{\bf 1}_{m}\end{pmatrix}
\]
where $n=2m$. Now we find that
\[
W=\begin{pmatrix}&P\\P^\top\end{pmatrix}
\]
for some $P\in\CC^{m\times m}$ such that $P^\top=-P$ and $P^2=-{\bf 1}_m$. These conditions can only be fulfilled when $2\mid m$. Assume this is the case. Then the Jordan normal form of $P$ is unique. Hence all such $P$ are orthogonally congruent. So acting with $\{\Diag(Q,Q)\mid Q\in O(m)\}$, we get
\[
W=\begin{pmatrix}&{\bf 1}_{n/4}\otimes\begin{pmatrix}&1\\-1\end{pmatrix}\\-{\bf 1}_{n/4}\otimes\begin{pmatrix}&1\\-1\end{pmatrix}\end{pmatrix}.
\]
\end{proof}

\subsection{The Jordan algebra $\Es_1$}
An embedding of $\Es_1$ into $\S^n$ is of the form $\CC\{X,Y,V,W\}$ where $U=X+Y\in\S^n$ is an invertible matrix and $X,Y,V,W\in\S^n$ satisfy 
\[
\begin{array}{cccc}
X\bullet_UX=X,&X\bullet_UY={\bf 0}_n,&X\bullet_UV=V,&X\bullet_UW=W/2,\\
&Y\bullet_UY=Y,&Y\bullet_UV={\bf 0}_n,&Y\bullet_UW=W/2,\\
&&V\bullet_UV={\bf 0}_n,&V\bullet_UW={\bf 0}_n,\\
&&&W\bullet_UW=V.
\end{array}
\]

\begin{prop}\label{prop:classify_E1}
If $n=4$, then every embedding of $\Es_1$ is congruent to one of
\[
\begin{pmatrix}v&w&x\\w&y\\x\\&&&x\end{pmatrix},\begin{pmatrix}v&w&x\\w&y\\x\\&&&y\end{pmatrix}.
\]
If $n=5$, then every embedding of $\Es_1$ is congruent to one of
\[
\begin{pmatrix}v&w&x\\w&y\\x\\&&&x\\&&&&x\end{pmatrix},\begin{pmatrix}v&w&x\\w&y\\x\\&&&x\\&&&&y\end{pmatrix},\begin{pmatrix}v&w&x\\w&y\\x\\&&&y\\&&&&y\end{pmatrix}.
\]
\end{prop}
\begin{proof}
After applying a congruence, we may assume that $X=\Diag({\bf 1}_r,{\bf 0}_{n-r})$, $Y=\Diag({\bf 0}_r,{\bf 1}_{n-r})$ for some $1\leq r\leq n-1$. Now we see that $V=\Diag(PP^\top,{\bf 0}_{n-r})$ and 
\[
W=\begin{pmatrix}
&P\\P^\top\end{pmatrix}
\]
for some matrix $P\in\CC^{r\times(n-r)}$ such that $P^\top P={\bf 0}_{n-r}$ and $PP^\top\neq{\bf 0}_r$. From $P^\top P={\bf 0}_{n-r}$ and $n\leq 5$, it follows that $P=vw^\top$ for some $v\in\CC^r\setminus\{0\}$ and $w\in\CC^{n-r}\setminus\{0\}$ such that $v^\top v=0$ and $w^\top w\neq0$. Hence $r\geq 2$. For fixed $n,r$, we have a single $\Diag(\O(r),\O(n-r))$-orbit. So for every $(n,r)$, we get one embedding up to congruence.
\end{proof}

\subsection{The Jordan algebra $\Es_2$}
An embedding of $\Es_2$ into $\S^n$ is of the form $\CC\{X,Y,V,W\}$ where $U=X+Y\in\S^n$ is an invertible matrix and $X,Y,V,W\in\S^n$ satisfy 
\[
\begin{array}{cccc}
X\bullet_UX=X,&X\bullet_UY={\bf 0}_n,&X\bullet_UV=V,&X\bullet_UW=W/2,\\
&Y\bullet_UY=Y,&Y\bullet_UV={\bf 0}_n,&Y\bullet_UW=W/2,\\
&&V\bullet_UV={\bf 0}_n,&V\bullet_UW={\bf 0}_n,\\
&&&W\bullet_UW={\bf 0}_n.
\end{array}
\]

\begin{prop}\label{prop:classify_E2}
If $n=4$, then every embedding of $\Es_2$ is congruent to $\Es_2$. If $n=5$, then every embedding of $\Es_2$ is congruent to one of
\[
\begin{pmatrix}v&x&w\\x\\w&&&y\\&&y\\&&&&x\end{pmatrix},\begin{pmatrix}v&x&w\\x\\w&&&y\\&&y\\&&&&y\end{pmatrix}.
\]
\end{prop}
\begin{proof}
After applying a congruence, we may assume that $X=\Diag({\bf 1}_r,{\bf 0}_{n-r})$, $Y=\Diag({\bf 0}_r,{\bf 1}_{n-r})$ for some $1\leq r\leq n-1$.  Now we see that $V=\Diag(P,{\bf 0}_{n-r})$ and 
\[
W=\begin{pmatrix}
&Q\\Q^\top\end{pmatrix}
\]
for some nonzero matrices $P\in\CC^{r\times r}$ and $Q\in\CC^{r\times(n-r)}$ such that $P^2=QQ^\top={\bf 0}_r$, $Q^\top Q={\bf 0}_{n-r}$ and $PQ={\bf 0}_{r\times(n-r)}$. This shows that $2\leq r\leq n-2$. Since $n\leq5$, it follows that $P=vv^\top$ and $Q=vw^\top$ for some vectors $v\in\CC^r\setminus\{0\}$ and $w\in\CC^{n-r}\setminus\{0\}$ such that $v^\top v=w^\top w=0$. For fixed $n,r$, we have a single $\Diag(\O(r),\O(n-r))$-orbit. So for every $(n,r)$, we get one embedding up to congruence.
\end{proof}

\subsection{The Jordan algebra $\Es_3$}
An embedding of $\Es_3$ into $\S^n$ is of the form $\CC\{U,X,Y,Z\}$ where $U\in\S^n$ is an invertible matrix and $X,Y,Z\in\S^n$ satisfy 
\[
\begin{array}{cccc}
X\bullet_UX=Y,&X\bullet_UY={\bf 0}_n,&X\bullet_UZ={\bf 0}_n,\\
&Y\bullet_UY={\bf 0}_n,&Y\bullet_UZ={\bf 0}_n,\\
&&Z\bullet_UZ=Y.
\end{array}
\]

\begin{prop}\label{prop:classify_E3}
If $n=4$, then every embedding of $\Es_3$ is congruent to $\Es_3$. If $n=5$, then every embedding of $\Es_3$ is congruent to one of
\[
\begin{pmatrix}y&x&u&z\\x&u\\u\\z&&&u\\&&&&u\end{pmatrix},\begin{pmatrix}y&x&u&&z\\x&u\\u\\&&&x&u\\z&&&u\end{pmatrix}.
\]
\end{prop}
\begin{proof}
Note that $\CC\{U,X,Y\}$ is an embedding of $\CC[x]/(x^3)$. So after applying a congruence, we may assume that
\[
uU+xX+yY\in\left\{\begin{pmatrix}y&x&u\\x&u\\u\\&&&u\end{pmatrix},\begin{pmatrix}y&x&u\\x&u\\u\\&&&u\\&&&&u\end{pmatrix},\begin{pmatrix}y&x&u\\x&u\\u\\&&&x&u\\&&&u\end{pmatrix}\right\}.
\]
When it is the first element, then it is easy to verify that
\[
Z=\begin{pmatrix}a&&&b\\\\\\b\end{pmatrix}
\]
for some $a,b\in\CC$ with $b\neq0$. After scaling and changing the basis, we get $(a,b)=(0,1)$. In the second case, we get
\[
Z=\begin{pmatrix}a&&&b&c\\\\\\b&&&d&e\\c&&&e&f\end{pmatrix}
\]
for some $a,b,c,d,e,f\in\CC$. The condition $Z^{\bullet2}=Y$ gives $b^2+c^2=1$. Acting with $\Diag({\bf1}_3,\O(2))$, we may assume that $(b,c)=(1,0)$. Now $Z^{\bullet2}=Y$ gives $d=e=f=0$. After changing the basis, we also get $a=0$. In the third case, it is easy to check that 
\[
Z=\begin{pmatrix}a&&&&b\\\\\\\\b\end{pmatrix}
\]
for some $a,b\in\CC$ with $b\neq0$. After scaling and changing the basis, we get $(a,b)=(0,1)$. 
\end{proof}

\subsection{The Jordan algebra $\Es_4$}
An embedding of $\Es_4$ into $\S^n$ is of the form $\CC\{U,X,Y,Z\}$ where $U\in\S^n$ is an invertible matrix and $X,Y,Z\in\S^n$ satisfy 
\[
\begin{array}{cccc}
X\bullet_UX=Y,&X\bullet_UY={\bf 0}_n,&X\bullet_UZ={\bf 0}_n,\\
&Y\bullet_UY={\bf 0}_n,&Y\bullet_UZ={\bf 0}_n,\\
&&Z\bullet_UZ={\bf 0}_n.
\end{array}
\]

\begin{prop}\label{prop:classify_E4}
If $n=4$, then $\Es_4$ has no embedding. If $n=5$, then every embedding of $\Es_4$ is congruent to one of
\[
\begin{pmatrix}y&x&u\\x&u\\u\\&&&z&u\\&&&u\end{pmatrix},\begin{pmatrix}y&x&u&z\\x&u\\u\\z&&&&u\\&&&u\end{pmatrix}.
\]
\end{prop}
\begin{proof}
Take $n\in\{4,5\}$. Note that $\CC\{U,X,Y\}$ is an embedding of $\CC[x]/(x^3)$. So after applying a congruence, we may assume that
\[
uU+xX+yY\in\left\{\begin{pmatrix}y&x&u\\x&u\\u\\&&&u\end{pmatrix},\begin{pmatrix}y&x&u\\x&u\\u\\&&&&u\\&&&u\end{pmatrix},\begin{pmatrix}y&x&u\\x&u\\u\\&&&x&u\\&&&u\end{pmatrix}\right\}.
\]
In the first and last cases, it is easy to check that there is no $Z$ linearly independent from $Y$ with the required properties. So we assume we are in the second case. Now we have
\[
Z=\begin{pmatrix}a&&&b&c\\\\\\b&&&d&e\\c&&&e&f\end{pmatrix}
\]
for some $a,b,c,d,e,f\in\CC$. Changing basis gives $a=0$. The equation $Z^{\bullet 2}={\bf 0}_5$ yields
\[
bc = cd + be = ce + bf = de = e^2 + df = ef = 0.
\]
From the last two equations, we see that $e=0$ and so $bc = cd = bf = df = 0$. As $Z$ must be nonzero, we have $(b,d)\neq0$ or $(c,f)\neq0$. Permuting the last two rows/columns, we may assume that $(b,d)\neq(0,0)$. The equations now give $(c,f)=0$. So
\[
Z=\begin{pmatrix}&&&b\\\\\\b&&&d\\&&&&0\end{pmatrix}
\]
Acting with matrices of the form $\Diag(\lambda,1,\lambda^{-1},\mu,\mu^{-1})$, we see that we are free to scale $b,d$ independently. Hence we get $(b,d)\in\{(1,0),(0,1),(1,1)\}$. We have
\[
\begin{pmatrix}
1&&&-1&\\
&1&&&\\
&&1&&\\
&&&1&\\
&&1&&1
\end{pmatrix}
\begin{pmatrix}y+z&x&u&z\\x&u\\u\\z&&&z&u\\&&&u\end{pmatrix}
\begin{pmatrix}
1&&&-1&\\
&1&&&\\
&&1&&\\
&&&1&\\
&&1&&1
\end{pmatrix}^\top=
\begin{pmatrix}y&x&u\\x&u\\u\\&&&z&u\\&&&u\end{pmatrix}
\]
and therefore the Jordan spaces  for $(b,d)=(1,0),(1,1)$ are congruent.
\end{proof}

\section{Degenerations between Jordan spaces}\label{sec:appendix_degenerations}

Fix integers $n,m\geq 1$ and let $(x_1,\ldots,x_n)$ be an $m$-tuple of formal symbols. Then the map 
\[
\Ls\mapsto (x_1,\ldots,x_n)\Ls(x_1,\ldots,x_n)^\top
\]
is a bijection $\Gr(m,\S^n)\to\Gr(m,\CC[x_1,\ldots,x_n]_2)$. Note that $\GL_n$ acts on $\Gr(m,\S^n)$ by congruence, on $\Gr(m,\CC[x_1,\ldots,x_n]_2)$ coordinate transformation and that this bijection is in fact a morphism of $\GL_n$-sets. Below, we will find degenerations in $\Gr(m,\S^n)$ by finding the equivalent degenerations in $\Gr(m,\CC[x_1,\ldots,x_n]_2)$.

\subsection{Families of degenerations between Jordan nets}\label{ssec:degen_nets}

We make the following identifications:
\begin{align*}
&A^{(1)}_{k_1,k_2,k_3}&=\spann&(a_1^2+\ldots+a_{k_1+k_2+k_3}^2,b_1^2+\ldots+b_{k_2+k_3}^2,c_1^2+\ldots+c_{k_3}^2)\\
&A^{(2)}_{r,k_1,k_2}&=\spann&(a_1^2+\ldots+a_r^2,2b_1c_1+\ldots+2b_{k_2}c_{k_2}+d_1^2+\ldots+d_{k_1}^2,b_1^2+\ldots+b_{k_2}^2)\\
&A^{(3)}_{k_1,k_2,k_3}&= \spann&((2a_1c_1+b_1^2)+\ldots+(2a_{k_3}c_{k_3}+b_{k_3}^2)+2d_1e_1+\ldots+2d_{k_2}e_{k_2}+f_1^2+\ldots+f_{k_1}^2,\\
&&&\hspace*{110pt} 2a_1b_1+\ldots+2a_{k_3}b_{k_3}+d_1^2+\ldots+d_{k_2}^2,a_1^2+\ldots+a_{k_3}^2)
\end{align*}

\begin{prop}\label{prop:degen_general_netsA}
We have the following degenerations:
\begin{itemize}
\item[$\mathrm{(a)}$] $A^{(1)}_{k_1,k_2,k_3}\to A^{(2)}_{k_3,k_1,k_2+k_3}$.
\item[$\mathrm{(b)}$] $A^{(1)}_{k_1,k_2,k_3}\to A^{(2)}_{k_2+k_3,k_1+k_2,k_3}$.
\item[$\mathrm{(c)}$] $A^{(1)}_{k_1,k_2,k_3}\to A^{(2)}_{k_1+k_2+k_3,k_2,k_3}$.
\item[$\mathrm{(d)}$] If $k_2>1$, then $A^{(2)}_{r,k_1,k_2}\to A^{(2)}_{r,k_1+2,k_2-1}$.
\item[$\mathrm{(e)}$] If $r\leq k_2$, then $A^{(2)}_{r,k_1,k_2}\to A^{(3)}_{k_1,k_2-r,r}$.
\item[$\mathrm{(f)}$] If $k_2\leq r\leq k_1+k_2$, then $A^{(2)}_{r,k_1,k_2}\to A^{(3)}_{k_1+k_2-r,r-k_2,k_2}$.
\item[$\mathrm{(g)}$] If $r\geq k_1+k_2$, then $A^{(2)}_{r,k_1,k_2}\to A^{(3)}_{r-(k_1+k_2),k_1,k_2}$.
\item[$\mathrm{(h)}$] If $k_3>1$, then $A^{(3)}_{k_1,k_2,k_3}\to A^{(3)}_{k_1+1,k_2+1,k_3-1}$.
\item[$\mathrm{(i)}$] If $k_2>0$, then $A^{(3)}_{k_1,k_2,k_3}\to A^{(3)}_{k_1+2,k_2-1,k_3}$.
\item[$\mathrm{(j)}$] If $k_1>0$ and $k_3>1$, then $A^{(3)}_{k_1,k_2,k_3}\to A^{(3)}_{k_1-1,k_2+2,k_3-1}$.
\end{itemize}
\end{prop}
\begin{proof}
We note that (a)-(d) also follow from the degenerations between orbits of pencils found in \cite{FMS:pencil_old_new}.
(a)
For $t\neq0$, we send 
\[
a_j\mapsto\left\{\begin{array}{cl}b_j+t^2c_j&\mbox{if $j\leq k_2+k_3$}\\td_{j-(k_2+k_3)}&\mbox{if $j>k_2+k_3$}\end{array}\right\}, b_j\mapsto b_j,c_j\mapsto a_j,
\]
we substract the second form from the first, we divide the first form by $t^2$ and let $t\to0$. We get
\[
2b_1c_1+\ldots+2b_{k_2+k_3}c_{k_2+k_3}+d_1^2+\ldots+d_{k_1}^2,~b_1^2+\ldots+b_{k_2+k_3}^2,~a_1^2+\ldots+a_{k_3}^2
\]
which is $A^{(2)}_{k_3,k_1,k_2+k_3}$. 

(b)
For $t\neq0$, we send 
\[
a_j\mapsto\left\{\begin{array}{cl}b_j+t^2c_j&\mbox{if $j\leq k_3$}\\td_{j-k_3}&\mbox{if $j>k_3$}\end{array}\right\}, b_j\mapsto a_j,c_j\mapsto b_j,
\]
we substract the third form from the first, we divide the first form by $t^2$ and let $t\to0$. We get
\[
2b_1c_1+\ldots+2b_{k_3}c_{k_3}+d_1^2+\ldots+d_{k_1+k_2}^2,~a_1^2+\ldots+a_{k_2+k_3}^2,~b_1^2+\ldots+b_{k_3}^2
\]
which is $A^{(2)}_{k_2+k_3,k_1+k_2,k_3}$. 

(c)
For $t\neq0$, we send 
\[
a_j\mapsto a_j,b_j\mapsto\left\{\begin{array}{cl}b_j+t^2c_j&\mbox{if $j\leq k_3$}\\td_{j-k_3}&\mbox{if $j>k_3$}\end{array}\right\}, c_j\mapsto b_j,
\]
we substract the third form from the second, we divide the second form by $t^2$ and let $t\to0$. We get
\[
a_{1}^2+\ldots+a_{k_1+k_2+k_3}^2,~2b_1c_1+\ldots+2b_{k_3}c_{k_3}+d_1^2+\ldots+d_{k_2}^2,~b_1^2+\ldots+b_{k_3}^2
\]
which is $A^{(2)}_{k_1+k_2+k_3,k_2,k_3}$. 

(d)
If $k_2>1$, for $t\neq0$, we send
\[
a_j\mapsto a_j,b_j\mapsto\left\{\begin{array}{cl}b_j&\mbox{if $j<k_2$}\\tx_+&\mbox{if $j=k_2$}\end{array}\right\}, c_j\mapsto\left\{\begin{array}{cl}t^2c_j&\mbox{if $j<k_2$}\\tx_-/2&\mbox{if $j=k_2$}\end{array}\right\},d_j\mapsto td_j
\]
where $x_{\pm}=d_{k_1+1}\pm i d_{k_1+2}$, we divide the second form by $t^2$ and let $t\to0$. We get
\[
a_1^2+\ldots+a_r^2,~2b_1c_1+\ldots+2b_{k_2-1}c_{k_2-1}+d_1^2+\ldots+d_{k_1+2}^2,~b_1^2+\ldots+b_{k_2-1}^2.
\]
which is $A^{(2)}_{r,k_1+2,k_2-1}$.

(e)
If $r\leq k_2$, for $t\neq0$, we send 
\[
a_j\mapsto a_j,b_j\mapsto\left\{\begin{array}{cl}a_j+t^2b_j&\mbox{if $j\leq r$}\\td_{j-r}&\mbox{if $j>r$}\end{array}\right\}, c_j\mapsto\left\{\begin{array}{cl}b_j+t^2c_j&\mbox{if $j\leq r$}\\t^{-1}d_{j-r}/2+te_{j-r}&\mbox{if $j>r$}\end{array}\right\},d_j\mapsto tf_j,
\]
we subtract the first form from the third, we divide the third form by $t^2$, we subtract the third form from the second, we divide the second form by $t^2$ and we let $t\to0$. We get
\[
a_1^2+\ldots+a_r^2,~(2a_1c_1+b_1^2)+\ldots+(2a_rc_r+b_r^2)+2d_1e_1+\ldots+2d_{k_2-r}e_{k_2-r}+f_1^2+\ldots+f_{k_1}^2,
\]
\[
2a_1b_1+\ldots+2a_rb_r+d_1^2+\ldots+d_{k_2-r}^2
\]
which is $A^{(3)}_{k_1,k_2-r,r}$.

(f)
If $k_2\leq r\leq k_1+k_2$, for $t\neq0$, we send 
\[
a_j\mapsto\left\{\begin{array}{cl}a_j-t^2b_j&\mbox{if $j\leq k_2$}\\itd_{j-k_2}&\mbox{if $j>k_2$}\end{array}\right\},b_j\mapsto a_j, c_j\mapsto b_j+t^2c_j,d_j\mapsto\left\{\begin{array}{cl}d_j+t^2e_j&\mbox{if $j\leq r-k_2$}\\tf_{j-(r-k_2)}&\mbox{if $j>r-k_2$}\end{array}\right\},
\]
we subtract the third form from the first, we divide the first form by $t^2$, we add the first form to the second, we divide the second form by $t^2$ and we let $t\to0$. We get
and get
\[
2a_1b_1+\ldots+2a_{k_2}b_{k_2}+d_1^2+\ldots+d_{r-k_2}^2,
\]
\[
(2a_1c_1+b_1^2)+\ldots+(2a_{k_2}c_{k_2}+b_{k_2}^2)+2d_1e_1+\ldots+2d_{r-k_2}e_{r-k_2}+f_1^2+\ldots +d_{k_1+k_2-r}^2,~a_1^2+\ldots+a_{k_2}^2
\]
which is $A^{(3)}_{k_1+k_2-r,r-k_2,k_2}$.

(g)
If $r\geq k_1+k_2$, for $t\neq0$, we send 
\[
a_j\mapsto\left\{\begin{array}{cl}a_j+t^2b_j+t^4c_j&\mbox{if $j\leq k_2$}\\td_{j-k_2}+t^3e_{j-k_2}&\mbox{if $k_2<j\leq k_1+k_2$}\\t^2f_{j-(k_1+k_2)}&\mbox{if $j>k_1+k_2$}\end{array}\right\},b_j\mapsto a_j, c_j\mapsto b_j,d_j\mapsto d_j,
\]
we subtract the third form from the first, we divide the first form by $t^2$, we subtract the second form from the first, we divide the first form by $t^2$ again and we let $t\to 0$. We get
\[
(2a_1c_1+b_1^2)+\ldots+(2a_{k_2}c_{k_2}+b_{k_2}^2)+2d_1e_1+\ldots+2d_{k_1}e_{k_1}+f_1^2+\ldots+f_{r-(k_1+k_2)}^2,
\]
\[
2a_1b_1+\ldots+2a_{k_2}b_{k_2}+d_1^2+\ldots+d_{k_1}^2,~a_1^2+\ldots+a_{k_2}^2
\]
which is $A^{(3)}_{r-(k_1+k_2),k_1,k_2}$.

(h)
If $k_3>1$, for $t\neq0$, we send
\[
a_{k_3}\mapsto td_{k_2+1},b_{k_3}\mapsto t^{-1}d_{k_2+1}/2+f_{k_1+1}, c_{k_3}\mapsto -t^{-3}d_{k_2+1}/8-t^{-2}f_{k_1+1}/2+t^{-1}e_{k_2+1}
\]
 and we let $t\to0$. We get 
\[
(2a_1c_1+b_1^2)+\ldots+(2a_{k_3-1}c_{k_3-1}+b_{k_3-1}^2)+(2d_{k_2+1}e_{k_2+1}+f_{k_1+1}^2)+2d_1e_1+\ldots+2d_{k_2}e_{k_2}+f_1^2+\ldots+f_{k_1}^2,
\]
\[
2a_1b_1+\ldots+2a_{k_3-1}b_{k_3-1}+d_{k_2+1}^2+d_1^2+\ldots+d_{k_2}^2,~a_1^2+\ldots+a_{k_3-1}^2.
\]
which is $A^{(3)}_{k_1+1,k_2+1,k_3-1}$.

(i)
If $k_2>0$, for $t\neq0$, we send
\[
d_{k_2}\mapsto tx_+,e_{k_2}\mapsto t^{-1}x_-/2
\]
where $x_{\pm}=f_{k_1+1}\pm if_{k_1+2}$ and we let $t\to0$. We get
\[
(2a_1c_1+b_1^2)+\ldots+(2a_{k_3}c_{k_3}+b_{k_3}^2)+2d_1e_1+\ldots+2d_{k_2-1}e_{k_2-1}+(f_{k_1+1}^2+f_{k_1+2}^2)+f_1^2+\ldots+f_{k_1}^2,
\]
\[
2a_1b_1+\ldots+2a_{k_3}b_{k_3}+d_1^2+\ldots+d_{k_2-1}^2,~a_1^2+\ldots+a_{k_3}^2.
\]
which is $A^{(3)}_{k_1+2,k_2-1,k_3}$.

(j)
If $k_1>0$ and $k_3>1$, for $t\neq0$, we send
\[
a_{k_3}\mapsto tx_+,b_{k_3}\mapsto t^{-1}x_-/2,c_{k_3}\mapsto t^{-1}y_-/2, f_{k_1}\mapsto i(t^{-1}x_-/2-ty_+)
\]
where $x_{\pm}=d_{k_2+1}\pm id_{k_2+2}, y_{\pm}=e_{k_2+1}\pm ie_{k_2+2}$ and we let $t\to0$. We get
\[
(2a_1c_1+b_1^2)+\ldots+(2a_{k_3-1}c_{k_3-1}+b_{k_3-1}^2)+2d_1e_1+\ldots+2d_{k_2+2}e_{k_2+2}+f_1^2+\ldots+f_{k_1-1}^2,
\]
\[
2a_1b_1+\ldots+2a_{k_3-1}b_{k_3-1}+d_1^2+\ldots+d_{k_2+2}^2,~a_1^2+\ldots+a_{k_3-1}^2
\]
which is $A^{(3)}_{k_1-1,k_2+2,k_3-1}$.
\end{proof}

Next, note that $B^{(1)}_{n/2}$ is congruent to
\[
\begin{pmatrix}xJ_{n/2}&z{\bf 1}_{n/2}\\z{\bf 1}_{n/2}&yJ_{n/2}\end{pmatrix}.
\]
We make the following identifications:
\begin{align*}
&B^{(1)}_{n/2}&=\spann&(a_1a_{n/2}+\ldots+a_{n/2}a_1,b_1b_{n/2}+\ldots+b_{n/2}b_1,a_1b_1+\ldots+a_{n/2}b_{n/2})\\
&B^{(2)}_{k,\ell_1,\ell_2}&=\spann&(a_1a_{\ell_2}+\ldots+a_{\ell_2}a_1,b_1b_{\ell_2}+\ldots+b_{\ell_2}b_1+c_1^2+\ldots+c_{\ell_1}^2,a_1b_1+\ldots+a_kb_k)
\end{align*}

\begin{prop}\label{prop:degen_general_netsB}
We have the following degenerations:
\begin{itemize}
\item[$\mathrm{(a)}$] If $2\mid n$ and for $1\leq k\leq n/4$, then $B^{(1)}_{n/2}\to B^{(2)}_{k,0,n/2}$.
\item[$\mathrm{(b)}$] If $k>1$, then $B^{(2)}_{k,\ell_1,\ell_2}\to B^{(2)}_{k-1,\ell_1,\ell_2}$.
\end{itemize}
\end{prop}
\begin{proof}
(a)
If $2\mid n$ and $1\leq k\leq n/4$, for $t\neq0$, we send
\[
a_i\mapsto\left\{\begin{array}{cl}a_i&\mbox{if $i\leq k$}\\ta_i&\mbox{if $k<i\leq n/2-k$}\\t^2a_i&\mbox{if $i>n/2-k$}\end{array}\right\},b_i\mapsto b_i,
\]
divide the first form by $t^2$ and let $t\to0$. We get
\[
a_1a_{n/2}+\ldots+a_{n/2}a_1,b_1b_{n/2}+\ldots+b_{n/2}b_1,a_1b_1+\ldots+a_kb_k
\]
which is $B^{(2)}_{k,0,n/2}$.

(b)
If $k>1$, for $t\neq0$, we send
\[
a_i\mapsto\left\{\begin{array}{cl}a_i&\mbox{if $i<k$}\\ta_i&\mbox{if $k\leq i\leq \ell_2-k+1$}\\t^2a_i&\mbox{if $i>\ell_2-k+1$}\end{array}\right\},b_i\mapsto b_i,c_i\mapsto c_i,
\]
divide the first form by $t^2$ and let $t\to0$. We get
\[
a_1a_{\ell_2}+\ldots+a_{\ell_2}a_1,~b_1b_{\ell_2}+\ldots+b_{\ell_2}b_1+c_1^2+\ldots+c_{\ell_1}^2,~a_1b_1+\ldots+a_{k-1}b_{k-1}
\]
which is $B^{(2)}_{k-1,\ell_1,\ell_2}$.
\end{proof}

\subsection{Jordan nets in $\S^5$}

We identify a Jordan net $\Ls\subseteq\S^5$ with the associated net of quadrics $(a,b,c,d,e)\Ls(a,b,c,d,e)^\top$. This gives the following list:

\begin{tabular}{lllll}
$A^{(1)}_{2,0,1}$&\!\!\!\!= $\spann(a^2+b^2+c^2,d^2,e^2)$,&&$A^{(3)}_{2,0,1}$&\!\!\!\!= $\spann(2ac+b^2+d^2+e^2,ab,a^2)$,\\
$A^{(1)}_{0,1,1}$&\!\!\!\!= $\spann(a^2+b^2,c^2+d^2,e^2)$,&&$A^{(3)}_{0,1,1}$&\!\!\!\!= $\spann(2ac+b^2+2de,2ab+d^2,a^2)$,\\
$A^{(2)}_{1,2,1}$&\!\!\!\!= $\spann(a^2,2bc+d^2+e^2,b^2)$,&&$B^{(2)}_{1,1,2}$&\!\!\!\!= $\spann(ab,2cd+e^2,ac)$,\\
$A^{(2)}_{1,0,2}$&\!\!\!\!= $\spann(a^2,be+cd,b^2+c^2)$,&&$C_{5,1}\phantom{^{(1)}}$&\!\!\!\!= $\spann(2ae+2bd+c^2,a^2,b^2)$,\\
$A^{(2)}_{2,1,1}$&\!\!\!\!= $\spann(a^2+b^2,2cd+e^2,c^2)$,&&$C_{5,2}\phantom{^{(1)}}$&\!\!\!\!= $\spann(2ae+2bd+c^2,ab,a^2)$,\\
$A^{(2)}_{3,0,1}$&\!\!\!\!= $\spann(a^2+b^2+c^2,de,d^2)$
\end{tabular}

\begin{prop}\label{prop:degen_nets_into_S5}
We have the following degenerations:\vspace{-10pt}
\end{prop}
\begin{multicols}{2}
\begin{enumerate}
\item[$\mathrm{(a)}$] $A^{(2)}_{1,2,1}\to C_{5,1}$.
\item[$\mathrm{(b)}$] $A^{(3)}_{2,0,1}\to C_{5,2}$.
\item[$\mathrm{(c)}$] $A^{(3)}_{0,1,1}\to C_{5,1}$.
\end{enumerate}
\columnbreak
\begin{enumerate}
\item[$\mathrm{(d)}$] $B^{(2)}_{1,1,2}\to C_{5,2}$.
\item[$\mathrm{(e)}$] $C_{5,1}\to C_{5,2}$.
\end{enumerate}
\end{multicols}
\begin{proof}
(a)
For $t\neq0$, we apply the coordinate transformation 
\[
(a,b,c,d,e)\to(a,b,d,c,t^{-1}a+te)
\]
to go from $A^{(2)}_{1,2,1}$ to
\[
\spann(a^2,2bd+c^2+(t^{-1}a+te)^2,b^2)=\spann(2ae+2bd+c^2+t^2e^2,a^2,b^2)
\]
and hence get $\spann(2ae+2bd+c^2,a^2,b^2)=C_{5,1}$ when $t\to0$.

(b)
For $t\neq0$, we apply the coordinate transformation 
\[
(a,b,c,d,e)\to(a,b,t^2e,tc,i(b-t^2d))
\]
to go from $A^{(3)}_{2,0,1}$ to
\[
\spann(2t^2ae+b^2+t^2c^2-(b-t^2d)^2,ab,a^2)=\spann(2ae+c^2+2bd-t^2d^2,ab,a^2)
\]
and hence get $\spann(2ae+c^2+2bd,ab,a^2)=C_{5,2}$ when $t\to0$.

(c)
For $t\neq0$, we apply the coordinate transformation 
\[
(a,b,c,d,e)\to(a,tc,t^2e,b,t^2d)
\]
to go from $A^{(3)}_{0,1,1}$ to
\[
\spann(2t^2ae+t^2c^2+2t^2bd,2tac+b^2,a^2)=\spann(2ae+c^2+2bd,b^2+2tac,a^2)
\]
and hence get $\spann(2ae+c^2+2bd,b^2,a^2)=C_{5,1}$ when $t\to0$.

(d)
For $t\neq0$, we apply the coordinate transformation 
\[
(a,b,c,d,e)\to(a-t^2d,b-t^2e,a,b,tc)
\]
to go from $B^{(2)}_{1,1,2}$ to
\[
\spann((a-t^2d)(b-t^2e),2ab+t^2c^2,(a-t^2d)a)=\spann(2bd+2ae+c^2-2t^2de,ab+t^2c^2/2,a^2-t^2ad)
\]
and hence get $\spann(2bd+2ae+c^2,ab,a^2)=C_{5,2}$ when $t\to0$.

(e)
For $t\neq0$, we apply the coordinate transformation 
\[
(a,b,c,d,e)\to(a,a+t^2b,tc,d,-d+t^2e)
\]
to go from $C_{5,1}$ to
\[
\spann(2a(-d+t^2e)+2(a+t^2b)d+t^2c^2,a^2,(a+t^2b)^2)=\spann(2ae+2bd+c^2,a^2,ab+t^2b^2/2)
\]
and hence get $\spann(2bd+2ae+c^2,ab,a^2)=C_{5,2}$ when $t\to0$.
\end{proof}

\subsection{Jordan nets in $\S^6$}
We identify a Jordan net $\Ls\subseteq\S^6$ with the associated net of quadrics $(a,b,c,d,e,f)\Ls(a,b,c,d,e,f)^\top$. This gives the following list:

\begin{tabular}{lllll}
$A^{(1)}_{3,0,1}$&\!\!\!\!= $\spann(a^2+b^2+c^2+d^2,e^2,f^2)$,&&$A^{(3)}_{3,0,1}$&\!\!\!\!= $\spann(2ac+b^2+d^2+e^2+f^2,ab,a^2)$,\\
$A^{(1)}_{1,1,1}$&\!\!\!\!= $\spann(a^2+b^2+c^2,d^2+e^2,f^2)$,&&$A^{(3)}_{1,1,1}$&\!\!\!\!= $\spann(2ac+b^2+2de+f^2,2ab+d^2,a^2)$,\\
$A^{(1)}_{0,0,2}$&\!\!\!\!= $\spann(a^2+b^2,c^2+d^2,e^2+f^2)$,&&$A^{(3)}_{0,0,2}$&\!\!\!\!= $\spann(2ac+b^2+2df+e^2,ab+de,a^2+d^2)$,\\
$A^{(2)}_{1,3,1}$&\!\!\!\!= $\spann(a^2,2bc+d^2+e^2+f^2,b^2)$,&&$C_{6,1}$&\!\!\!\!= $\spann(af+be+cd,a^2+c^2,b^2+c^2)$,\\
$A^{(2)}_{1,1,2}$&\!\!\!\!= $\spann(a^2,2bc+2de+f^2,b^2+d^2)$,&&$C_{6,2}$&\!\!\!\!= $\spann(af+be+cd,a^2+c^2,ab)$,\\
$A^{(2)}_{2,2,1}$&\!\!\!\!= $\spann(a^2+b^2,2cd+e^2+f^2,c^2)$,&&$C_{6,3}$&\!\!\!\!= $\spann(af+be+cd,2ac+b^2,ab)$,\\
$A^{(2)}_{2,0,2}$&\!\!\!\!= $\spann(a^2+b^2,cd+ef,c^2+e^2)$,&&$C_{6,4}$&\!\!\!\!= $\spann(af+be+cd,a^2+b^2,c^2)$,\\
$A^{(2)}_{3,1,1}$&\!\!\!\!= $\spann(a^2+b^2+c^2,2de+f^2,d^2)$,&&$C_{6,5}$&\!\!\!\!= $\spann(af+be+cd,2ab+c^2,a^2)$,\\
$A^{(2)}_{4,0,1}$&\!\!\!\!= $\spann(a^2+b^2+c^2+d^2,ef,e^2)$,&&$C_{6,6}$&\!\!\!\!= $\spann(af+be+cd,ab,ac)$,\\
$B^{(2)}_{1,2,2}$&\!\!\!\!= $\spann(ab,2cd+e^2+f^2,ac)$,&&$C_{6,7}$&\!\!\!\!= $\spann(af+be+cd,a^2,b^2)$,\\
$B^{(2)}_{1,0,3}$&\!\!\!\!= $\spann(2ac+b^2,2df+e^2,ad)$,&&$C_{6,8}$&\!\!\!\!= $\spann(af+be+cd,ab,a^2)$,\\
$B^{(1)}_3$&\!\!\!\!= $\spann(a^2+b^2+c^2,d^2+e^2+f^2,ad+be+cf)$\hspace*{-100pt}
\end{tabular}

\begin{prop}\label{prop:degen_nets_into_S6}
We have the following degenerations:\vspace{-10pt}
\end{prop}
\begin{multicols}{2}
\begin{enumerate}[(a)]
\item $A^{(1)}_{0,0,2}\to C_{6,1}$
\item $A^{(2)}_{1,3,1}\to C_{6,7}$
\item $A^{(2)}_{1,1,2}\to C_{6,4}$
\item $A^{(2)}_{2,2,1}\to C_{6,4}$
\item $A^{(2)}_{2,0,2}\to C_{6,2}$
\item $B^{(2)}_{1,2,2}\to C_{6,6}$
\item $B^{(2)}_{1,0,3}\to C_{6,4}$
\item $A^{(3)}_{3,0,1}\to C_{6,8}$
\item $A^{(3)}_{1,1,1}\to C_{6,5}$
\item $A^{(3)}_{0,0,2}\to C_{6,3}$
\item $C_{6,1}\to C_{6,2}$ 
\item $C_{6,2}\to C_{6,3}$
\item $C_{6,2}\to C_{6,4}$
\item $C_{6,3}\to C_{6,5}$
\item $C_{6,3}\to C_{6,6}$
\item $C_{6,4}\to C_{6,5}$
\item $C_{6,5}\to C_{6,7}$
\item $C_{6,6}\to C_{6,8}$
\item $C_{6,7}\to C_{6,8}$
\end{enumerate}
\end{multicols}
\begin{proof}
The proof follows the same structure as the proof of Proposition~\ref{prop:degen_nets_into_S5}. With $t\neq0$, we first apply a coordinate transformation:

\begin{enumerate}[(a)]
\item $(a,b,c,d,e,f)\mapsto(c,b,a,c+td,b+te,i(a-tf))$
\item $(a,b,c,d,e,f)\mapsto(a,b,e,t^{-1}a+tf,it^{-1}c,t^{-1}c+td)$
\item $(a,b,c,d,e,f)\mapsto(c,b,e,a,f,t^{-1}c+td)$
\item $(a,b,c,d,e,f)\mapsto(a,b,c,d,t^{-1}b+te,t^{-1}a+tf)$
\item $(a,b,c,d,e,f)\mapsto(a+b+t(e+f),i(a-b)+ti(e-f),a,b,c,-td)$
\item $(a,b,c,d,e,f)\mapsto(a,b,c,td,a+b+t(f+e)/2,i(a-b)-ti(f-e)/2)$
\item $(a,b,c,d,e,f)\mapsto(a+ib,c+td,t(f-ie)/2,a-ib,c,-t(f+ie)/2)$
\item $(a,b,c,d,e,f)\mapsto(a,b,t^2f,i(b-t^2e),t(d+c)/\sqrt{2},ti(d-c)/\sqrt{2})$
\item $(a,b,c,d,e,f)\mapsto(a,b,tf,c,td,i(b-te))$
\item $(a,b,c,d,e,f)\mapsto(b,-2ic+tb+t^2id/2,ti(2c-f)+t^2(e-b/2),ib+ta,2c,tf)$
\item $(a,b,c,d,e,f)\mapsto(a+tb,a,c,td,-e+tf,e)$
\item $(a,b,c,d,e,f)\mapsto(a,b,ia+tib+t^2ic,-id,-td+t^2e,-d+t^2f)$
\item $(a,b,c,d,e,f)\mapsto(t(a-ib),a+ib,c,t^2d,t^2(f-ie)/2,t(f+ie)/2)$
\item $(a,b,c,d,e,f)\mapsto(a,a+tc,-tc/2+t^2b,e,e/2+td,-e/2-td+t^2f)$
\item $(a,b,c,d,e,f)\mapsto(a,tb,c,td,e,tf)$
\item $(a,b,c,d,e,f)\mapsto(a+t^2b,tc,a,-e+t^2f,td,e)$
\item $(a,b,c,d,e,f)\mapsto(ta,c,b,te,td,f)$
\item $(a,b,c,d,e,f)\mapsto(a,a+tc,b,te,d,-d+tf)$
\item $(a,b,c,d,e,f)\mapsto(a,a+tb,c,td,e,-e+tf)$
\end{enumerate}
Afterwards we apply a basechange over $\CC[t,t^{-1}]$ and let $t\to0$ to obtain the degeneration.
\end{proof}

\subsection{Jordan webs in $\S^4$}
We identify a Jordan web $\Ls\subseteq\S^4$ with the associated web of quadrics $(a,b,c,d,)\Ls(a,b,c,d)^\top$. This gives the following list:

\begin{tabular}{lllll}
$A^{(1)}_{0,0,0,1}$&\!\!\!\!= $\spann(a^2,b^2,c^2,d^2)$,&&$C^{(1)}_1$&\!\!\!\!= $\spann(a^2+b^2,c^2+d^2,ac+bd,ad-bc)$,\\
$A^{(2)}_{0,1,0,1}$&\!\!\!\!= $\spann(a^2,b^2,cd,c^2)$,&&$E^{(1)}_{4,1}$&\!\!\!\!= $\spann(2ac+d^2,b^2,ab,a^2)$,\\
$A^{(3)}_{0,1,0,1}$&\!\!\!\!= $\spann(ab,a^2,cd,c^2)$,&&$E^{(1)}_{4,2}$&\!\!\!\!= $\spann(ac,b^2+d^2,ab,a^2)$,\\
$A^{(4)}_{1,0,0,1}$&\!\!\!\!= $\spann(a^2,2bd+c^2,bc,b^2)$,&&$E^{(2)}_4$&\!\!\!\!= $\spann(ab,cd,ac,a^2)$,\\
$A^{(5)}_{0,0,0,1}$&\!\!\!\!= $\spann(ad+bc,2ac+b^2,ab,a^2)$,&&$E^{(3)}_4$&\!\!\!\!= $\spann(2ac+b^2+d^2,ad,ab,a^2)$,\\
$B^{(1)}_{2,1}$&\!\!\!\!= $\spann(a^2+b^2,c^2,d^2,cd)$,&&$F_4$&\!\!\!\!= $\spann(ad+bc,a^2,b^2,ab)$
\end{tabular}

\begin{prop}\label{prop:degen_webs_into_S4}
We have the following degenerations:\vspace{-10pt}
\end{prop}
\begin{multicols}{2}
\begin{enumerate}[(a)]
\item $A^{(1)}_{0,0,0,1}\to A^{(2)}_{0,1,0,1}$
\item $A^{(2)}_{0,1,0,1}\to A^{(3)}_{0,1,0,1}$
\item $A^{(2)}_{0,1,0,1}\to A^{(4)}_{1,0,0,1}$
\item $A^{(3)}_{0,1,0,1}\to A^{(5)}_{0,0,0,1}$
\item $A^{(4)}_{1,0,0,1}\to A^{(5)}_{0,0,0,1}$
\item $A^{(5)}_{0,0,0,1}\to E^{(3)}_4$
\item $A^{(5)}_{0,0,0,1}\to F_4$
\item $B^{(1)}_{2,1}\to E^{(1)}_{4,1}$
\item $C^{(1)}_1\to F_4$
\item $E^{(1)}_{4,1}\to F_4$
\item $E^{(1)}_{4,2}\to E^{(2)}_4$
\item $E^{(2)}_4\to E^{(3)}_4$
\end{enumerate}
\end{multicols}
\begin{proof}
The proof follows the same structure as the proof of Proposition~\ref{prop:degen_nets_into_S5}. With $t\neq0$, we first apply a coordinate transformation:

\begin{enumerate}[(a)]
\item $(a,b,c,d)\mapsto(a,b,c,c+td)$
\item $(a,b,c,d)\mapsto(a,a+tb,c,d)$
\item $(a,b,c,d)\mapsto(a,b+tc+t^2d,b,c)$
\item $(a,b,c,d)\mapsto(a,b,a+tb+t^2c,b+2tc+t^2d)$
\item $(a,b,c,d)\mapsto(a+tb+t^2c+t^3d,a,b,c)$
\item $(a,b,c,d)\mapsto(a,t(b+id),b-id,2tc)$
\item $(a,b,c,d)\mapsto(a,b,tc,td)$
\item $(a,b,c,d)\mapsto(a+t^2c,td,a,b)$
\item $(a,b,c,d)\mapsto(a,-tc,b,td)$
\item $(a,b,c,d)\mapsto(a,b,td,b+tc)$
\item $(a,b,c,d)\mapsto(a,ic,b,c+td)$
\item $(a,b,c,d)\mapsto(a,b+id,a+t(b-id),(b+id)+2tc)$
\end{enumerate}
Afterwards we apply a basechange over $\CC[t,t^{-1}]$ and let $t\to0$ to obtain the degeneration.
\end{proof}

\end{document}